\documentclass{memo-l}
\usepackage{amsthm, amsmath, amssymb, enumerate, paralist, calc, verbatim, stmaryrd, ifthen, comment, accents}
\usepackage{tikz,url}
\usepackage[T1]{fontenc}
\usepackage[english]{babel}
\usepackage{hyperref}
\hypersetup{bookmarksopen=true}
\usepackage{mathtools}

\usepackage{bm}

\newcommand{\diagnode}[1]{\fill #1 circle (.1);}
\newcommand{\distorbit}[1]{\draw #1 circle (.2);}
\newcommand{\distlongorbit}[1]{\draw (#1 - .2, .5) -- (#1 - .2, -.5) arc (-180:0:.2) -- (#1 + .2, -.5) -- (#1 + .2, .501) arc (0:180:.2) -- cycle;}

\newtheorem{theorem}{Theorem}[section]

\newtheorem{lemma}[theorem]{Lemma}
\newtheorem{proposition}[theorem]{Proposition}
\newtheorem{corollary}[theorem]{Corollary}

\newtheorem{step}{Step}

{\theoremstyle{definition}
\newtheorem{remark}[theorem]{Remark}
\newtheorem*{remark*}{Remark}
\newtheorem{definition}[theorem]{Definition}
\newtheorem{notation}[theorem]{Notation}
\newtheorem{construction}[theorem]{Construction}
\newtheorem{example}[theorem]{Example}

\newtheorem*{construction*}{Construction}
}

\numberwithin{section}{chapter}
\numberwithin{equation}{chapter}

\newcommand{\GG}{\mathbf{G}}   
\newcommand{\GH}{\mathbf{H}}   
\newcommand{\GX}{\mathbf{X}}   
\newcommand{\G}{\mathfrak{g}}   
\newcommand{\GS}{\mathbf{S}}   
\newcommand{\GT}{\mathbf{T}}   
\newcommand{\GP}{\mathbf{P}}   
\newcommand{\radu}{\mathcal{R}_u}   

\newcommand{\PSL}{\mathsf{PSL}}
\newcommand{\PSU}{\mathsf{PSU}}
\newcommand{\SO}{\mathsf{SO}}
\newcommand{\SU}{\mathsf{SU}}
\newcommand{\GL}{\mathsf{GL}}

\newcommand{\mouf}{\mathbb{M}} 
\newcommand{\Hh}{\mathbb{H}}
\newcommand{\A}{\mathcal{A}} 
\newcommand{\B}{\mathcal{B}} 
\newcommand{\Herm}{\mathcal{H}} 
\newcommand{\Ss}{\mathcal{S}}

\newcommand{\LL}{\mathcal{L}}
\newcommand{\U}{\mathcal{U}}

\newcommand{\bt}{\bullet}

\newcommand{\al}{\alpha}
\newcommand{\be}{\beta}
\newcommand{\barop}{\overline{\phantom{x}}}

\newcommand\half{{\tfrac{1}{2}}}
\newcommand{\subhalf}{{\leavevmode \raise.5ex\hbox{\the\scriptfont0 1}\kern-.13em /\kern-.07em\lower.25ex\hbox{\the\scriptfont0 2}}}

\newcommand{\ga}{\gamma}

\newcommand{\si}{\sigma}
\newcommand{\eps}{\epsilon}

\newcommand{\is}{^{\langle \hat{x}\rangle}}
\newcommand{\hx}{{\hat{x}}}

\newcommand{\dash}{\nobreakdash-\hspace{0pt}}

\newcommand{\id}{\mathrm{id}}

\newlength{\dhatheight}
\newcommand{\doublehat}[1]{%
    \settoheight{\dhatheight}{\ensuremath{\hat{#1}}}%
    \addtolength{\dhatheight}{-0.27ex}%
    \hat{\vphantom{\rule{1pt}{\dhatheight}}%
    \smash{\hat{#1}}}}

\DeclareMathOperator{\Aut}{Aut}
\DeclareMathOperator{\CD}{CD}
\DeclareMathOperator{\Char}{char}
\DeclareMathOperator{\Gal}{Gal}
\DeclareMathOperator{\End}{End}
\DeclareMathOperator{\Hom}{Hom}

\DeclareMathOperator{\Fix}{Fix}

\DeclareMathOperator{\Rad}{Rad}

\DeclareMathOperator{\Spec}{Spec}
\DeclareMathOperator{\Span}{Span}
\DeclareMathOperator{\N}{N}

\DeclareMathOperator{\cha}{char}

\DeclareMathOperator{\Stab}{Stab}
\DeclareMathOperator{\Strl}{Strl}
\DeclareMathOperator{\Instrl}{Instrl}
\DeclareMathOperator{\ad}{ad}

\DeclareMathOperator{\Sym}{Sym}
\DeclareMathOperator{\Alt}{Alt}
\DeclareMathOperator{\Cyc}{Cyc}
\DeclareMathOperator{\Dih}{Dih}

\DeclareMathOperator{\Lie}{Lie}
\DeclareMathOperator{\Dist}{Dist}
\newcommand{\ZZ}{\mathbb Z}
\newcommand{\Gm}{{\mathbb G}_\mathrm{m}}
\DeclareMathOperator{\Cent}{Cent}
\DeclareMathOperator{\Norm}{Norm}
\DeclareMathOperator{\Der}{Der}

\newcommand\newl{\widetilde}
\newcommand{\GU}{\mathbf{U}}   
\newcommand{\GQ}{\mathbf{Q}}

\newcommand{\GLe}{\mathbf{L}}
\newcommand{\GZ}{\mathbf{Z}}
\DeclareMathOperator{\Ad}{Ad}
\newcommand{\PGL}{\mathsf{PGL}}
\newcommand{\PGU}{\mathsf{PGU}}
\newcommand{\SL}{\mathsf{SL}}
\newcommand{\FF}{\mathbb F}
\newcommand{\lK}{\bm{\{}}
\newcommand{\rK}{\bm{\}}}
\newcommand{\VK}{\mathbf{V}}
\newcommand{\KK}{\mathbf{K}}
\newcommand{\TT}{\mathcal{T}}
\newcommand{\DD}{\mathcal{D}}
\newcommand{\CC}{\mathcal{C}}

\newcounter{M}
\newcommand{\Mitem}[1][]{%
        \ifthenelse{\equal{#1}{}}
		{\item[(${\bf M}_{\refstepcounter{M}\theM}$)]}
		{\item[(${\bf M}_{\refstepcounter{M}\label{qitem:#1}\theM}$)]}
}
\newcommand{\Mref}[1]{(${\bf M}_{\ref{qitem:#1}}$)}

\allowdisplaybreaks

\begin{document}

\frontmatter

\title{Moufang sets and structurable division algebras}

\author{Lien Boelaert}
\address{Ghent University, Department of Mathematics, Krijgslaan 281-S22, 9000 Gent, Belgium}
\email{lboelaer@cage.UGent.be}

\author{Tom De Medts}
\address{Ghent University, Department of Mathematics, Krijgslaan 281-S22, 9000 Gent, Belgium}
\email{Tom.DeMedts@UGent.be}

\author{Anastasia Stavrova}
\address{Chebyshev Laboratory, St. Petersburg State University,
14th Line V.O. 29B, 197342 Saint Petersburg, Russia}
\email{anastasia.stavrova@gmail.com}

\date{\today}



\maketitle

\tableofcontents

\begin{abstract}
    A {\em Moufang set} is essentially a doubly transitive permutation group
    such that each point stabilizer contains a normal subgroup which is regular on the remaining vertices;
    these regular normal subgroups are called the root groups, and they are assumed to be conjugate and to generate the whole group.

    It has been known for some time that every Jordan division algebra gives rise to a Moufang set with abelian root groups.
    We extend this result by showing that every structurable division algebra gives rise to a Moufang set,
    and conversely, we show that every Moufang set arising from a simple linear algebraic group of relative rank one
    over an arbitrary field $k$ of characteristic different from $2$ and $3$ arises from a structurable division algebra.

    We also obtain explicit formulas for the root groups, the $\tau$-map and the Hua maps of these Moufang sets.
    This is particularly useful for the Moufang sets arising from exceptional linear algebraic groups.
\end{abstract}

%
%
\chapter*{Introduction}\label{se:intro}

A {\em Moufang set} is essentially a doubly transitive permutation group
such that each point stabilizer contains a normal subgroup which is regular on the remaining vertices;
these regular normal subgroups are called the root groups, and they are assumed to be conjugate and to generate the whole group.
A Moufang set is called {\em proper} if it is not sharply doubly transitive.

Moufang sets were introduced by Jacques Tits in 1990 \cite{Durham} in the context of twin buildings,
and they provide an important tool to study absolutely simple linear algebraic groups of relative rank one,
in the same spirit as spherical buildings play a crucial role in the understanding of isotropic algebraic groups of higher relative rank.
There is, in fact, an ongoing attempt to understand the classification of absolutely simple algebraic groups over an arbitrary field
purely in terms of the theory of spherical buildings, using a theory of Galois descent in buildings \cite{descent-book,descent-2}.

In \cite{DW}, it is shown that every Jordan division algebra gives rise to a Moufang set with abelian root groups,
and in fact, {\em every known} proper Moufang set with abelian root groups arises from a Jordan division algebra.
The description of the known proper Moufang sets with non-abelian root groups, on the other hand, was done case by case.

In the current paper, we show that every {\em structurable} division algebra gives rise to a Moufang set.
It turns out that {\em every known} proper Moufang set arises from a structurable division algebra,
provided that the root groups do not contain elements of order $2$ or $3$.
(The known examples where this condition does not hold all arise by twisting certain algebraic structures over fields of characteristic $2$ or $3$,
and include the Suzuki groups and the rank one Ree groups over fields admitting a Tits endomorphism.)

Structurable algebras have been introduced by Bruce Allison in \cite{A78} in order to construct isotropic simple Lie algebras
using a generalization of the Tits--Kantor--Koecher construction
(see also section~\ref{se:constr lie alg} below).
Because of the connections between simple linear algebraic groups and simple Lie algebras, it is not surprising that there is also
a deep connection between structurable algebras and linear algebraic groups.
In case of Jordan algebras, this connection has been thoroughly investigated and
uniformized by Tonny Springer~\cite{Springer} and Ottmar Loos~\cite{Loos-homog,Loos-pairs} (in a more
general setting of Jordan pairs). Namely, one obtains a classification-free natural
correspondence between semisimple algebraic groups having a parabolic subgroup with abelian unipotent radical,
and separable Jordan pairs. The respective correspondence for arbitrary structurable algebras
is not yet available in the literature in full generality, although a lot of results
have been established on a case-by-case basis, most notably for simple structurable algebras of
skew-dimension one~\cite{Brown,Gar-stralg,Krut}.

In the current paper, we focus on the Moufang sets corresponding to absolutely simple linear algebraic groups of relative rank one.
More precisely, we show that, over fields $k$ with $\Char(k) \neq 2,3$, every structurable division algebra gives rise to a Moufang set,
and conversely, that every Moufang set arising from a simple linear algebraic group of $k$\dash rank one over a field $k$
with $\Char(k) \neq 2,3$ arises from a structurable division algebra in this fashion.
This necessarily includes an investigation of the aforementioned connection
between structurable (division) algebras and linear algebraic groups (of relative rank one).
Our treatment here is uniform and free from case-by-case considerations.

Our new construction of Moufang sets from structurable division algebras gives, in particular,
an explicit description of the Moufang sets arising from exceptional algebraic groups,
including the permutation ``$\tau$'' of the Moufang sets,
and also including a surprisingly easy description of the ``Hua maps''.
This effectively allows to perform computations, which in turn help to understand the structure of these groups in a different way.
The explicit description of the Moufang sets arising from exceptional groups has been considered to be one of the major challenges in the theory of Moufang sets.

At this moment, such a description was known for some particular cases, namely for groups of type $F_{4,1}^{21}$ (see \cite{DV2}),
for groups of type $^2\!E_{6,1}^{29}$ (see \cite{CD}) and $E_{7,1}^{78}$ (for which we could not find an explicit proof in the literature).
The structure of the root groups for groups of type $\prescript{3,6}{}D_{4,1}^{9}$ and for groups of type $^2\!E_{6,1}^{35}$ was obtained in \cite{CHM},
but this description is in terms of coordinates (and hence gives less insight).
For the other exceptional groups of relative rank one, even the structure of the root groups was, to the best of our knowledge, unknown.

We were inspired by the examples of Moufang sets arising from Jordan algebras and by the known examples of Moufang sets with non-abelian root groups,
in particular those of type $F_{4,1}^{21}$ and those arising as residues of exceptional Moufang quadrangles (implicit in \cite[(33.14)]{TW},
and explicitly described in \cite{DMnota}).
It turns out that each of these examples arises from a specific type of structurable algebra.
Indeed, every Jordan division algebra is also a structurable algebra (with trivial involution);
the Moufang sets of type $F_{4,1}^{21}$ arise from an octonion division algebra with standard involution;
and the residues of exceptional Moufang quadrangles arise from structurable algebras as described in~\cite{BD1}.

As a consequence of our main result, we can rely on the classification of structurable division algebras
(over a field $k$ of characteristic $\neq 2,3$)
to get a complete,
and in most cases very explicit description for all Moufang sets arising from exceptional algebraic groups of
$k$\dash rank one.
The classification of structurable division algebras is deduced from the classification of all central simple
structurable algebras carried out by Allison and Smirnov~\cite{A78,S} when $\Char k\neq 2,3,5$;
in the present paper, we show in addition that
the classification by Allison and Smirnov remains valid
over fields of characteristic $5$, i.e., no new classes of
central simple structurable algebras arise.

\medskip

\section*{Organization of the paper}

In Chapter~\ref{se:moufsets}, we introduce the notions from the theory of Moufang sets that we will need later.
In particular, we give some details about how to associate a Moufang set to each absolutely simple linear algebraic group of $k$\dash rank $1$.
We then recall the construction of Moufang sets from Jordan division algebras and the construction of Moufang sets from skew-hermitian forms,
and in particular, we point out that these two classes together already describe all Moufang sets arising from {\em classical} linear algebraic groups
of $k$\dash rank $1$; see Remark~\ref{rem:ab-rootgroups} and Corollary~\ref{cor:mouf-classical}.
(The Moufang sets arising from Jordan division algebras also include one case of exceptional Moufang sets, namely those of type $E_{7,1}^{78}$.)

In Chapter~\ref{se:structalg}, we recall some facts from the theory of structurable algebras.
We then focus on structurable division algebras, and in section~\ref{se:struct-ex},
we spend some time going over each of the different classes and pointing out when these algebras are division algebras.
This is a non-trivial question, which, in fact, is open in some specific cases; see Remark~\ref{rem:skewdim1} and Remark~\ref{rem:product-div} below.
We then recall the construction of Lie algebras from structurable algebras and we collect some facts about isotopies.

The real work then starts in Chapter~\ref{se:one-inv}.
We build upon earlier work by Bruce Allison and John Faulkner, and study {\em one-invertibility} for structurable algebras $\A$,
which is a property for pairs of elements in $\A \times \Ss$, where $\Ss$ is the subspace of skew elements of $\A$.
Our main result of that section is Theorem~\ref{th:formula-oneinverse} and its Corollary~\ref{cor:oneinverse}, showing that for structurable division algebras,
all non-zero pairs are one-invertible, and providing an explicit formula for the left and right inverse of such a pair.
The notion of one-invertibility turns out to be a crucial ingredient for the construction of the corresponding Moufang sets.
One issue that comes up here is that, in order to include the case $\cha(k) = 5$, we need an additional condition on the structurable algebras,
which we call {\em algebraicity}; see section~\ref{se:alg-gr}.
This condition also plays important role in Chapter~\ref{se:SSA and SAG}.

Chapter~\ref{se:SSA and SAG} deals with the connection between structurable algebras and simple linear algebraic groups.
We do the construction of simple algebraic groups from simple structurable algebras in full generality
(section~\ref{ss:SAG-SSA}),
but for the converse, we restrict to the case of $k$\dash rank one (section~\ref{ss:SSA-SAG}).
Although this connection is believed to be well known, the situation is, in fact, rather delicate.
One issue is that in characteristic $5$, this requires the algebraicity condition mentioned above;
another issue is that there are several different but closely related Lie algebras in the picture, and
it is crucial to distinguish them
carefully and to understand what the connection is between these different Lie algebras.
See, in particular, Construction~\ref{constr:newL}, Corollary~\ref{cor:newl-K} and most notably Lemma~\ref{lem:Gkantor}, which in turn
depends on the recognition Lemma~\ref{lem:H1H2}, the proof of which is rather involved.
Our main results in Chapter~\ref{se:SSA and SAG} are Theorem~\ref{thm:AGadjoint} and Theorem~\ref{thm:AG1}.
We also show that structurable {\em division} algebras are always algebraic, even if $\cha(k) = 5$; see Theorem~\ref{thm:div-alg}.

The main goal of the paper is reached in Chapter~\ref{se:MS-SDA}.
In that section, we show that every structurable division algebra over a field $k$ of characteristic different from $2$ and $3$ gives rise
to a Moufang set, and conversely, every Moufang set arising from an absolutely simple linear algebraic $k$\dash group of $k$\dash rank $1$
arises from a structurable division algebra;
see Theorem~\ref{mainth:moufset} and Theorem~\ref{mainth:algmoufset}.
In Corollary~\ref{co:iso}, we show that the resulting Moufang sets are isomorphic if and only if the structurable division algebras are isotopic.
Theorem~\ref{mainth:moufset} also gives an explicit description of the Moufang set (including its $\tau$-map), directly in terms of the structurable division algebra.
In addition, we get a surprisingly simple expression for the so-called {\em Hua maps} of the Moufang set;
these maps are usually very hard to compute directly from their definition, and it is an unexpected side result of our approach
that we get a description of these maps without any effort in Theorem~\ref{th:hua}.

In the final Chapter~\ref{se:examples}, we illustrate our results by presenting explicit descriptions for all Moufang sets
arising from linear algebraic groups of $k$\dash rank~$1$;
of course, the most interesting results arise for the exceptional linear algebraic groups.
We recover all of the earlier known results, and in addition, we get new descriptions for some Moufang sets that had not been described
earlier in the literature.
See, in particular, Theorem~\ref{th:mouf-D4}, Theorem~\ref{th:mouf-E} and Corollary~\ref{co:tensor}, which together cover all Moufang sets
with non-abelian root groups that arise from exceptional linear algebraic groups of $k$\dash rank $1$.
In section~\ref{se:classchar5}, finally, we extend the classification of structurable division algebras to fields of characteristic $5$.

\section*{Acknowledgments}

Parts of this work appeared, in a different form, in the PhD theses of the first and last author \cite{Bo-thes, St-thes},
and it is our pleasure to thank our PhD supervisors Hendrik Van Maldeghem and Nikolai Vavilov
for their support. The last author's thesis also benefited a lot from conversations and collaboration with Victor Petrov.

We are grateful to Skip Garibaldi for several interesting discussions about the connection between structurable algebras, Lie algebras and linear algebraic groups.

The results in section~\ref{se:divtooneinv}, which were an important step in the project,
arose from discussions with John Faulkner during his visit at Ghent University in October 2012.
We thank him for providing the key idea.

We also express our gratitude to T.N.~Venkataramana for sharing his insight into the root groups of unitary groups,
which was important for our proof of Theorem~\ref{th:mouf-herm},
and to Ottmar Loos, for pointing out to us that some parts of \cite{AF99} require the missing assumption $\cha(k) \neq 5$ (see Remark~\ref{rem:loos}).

The research is supported by the Russian Science Foundation grant 14-21-00035. The intermediate results
in~section~\ref{se:SSA and SAG} were obtained while the last author was
a postdoctoral fellow of the program
6.50.22.2014 ``Structure theory, representation theory and geometry of algebraic groups''
at St. Petersburg State University.

\mainmatter

%
%
\chapter{Moufang sets}\label{se:moufsets}

We repeat the most important notions from the theory of Moufang sets which we will need.
For a general introduction to Moufang sets, we refer the reader to \cite{DS}.

\section{Definitions and basic properties}

As is common in the theory of Moufang sets, we will always denote group actions {\em on the right}.
Also, when $A$ is a group, we will often use the notation $A^*$ to denote $A$ without its identity element.
We denote conjugation in a group by $g^h := h^{-1} g h$.

\begin{definition}\label{def:MS}
    Let $X$ be a set (with $|X|\geq 3$) and $\{U_x \mid x\in X\}$ be a collection of subgroups of $\Sym(X)$.
    The data $\bigl( X,\{U_x\}_{x\in X} \bigr)$ is a {\em Moufang set} if the following two properties are satisfied:
    \begin{compactitem}\itemindent-.8ex
        \Mitem[reg]
        For each $x\in X$, $U_x$ fixes $x$ and acts regularly (i.e.\@ sharply transitively) on $X\setminus\{x\}$.
        \Mitem[conj]
        For each $g \in G^+:=\langle U_x \mid x\in X\rangle$ and each $y \in X$, we have $U_y^g = U_{y.g}$.
    \end{compactitem}
    The group $G^+$ is called the {\em little projective group} of the Moufang set, and the groups $U_x$ are called the {\em root groups}.

    Note that \Mref{reg} implies that $G^+$ acts doubly transitively on $X$.
    The Moufang set is called {\em proper} if the action is not sharply doubly transitive.
\end{definition}

Moufang sets are essentially equivalent to groups with a saturated split BN-pair of rank one, and to abstract rank one groups, a notion introduced by Franz Timmesfeld.
This point of view will, in fact, be suitable for our purposes.
\begin{definition}[{\cite{Timm}}]\label{def:rankonegroup}
An {\em abstract rank one group} (with {\em unipotent subgroups} $A$ and $B$) is a group $G$ together with a pair of
distinct nilpotent subgroups $A$ and $B$ such that $G=\langle A,B\rangle$, and such that
\[ \text{for each } a\in A^*,  \text{there is a (unique) } b(a) \in B^* \text{ such that } A^{b(a)} = B^a , \]
and similarly with the roles of $A$ and $B$ interchanged.
\end{definition}

The following lemma shows that the two previous definitions are essentially equivalent.
\begin{lemma}\label{le:MSandR1G}
\begin{compactenum}[\rm (i)]
    \item
    	Let $ (X,\{U_x\}_{x\in X})$ be a Moufang set with nilpotent root groups, and choose two arbitrary elements $0,\infty \in X$.
	Then  $G^+=\langle U_0,U_\infty \rangle$ is an abstract rank one group.
    \item
	Let $G=\langle A,B\rangle$ be an abstract rank one group. Define
	\begin{align*}
		Y &:= \{ A^g\mid g\in G \} = \{ A^b\mid b\in B \} \cup \{B\} \\
		&\phantom{:}=\{B^g\mid g\in G\}=\{B^a\mid a\in A\}\cup \{A\}.
	\end{align*}
	For each $x\in Y$, define $U_x:=x\leq G$, and consider the action of $G$ on $Y$ by conjugation.
	Then $ (Y,\{U_x\}_{x\in Y})$ is a Moufang set with $G^+ \cong G / Z(G)$.
\end{compactenum}
\end{lemma}
\begin{proof}
	See \cite[Section 2.2]{DS}.
\end{proof}
\begin{remark}
	For many purposes, the additional requirement on the root groups to be nilpotent is not needed, and in fact,
	it is an open problem whether there exist proper Moufang sets with root groups that are not nilpotent.
	Even more, all known examples of proper Moufang sets have root groups of nilpotency class at most three.
\end{remark}

In the following construction we show how to describe a Moufang set using only one group $U$ and a permutation of $U^*$.
The advantage of this description is that the only required data are a group and a permutation.
This should be compared with Definition \ref{def:MS}, where we need a set $X$ and a collection of subgroups of $\Sym(X)$,
and with Definition \ref{def:rankonegroup}, where we need to describe the ambient group $G$ (and two of its subgroups).

\begin{construction}[{\cite[Section 3]{DW}}]\label{constr:MS}
Let $\mouf=(X,(U_x)_{x\in X})$ be an arbitrary Moufang set, and let $0\neq \infty\in X$ be two arbitrary elements.

Define the set $U:=X\setminus\{\infty\}$.
For each $a\in U$, we define $\alpha_a$ as the unique element in $U_\infty$ mapping $0$ to $a$.
By \Mref{reg}, $U_\infty=\{\alpha_a\mid a\in U\}$.

We now make $U$ into a group $(U, +)$ by defining $a+b := a.\alpha_b = 0.\alpha_a\alpha_b$ for all $a,b\in U$.
It is clear that this is a (not necessarily commutative) group, with neutral element $\alpha_0$, and with $\alpha_a^{-1} = \alpha_{-a}$;
moreover, $U \cong U_\infty$.
(The action of $U_\infty$ on $X$ is essentially the right regular representation of the group $(U,+)$.)

By \cite[Proposition 4.1.1]{DS}, there is,
for each $a\in U\setminus\{0\}$, a unique permutation
\begin{equation}\label{eq:mu}
	\mu_a \in U_0^* \alpha_a U_0^* \subseteq G^+
\end{equation}
interchanging $0$ and $\infty$.

Now fix an element $e\in U\setminus\{0\}$ and define $\tau:=\mu_e$; then $\tau$ induces a permutation of $X\setminus\{0,\infty\}=U^*$ (which we also denote by $\tau$).
Since $\tau$ and $\alpha_a$ are both in $G^+$ for all $a\in U$, we have $U_\infty^\tau=U_{\infty.\tau}=U_0$ and $U_0^{\alpha_a}=U_{0.\alpha_a}=U_a$.

It follows that in order to describe the Moufang set $\mouf=(X,(U_x)_{x\in X})$, it is sufficient to know the group $U$ and the permutation $\tau \in \Sym(U^*)$.
Therefore we denote the Moufang set $\mouf=(X,(U_x)_{x\in X})$ by $\mouf(U,\tau)$.
\end{construction}

\begin{remark}
	A slight disadvantage of this description is that the permutation $\tau$ is not uniquely determined by the Moufang set.
	However, it is shown in \cite[Lemma 4.1.2]{DS} that $\mouf(U,\tau)=\mouf(U,\mu_a)$ for all $a\in U^*$,
	and in fact, the data $\bigl( U, (\mu_a)_{a \in U^*} \bigr)$ is uniquely determined by the Moufang set.
    On the other hand, see~\cite{Loos-divpairs} for a different approach to Moufang sets that avoids this issue.
\end{remark}

\begin{remark}\label{rem:constrMSUtau}
	Starting from an arbitrary group $U$ together with a permutation $\tau$ of $U^*$,
    the previous construction can be applied backwards.
	However, the result is not always a Moufang set.
	We refer to \cite[Section 3]{DS} for the details of this construction, and for the condition needed on $U$ and $\tau$ in order to get a Moufang set.
\end{remark}

We will need the following explicit description of the $\mu$-maps.

\begin{lemma}\label{compute tau}
	Let $\mouf(U, \tau)$ be a Moufang set.
	For each $a\in U^*$, we have
	\[ \mu_a = \alpha^\tau_{(-a).\tau^{-1}}\alpha_a \alpha^\tau_{-(a.\tau^{-1})} . \]
\end{lemma}
\begin{proof}
	See \cite[Proposition 4.1.1]{DS}.
\end{proof}

The following concept plays an important role in the theory of Moufang sets.
\begin{definition}\label{def:hua}
	Let $\mouf(U,\tau)$ be a Moufang set.
	Then for each $a\in U^*$, we define a corresponding {\em Hua map} $h_a := \tau \mu_a\in \Sym(X)$.
	It is clear that each $h_a$ fixes $0$ and $\infty$, and it is not hard to verify that $(b+c).h_a = b.h_a + c.h_a$ for all $b,c\in U$,
	i.e.\@ the Hua maps induce automorphisms of the group $U$.
	We also define the {\em Hua group}
	\[ H=\langle h_a\mid a\in U^*\rangle=\langle \mu_a \mu_b\mid a,b\in U^*\rangle\leq \Aut(U) ; \]
	by \cite[Theorem 3.1(ii)]{DW}, the Hua group is precisely the two point stabilizer $\Stab_G(0, \infty)$.
	See also \cite[Lemma 4.2.2]{DS}.
	In particular, a Moufang set is proper if and only if $H\neq 1$.
\end{definition}

\begin{example}\label{ex:PSL2}
	In order to illustrate the various definitions we have introduced, we present the prototypical example of a Moufang set,
	which is the projective line $\mathbb{P}^1(D)$, acted upon by $G = \PSL_2(D)$, where $D$ is an arbitrary field or skew field.
	In this case, the root groups are isomorphic to the additive group of $D$.
	More precisely, if we let $U = (D, +)$, and $\tau \colon D^\times \to D^\times \colon x \mapsto -x^{-1}$,
	then $\mouf(U, \tau)$ is a Moufang set with little projective group $G$; we denote it by $\mouf(D)$.

	For any $a \in D^\times$, the corresponding $\mu$-map $\mu_a$ is the map that takes any $x \in X = D \cup \{ \infty \}$ to $-a x^{-1} a$.
	Notice that $\tau = \mu_1$.
	The corresponding Hua map $h_a$ maps any $x \in X$ to $axa$.
\end{example}

Now we introduce what it means for Moufang sets to be isomorphic.
\begin{definition}\label{def:misom}
    Let $\mouf=(X,\{U_x\}_{x\in X})$ and $\mouf'=(X',\{U_y\}_{y\in X'})$ be two Moufang sets.
    We say that $\mouf$ and $\mouf'$ are {\em isomorphic} if there exists a bijection $\varphi \colon X\to X'$ such that
    the induced map $\Sym(X)\to \Sym(X') \colon g\mapsto \varphi^{-1} g\varphi$ maps each root group $U_x$ isomorphically onto the corresponding root group $U_{x.\varphi}$.

    We call $\varphi$ an {\em isomorphism} from $\mouf$ to $\mouf'$.
\end{definition}
Next we translate the definition of isomorphic Moufang sets into several more useful criteria.
\begin{lemma}\label{le:MSisom}
\begin{compactenum}[\rm (i)]
    \item
        Let $\mouf = \mouf(U,\tau)$ and $\mouf'=\mouf(U',\tau')$ be two Moufang sets.
        Then $\mouf$ and $\mouf'$ are isomorphic if and only if there exists a group isomorphism $\varphi \colon U\to U'$ such that
        $\mouf(U', \tau') = \mouf(U',  \varphi^{-1} \tau \varphi)$.
    \item
        Let $\mouf = \mouf(U,\tau)$ and $\mouf'=\mouf(U',\tau')$ be two Moufang sets,
        and assume that $\tau$ and $\tau'$ are contained in the little projective group $G$ of $\mouf$ and $G'$ of $\mouf'$, respectively.
        Then $\mouf$ and $\mouf'$ are isomorphic if and only if there exists a group isomorphism $\varphi \colon U\to U'$ such that
        $\tau' h' = \varphi^{-1} \tau \varphi$ for some element $h' \in H'$, the Hua group of $\mouf'$.
    \item
        Let $\mouf$ and $\mouf'$ be two Moufang sets with corresponding abstract rank one groups $G=\langle A,B\rangle$ and $G'=\langle A',B'\rangle$.
        Then $\mouf$ and $\mouf'$ are isomorphic if and only if there exists a group isomorphism $\varphi \colon G\to G'$ mapping $A$ to $A'$ and $B$ to $B'$.
\end{compactenum}
\end{lemma}
\begin{proof}
\begin{compactenum}[(i)]
    \item
        First assume that $\mouf=(X,\{U_x\}_{x\in X})=\mouf(U,\tau)$ and $\mouf'=(X',\{U_y\}_{y\in X'})=\mouf(U',\tau')$ are isomorphic with isomorphism $\varphi \colon X\to X'$.
        We follow Construction \ref{constr:MS} with two arbitrary elements $0,\infty\in X$ and with $0'=0.\varphi, \infty'=\infty.\varphi \in X'$.

        It is clear that for all $a\in U^*$, $\alpha_a^\varphi=\varphi^{-1}\alpha_a\varphi=\alpha_{a.\varphi}\in U_{\infty'}$; therefore $\mu_a^\varphi=\mu_{a.\varphi}$.
        Now the restriction $\varphi \colon U=X\setminus\{\infty\}\to U'=X'\setminus\{\infty'\}$ is a group isomorphism, since for all $a,b\in U$
        \[ (a+b).\varphi = (a\alpha_b).\varphi = a.\varphi \alpha_b^\varphi = a.\varphi \alpha_{b.\varphi} = a.\varphi+b.\varphi . \]


        Moreover, by Definition~\ref{def:misom}, $\varphi^{-1} U_0 \varphi = U_{0'}$, and hence
        \begin{equation}\label{eq:misom}
            U_{\infty'}^{\tau'} = U_{0'} = \varphi^{-1} U_0 \varphi = \varphi^{-1} U_\infty^\tau \varphi
            = {\varphi^{-1} \tau^{-1} \varphi}U_{\infty'} {\varphi^{-1} \tau \varphi} = U_{\infty'}^{\varphi^{-1} \tau \varphi} ,
        \end{equation}
        and hence the two Moufang sets $\mouf(U', \tau')$ and $\mouf(U',  \varphi^{-1} \tau \varphi)$ coincide.

        \smallskip

        Conversely, assume that there exists a group isomorphism $\varphi \colon U\to U'$ such that $\mouf(U', \tau') = \mouf(U',  \varphi^{-1} \tau \varphi)$.
        We extend $\varphi \colon X\to X'$ to a bijection of sets by defining $\infty.\varphi=\infty'$.
        Following Construction \ref{constr:MS},
        it is clear that the induced map $\overline{\varphi} \colon \Sym(X) \to \Sym(X')$ maps $U_\infty$ to~$U_{\infty'}$.
        The same computation as in equation~\eqref{eq:misom}, but now using the assumption that $\mouf(U', \tau') = \mouf(U',  \varphi^{-1} \tau \varphi)$,
        yields that $\overline{\varphi}$ maps $U_0$ to~$U_{0'}$.
        The result now follows.
    \item
        This follows immediately from (i), since two elements of $G'$ that swap $0'$ and $\infty'$ differ by an element of $H'$;
        see Definition~\ref{def:hua}.
    \item
        This is clear from Lemma \ref{le:MSandR1G}.
    \qedhere
\end{compactenum}
\end{proof}

\section{Moufang sets from linear algebraic groups}

One of the main motivations for studying Moufang sets is their connection with (semi)simple linear algebraic groups of relative rank one.
We now explain this connection in some detail.

Let $\GG$ be a semisimple linear algebraic group over an arbitrary field $k$, and assume that $\GG$ has $k$\dash rank $1$.
Let $X$ be the set of all proper parabolic $k$\dash subgroups of $\GG$;
note that the assumption on the $k$\dash rank implies that every proper parabolic $k$\dash subgroup is in fact a minimal parabolic $k$\dash subgroup.
For each $\GP \in X$, we let $U_\GP := \radu(\GP)(k)$, i.e.\@ the set of $k$\dash rational points of the unipotent radical of $\GP$.
\begin{theorem}\label{th:M(G)}
	Let $\GG$ be a semisimple linear algebraic group over an arbitrary field $k$ of $k$\dash rank $1$,
	and let $X$ and the groups $U_\GP$ be as above.
	Then $\mouf(\GG) := \bigl( X, (U_\GP)_{\GP \in X} \bigr)$ is a Moufang set.
\end{theorem}
\begin{proof}
	We decided to include an explicit proof of this (known) fact, which seems hard to find in the literature.
	Recall that two $k$\dash parabolics $\GP$ and $\GP'$ in $X$ are called {\em opposite} if their intersection is a Levi subgroup of both,
	or equivalently, if $\GP \cap \radu(\GP') = \GP' \cap \radu(\GP) = 1$; see \cite[4.8 and 4.10]{BorelTits}.
	By \cite[4.8]{BorelTits} again, each $k$\dash parabolic $\GP$ has at least one opposite $k$\dash parabolic $\GP'$,
	and any two such opposites $\GP'$ and $\GP''$ are conjugate by a unique element of $\radu(\GP)(k)$.

	We claim that, in our situation where the $k$\dash rank is $1$, any two elements $\GP, \GP'$ of $X$ are, in fact, opposite to each other.
	Indeed, let $\GT$ be a maximal $k$\dash split torus for which $C_G(\GT) \leq \GP \cap \GP'$;
	such a torus exists by \cite[4.18]{BorelTits}.
	Also, because $\GG$ has $k$\dash rank $1$ and hence a relative Weyl group of order $2$, \cite[5.9]{BorelTits} implies that $\GP$ and $\GP'$
	are the only two minimal $k$\dash parabolics containing $\GT$.

	Now $\GP$ is opposite to some minimal $k$\dash parabolic $\GP''$, and again there is some maximal $k$\dash split torus $\GS$ such that
	$C_G(\GS) \leq \GP \cap \GP''$.
	Since all maximal $k$\dash split tori are $k$\dash conjugate, it follows that $\GP \cap \GP''$ is $k$\dash conjugate to $\GP \cap \GP'$.
	Using our previous observation that $\GP$ and $\GP'$ are the only two minimal $k$\dash parabolics containing $\GT$,
	we conclude that $\GP$ and $\GP'$ are also opposite to each other, and this proves the claim.

	It follows that for any $\GP \in X$, the group $U_\GP = \radu(\GP)(k)$ acts sharply transitively on the remaining elements of $X$.
	It is also clear that the groups $U_\GP$ are permuted by conjugation by elements of $\GG(k)$,
	and this shows that $\mouf(\GG) := \bigl( X, (U_\GP)_{\GP \in X} \bigr)$ is indeed a Moufang set.
\end{proof}

\begin{example}
	Let $k$ be a commutative field, and let $D$ be a division algebra over $k$ of degree $d$.
	Then the group $\PSL_2(D)$ is (the group of $k$\dash rational points of) an algebraic group $\GG$ over $k$ with Tits index
	\[ A_{2d-1,1}^{(d)} \quad
	\begin{tikzpicture}[line width=1pt, scale=.8, baseline={(0,-.1)}]
	    \draw (0,0) -- (1.3,0);
	    \draw[dotted] (1.3,0) -- (2.7,0);
	    \draw (2.7,0) -- (5.3,0);
	    \draw[dotted] (5.3,0) -- (6.7,0);
	    \draw (6.7,0) -- (8,0);
	    \diagnode{(0,0)}
	    \diagnode{(1,0)}
	    \diagnode{(3,0)}
	    \diagnode{(4,0)}
	    \diagnode{(5,0)}
	    \diagnode{(7,0)}
	    \diagnode{(8,0)}
	    \distorbit{(4,0)}
	\end{tikzpicture} \ , \]
	and every group $\GG$ over $k$ with this Tits index is obtained in this fashion; see \cite[p.\@ 55]{Boulder}.
	The root groups of the Moufang set $\mouf(\GG)$ are isomorphic to the additive group of $D$, and since $\GG(k) = \PSL_2(D)$ is simple,
	it coincides with the little projective group of $\mouf(\GG)$.
	Hence we have recovered the Moufang sets from Example~\ref{ex:PSL2} for skew fields which are finite-dimensional over their center.
	(Clearly, if the skew field is infinite-dimensional over its center, then the corresponding Moufang set cannot arise from an algebraic group in this fashion.)
\end{example}

\begin{definition}[{\cite[0.7]{BorelTits}}]
    A semisimple $k$\dash group $\GG$ is called \emph{$k$\dash simple} (respectively,
    \emph{almost $k$\dash simple}), if every proper normal closed $k$\dash subgroup of $\GG$ is trivial (respectively, finite).
    A semisimple $k$\dash group $\GG$ is called \emph{absolutely simple}, or just \emph{simple}, if $\GG_{\bar k}$ is $\bar k$-simple, where
    $\bar k$ is an algebraic closure of $k$, or, equivalently, if $\GG$ is adjoint and the
    root system of $\GG_{\bar k}$ is irreducible.
\end{definition}

Note that a $k$\dash simple algebraic $k$\dash group $\GG$ is automatically adjoint, since $\Cent(\GG)$ is trivial.
Conversely, an adjoint almost $k$\dash simple group is $k$\dash simple.

\begin{lemma}\label{lem:M(G)-E}
    Let $k$ be a field such that $|k|\ge 4$.
    Let $\GG$ be an adjoint semisimple algebraic $k$\dash group of $k$\dash rank 1. For any two opposite proper parabolic
    subgroups $\GP_+$, $\GP_-$ of $\GG$, the subgroup $\GG(k)^+=\left<\radu(\GP_+)(k),\radu(\GP_-)(k)\right>$
    of $\GG(k)$ is isomorphic to the little projective group of the Moufang set $\mouf(\GG)$.
\end{lemma}
\begin{proof}
    By Lemma~\ref{le:MSandR1G} the group $\GG(k)^+$ is an abstract rank one group corresponding to $\mouf(\GG)$,
    and the little projective group $\mouf(\GG)^+$ is isomorphic to $\GG(k)^+/Z(\GG(k)^+)$. Thus, it is enough
    to show that $Z(\GG(k)^+)=1$. By~\cite[Proposition 6.2]{BoTi-hom} $\GG(k)^+$ coincides with the subgroup of $\GG(k)$ generated
    by the $k$\dash points of all unipotent radicals of all parabolic $k$\dash subgroups of $\GG$.
    Since $|k|\ge 4$, by the main theorem of~\cite{Tits64}, any subgroup
    of $\GG(k)$ normalized by $\GG(k)^+$ is central in $\GG$ or contains $\GG(k)^+$. Clearly, $\GG(k)^+$
    is non-abelian (for example, since it is perfect by~\cite[Corollaire 6.4]{BoTi-hom}). Therefore,
    $Z(\GG(k)^+)\subseteq \Cent(\GG)(k)$. Since $\GG$ is adjoint, $\Cent(\GG)$ is trivial, and hence $Z(\GG(k)^+)$
    is trivial.
\end{proof}

\begin{lemma}\label{lem:M(G)-isog}
    Let $k$ be an arbitrary field, and let $\GG$ be a semisimple algebraic $k$\dash group of $k$\dash rank 1. Let
    $\GP$ be a minimal parabolic $k$\dash subgroup of $\GG$.
    \begin{compactenum}[\rm (i)]
        \item\label{M(G)-isog:isog}
            Let $\GG'$ be a semisimple $k$\dash group and $f \colon \GG\to\GG'$ be a central $k$\dash isogeny.
            Then $\GG'$ has $k$\dash rank 1 and the assignment $\GP\mapsto f(\GP)$ induces an isomorphism of Moufang sets $\mouf(\GG)\cong\mouf(\GG')$.
        \item
            There is a unique minimal closed normal semisimple $k$\dash subgroup $\GG_0$ of $\GG$ such that the $k$\dash rank of $\GG_0$ is 1.
            This $k$\dash subgroup $\GG_0$ is almost $k$\dash simple, and
            the assignment $\GP\mapsto \GP\cap\GG_0$ induces an isomorphism $\mouf(\GG_0)\cong\mouf(\GG)$.
            If $\GG$ is adjoint, then $\GG_0$ is adjoint.
        \item
            If $\GG$ is $k$\dash simple, then there is a finite separable field extension $l/k$ and an
            absolutely simple algebraic $l$-group $\GH$ of $l$-rank 1, such that $\GG$ is $k$\dash isomorphic to the Weil restriction $R_{l/k}(\GH)$.
            The assignment $\GQ\mapsto  R_{l/k}(\GQ)$, where $\GQ$ is a minimal parabolic $l$-subgroup of $\GH$, induces an isomorphism $\mouf(\GH)\cong\mouf(\GG)$.
    \end{compactenum}
\end{lemma}
\begin{proof}
    \begin{compactenum}[\rm (i)]
        \item
            By~\cite[Theorem 22.6]{Bo-book} parabolic subgroups of $\GG'$ are the images of parabolic subgroups
            of $\GG$, and parabolic subgroups of $\GG$ are preimages of parabolic subgroups of $\GG'$. Since $f$
            is central, $\ker f$ is contained in any Levi subgroup of any minimal parabolic subgroup $\GP$ of $\GG$.
            Therefore, $f$ induces a bijection beween the sets of minimal parabolic subgroups of $\GG$ and $\GG'$,
            and for any minimal parabolic subgroup $\GP$ of $\GG$ we have $\radu(\GP)\cong \radu(f(\GP))$.
            Therefore, $f$ induces an isomorphism $\mouf(\GG)\cong\mouf(\GG')$.
            (See also~\cite[Th\'eor\`eme 2.20]{BorelTits-compl}.)
        \item
            By~\cite[Theorem 22.10]{Bo-book} there is a central $k$\dash isogeny
            $f \colon \prod\limits_{i=0}^n\GG_i\to \GG$, where
            $\GG_0,\GG_1,\ldots,\GG_n$ are all minimal closed connected normal $k$\dash subgroups of $\GG$ of positive dimension.
            Also, $\GG_0,\ldots,\GG_n$ are almost $k$\dash simple. By~\cite[Proposition 22.9]{Bo-book} the $k$\dash rank of $\GG$
            is equal to the sum of $k$\dash ranks of $\GG_i$, $0\le i\le n$. Therefore, there is a unique index $i$ such that
            the $k$\dash rank of $\GG_i$ is 1; we can assume that $i=0$ without loss of generality. Since, clearly, $\Cent(\GG_0)\subseteq\Cent(\GG)$,
            we conclude that $\GG_0$ is adjoint, once $\GG$ is adjoint.

            In order to show that the assignment $\GP\mapsto \GP\cap\GG_0$ induces an isomorphism of Moufang sets,
            by~\eqref{M(G)-isog:isog} we can assume that $\prod\limits_{i=0}^n\GG_i=\GG$.
            Let $\GP$ be a minimal parabolic $k$\dash subgroup of $\GG$, and let $\GS$ be
            a maximal split $k$\dash subtorus $\GS$ of $\GG$ such that $\Cent_{\GG}(\GS)$ is a Levi subgroup of $\GP$.
            Clearly, $\GS$ is contained in $\GG_0$. Hence
            \begin{equation}\label{eq:Gi-Levi}
            \prod_{i=1}^n\GG_i\subseteq \Cent_{\GG}(\GS)\subseteq\GP.
            \end{equation}
            This shows that the map $\GP\mapsto\GP\cap\GG_0$ is an injection from the set of minimal parabolic subgroups
            of $\GG$ to the set of smooth closed $k$\dash subgroups of $\GG_0$.
            Furthermore, since $\GG_0/\GP\cap\GG_0\cong\GG/\GP$ is projective, we conclude that $\GP\cap\GG_0$ is parabolic.
            Conversely, for any proper parabolic $k$\dash subgroup $\GQ$ of $\GG_0$, the proper smooth
            $k$\dash subgroup $\GP=\GQ\times\prod\limits_{i=1}^n\GG_i$ of $\GG$ is parabolic.
            This shows that the minimal parabolic subgroups of $\GG$ and $\GG_0$ are in bijective correspondence.
            Moreover, by~\eqref{eq:Gi-Levi} we have
            \[
            \radu(\GP)=\radu(\GQ\times\prod_{i=1}^n\GG_i)=\radu(\GQ).
            \]
            Hence the map $\GP\mapsto\GP\cap\GG_0$ induces an isomorphism $\mouf(\GG)\cong\mouf(\GG_0)$.
        \item
            Since $\GG$ is $k$\dash simple, it is adjoint.
            By~\cite[Exp. XXIV Proposition 5.10]{SGA3} any adjoint semisimple
            $k$\dash group is isomorphic to a direct product of semisimple groups of the form $R_{A/k}(\GH)$, where $A$ is a finite-dimensional
            commutative \'etale $k$\dash algebra, and $\GH$ is an adjoint semisimple group scheme over $A$ such that
            $\GH_{\overline{k(s)}}$ is simple for any $s\in\Spec A$. Since $\GG$ is $k$\dash simple,
            there can be only one factor of this form, that is, $\GG\cong R_{A/k}(\GH)$. Further, such an algebra $A$
            over $k$ is isomorphic to a finite product $l_1\times l_2\times\ldots\times l_n$, where $l_i$, $1\le i\le n$, are
            finite separable field extensions of $k$. Since $R_{l_1\times l_2\times\ldots\times l_n/k}(\GH)\cong\prod_{i=1}^nR_{l_i/k}(\GH)$,
            we conclude that $n=1$, that is, $\GG\cong R_{l/k}(\GH)$, where $l$ is a finite separable field extension of $k$
            and $\GH$ is an adjoint absolutely simple $l$-group.
            (See also~\cite[6.21(ii)]{BorelTits} for a similar statement for semisimple almost $k$\dash simple groups.)

            By~\cite[Corollaire 6.19]{BorelTits} the functor $R_{l/k}$ provides a bijection between the sets of
            parabolic $l$-subgroups of $\GH$ and parabolic $k$\dash subgroups of $\GG$.
            Since the $k$\dash rank of $\GG$ is 1, every proper parabolic $k$\dash subgroup of $\GG$ is minimal, hence
            the same is true for $\GH$. This proves that $\GH$ has $l$-rank~1. Further, for any
            any minimal parabolic $l$-subgroup $\GQ$ of $\GH$, we have
            $\radu(R_{l/k}\bigl(\GQ)\bigr)(k)=R_{l/k}\bigl(\radu(\GQ)\bigr)(k)=\radu(\GQ)(l)$. Then, clearly,
            $R_{l/k}$ induces an isomorphism $\mouf(\GH)\cong\mouf(\GG)$.
        \qedhere
    \end{compactenum}
\end{proof}

\begin{remark}
	For later reference, we present the complete list of {\em exceptional} Tits indices of relative rank $1$ in Table~\ref{ta:exc};
    there are exactly nine families.
	In every case, each root group $U=U_{\GP}$ is either an abelian group (when the relative root system is of type~$A_1$),
	or $U$ is nilpotent of class two (when the relative root system is of type~$BC_1$;
	in either case, both $Z(U)$ and $U/Z(U)$ have the structure of a vector space over $k$,
	and their dimensions can be computed from the root systems.
	See, for example, \cite[p.\@ 44]{Boulder}.

	\begin{table}[ht!]
	\renewcommand{\arraystretch}{2.6}
	\hspace*{-4ex}
	\begin{center}
	\begin{tabular}{|c|c|c|c|}
		\hline
		{index} & {diagram} & $\dim(U/Z(U))$ & $\dim(Z(U))$ \\
		\hline
		\hline
		$\prescript{3,6}{}\!D_{4,1}^9$ &
		\begin{tikzpicture}[line width=1pt, scale=.7, baseline={(0,-.2)}]
		    \draw (0,0) arc (180:90:.5) -- (1,.5);
		    \draw (0,0) arc (180:270:.5) -- (1,-.5);
		    \draw (0,0) -- (1,0);
		    \diagnode{(0,0)}
		    \diagnode{(1,0)}
		    \diagnode{(1,.5)}
		    \diagnode{(1,-.5)}
		    \distorbit{(0,0)}
		\end{tikzpicture}
		& $8$ & $1$
		\\
		\hline
		$F_{4,1}^{21}$ &
		\begin{tikzpicture}[line width=1pt, scale=.7, baseline={(0,-.2)}]
			\draw (0,0) -- (1,0);
			\draw (1,.10) -- (2,.10);
			\draw (1,-.10) -- (2,-.10);
			\draw (2,0) -- (3,0);
			\draw[line width=.7pt] (1.7,.3) -- (1.4,0) -- (1.7,-.3);
			\diagnode{(0,0)}
			\diagnode{(1,0)}
			\diagnode{(2,0)}
			\diagnode{(3,0)}
			\distorbit{(0,0)}
		\end{tikzpicture}
		& $8$ & $7$
		\\
		\hline
		$^2\!E_{6,1}^{35}$ &
		\begin{tikzpicture}[line width=1pt, scale=.7, baseline={(0,-.2)}]
		    \draw (0,0) -- (1,0);
		    \draw (1,0) arc (180:90:.5) -- (3,.5);
		    \draw (1,0) arc (180:270:.5) -- (3,-.5);
		    \diagnode{(0,0)}
		    \diagnode{(1,0)}
		    \diagnode{(2,.5)} \diagnode{(2,-.5)}
		    \diagnode{(3,.5)} \diagnode{(3,-.5)}
		    \distorbit{(0,0)}
		\end{tikzpicture}
		& $20$ & $1$
		\\
		\hline
		$^2\!E_{6,1}^{29}$ &
		\begin{tikzpicture}[line width=1pt, scale=.7, baseline={(0,-.2)}]
		    \draw (0,0) -- (1,0);
		    \draw (1,0) arc (180:90:.5) -- (3,.5);
		    \draw (1,0) arc (180:270:.5) -- (3,-.5);
		    \diagnode{(0,0)}
		    \diagnode{(1,0)}
		    \diagnode{(2,.5)} \diagnode{(2,-.5)}
		    \diagnode{(3,.5)} \diagnode{(3,-.5)}
		    \distlongorbit{3}
		\end{tikzpicture}
		& $16$ & $8$
		\\
		\hline
		$E_{7,1}^{78}$ &
		\begin{tikzpicture}[line width=1pt, scale=.7, baseline={(0,.1)}]
		    \draw (0,0) -- (5,0);
		    \draw (2,0) -- (2,1);
		    \diagnode{(0,0)}
		    \diagnode{(1,0)}
		    \diagnode{(2,0)}
		    \diagnode{(3,0)}
		    \diagnode{(4,0)}
		    \diagnode{(5,0)}
		    \diagnode{(2,1)}
		    \distorbit{(5,0)}
		\end{tikzpicture}
		& $0$ & $27$
		\\
		\hline
		$E_{7,1}^{66}$ &
		\begin{tikzpicture}[line width=1pt, scale=.7, baseline={(0,.1)}]
		    \draw (0,0) -- (5,0);
		    \draw (2,0) -- (2,1);
		    \diagnode{(0,0)}
		    \diagnode{(1,0)}
		    \diagnode{(2,0)}
		    \diagnode{(3,0)}
		    \diagnode{(4,0)}
		    \diagnode{(5,0)}
		    \diagnode{(2,1)}
		    \distorbit{(0,0)}
		\end{tikzpicture}
		& $32$ & $1$
		\\
		\hline
		$E_{7,1}^{48}$ &
		\begin{tikzpicture}[line width=1pt, scale=.7, baseline={(0,.1)}]
		    \draw (0,0) -- (5,0);
		    \draw (2,0) -- (2,1);
		    \diagnode{(0,0)}
		    \diagnode{(1,0)}
		    \diagnode{(2,0)}
		    \diagnode{(3,0)}
		    \diagnode{(4,0)}
		    \diagnode{(5,0)}
		    \diagnode{(2,1)}
		    \distorbit{(4,0)}
		\end{tikzpicture}
		& $32$ & $10$
		\\
		\hline
		$E_{8,1}^{133}$ &
		\begin{tikzpicture}[line width=1pt, scale=.7, baseline={(0,.1)}]
		    \draw (0,0) -- (6,0);
		    \draw (2,0) -- (2,1);
		    \diagnode{(0,0)}
		    \diagnode{(1,0)}
		    \diagnode{(2,0)}
		    \diagnode{(3,0)}
		    \diagnode{(4,0)}
		    \diagnode{(5,0)}
		    \diagnode{(6,0)}
		    \diagnode{(2,1)}
		    \distorbit{(6,0)}
		\end{tikzpicture}
		& $56$ & $1$
		\\
		\hline
		$E_{8,1}^{91}$ &
		\begin{tikzpicture}[line width=1pt, scale=.7, baseline={(0,.1)}]
		    \draw (0,0) -- (6,0);
		    \draw (2,0) -- (2,1);
		    \diagnode{(0,0)}
		    \diagnode{(1,0)}
		    \diagnode{(2,0)}
		    \diagnode{(3,0)}
		    \diagnode{(4,0)}
		    \diagnode{(5,0)}
		    \diagnode{(6,0)}
		    \diagnode{(2,1)}
		    \distorbit{(0,0)}
		\end{tikzpicture}
		& $64$ & $14$
		\\
		\hline
	\end{tabular}
	\renewcommand{\arraystretch}{1}
    \newline
	\caption{Exceptional Tits indices of relative rank one}\label{ta:exc}
	\end{center}
	\vspace*{-2ex}
	\end{table}
\end{remark}

\section{Moufang sets from Jordan algebras}\label{MS:Jordan}

We first recall the definition of quadratic Jordan algebras, as introduced by K.~McCrimmon \cite{Mc1}.

\begin{definition}
        Let $k$ be an arbitrary commutative field,
        let $J$ be a vector space over $k$ of arbitrary dimension,
        and let $1\in J^*$ be a distinguished element.
        For each $x\in J$, let%
        \footnote{It is customary to denote these operators by $U_x$, and we hope that this will not cause any confusion with our similar notation for the root groups
        of a Moufang set.  In fact, this notation is so standard that these operators are often referred to as {\em the U-operators} of the Jordan algebra.}
        $U_x\in\End_k(J)$,
        and assume that the map
        $U \colon J \to \End(J) \colon x \mapsto U_x$ is quadratic, i.e.
        \begin{align*}
                &U_{tx} = t^2 U_x \text{ for all $t\in k$, and} \\
                &\text{the map }(x,y) \mapsto U_{x,y}\text{ is $k$\dash bilinear},
        \end{align*}
        where
        \[ U_{x,y} := U_{x+y}-U_x-U_y \]
        for all $x, y\in J$.
        Let
        \[  V_{x,y} z := \{ x \ y \ z \} := U_{x,z} y \]
        for all $x, y, z\in J$.
        Then the triple $(J, U, 1)$ is a \textit{quadratic Jordan algebra} if the identities
        \begin{compactitem}
           \item[(QJ$_1$)]  $U_1={\rm id}_J$\,;
           \item[(QJ$_2$)]  $U_xV_{y,x}=V_{x,y}U_x$\,;
           \item[(QJ$_3$)]  $U_{U_x y} = U_xU_yU_x$ \qquad [``the fundamental identity''] \label{page:QJ}
        \end{compactitem}
        hold \textit{strictly}, i.e.\@ if they continue to hold
        in all scalar extensions of $J$.  (It suffices for them to hold in the polynomial
        extension $J_{k[t]}$ and this is automatically true if the base field $k$ has at least $4$
        elements.)

        An element $x\in J$ is called \textit{invertible} if there exists $y\in J$
        such that
        \[ U_x y = x \quad \text{ and } \quad U_x U_y 1 = 1. \]
        In this case $y$ is called the \textit{inverse} of $x$ and is denoted $y=x^{-1}$.
        An element $x\in J$ is invertible if and only if $U_x$ is invertible, and we then have $U_x^{-1}=U_{x^{-1}}$.
        If all elements in $J^*$ are invertible, then $(J, U, 1)$ is called a quadratic Jordan \textit{division} algebra.

        We will often simply write $J$ in place of the triple $(J, U, 1)$.
\end{definition}

\begin{theorem}[\cite{DW}]\label{th:jordan}
	Let $J$ be a quadratic Jordan division algebra.
	Let $U$ be the additive group of $J$, and let $\tau \colon U^* \to U^* \colon x \mapsto -x^{-1}$.
	Then $\mouf(U,\tau)$ is a Moufang set, which we will denote by $\mouf(J)$.
	Moreover, for each $a \in U^*$, the Hua map $h_a$ coincides with the map $U_a$.
\end{theorem}

\begin{remark}\label{rem:ab-rootgroups}
	All known examples of proper Moufang sets with abelian rout groups can be described in this fashion,
	but it seems a hard problem to determine whether there exist other examples.
\end{remark}

\begin{remark}
    Over fields $k$ with $\Char(k) \neq 2$, an equivalent (and in fact the more classical) definition is that $J$ is a Jordan algebra over $k$
    if it is a unital commutative non-associative%
    \footnote{It is customary to talk about non-commutative or non-associative algebras to mean
        ``not necessarily commutative'' and ``not necessarily associative'' algebras, respectively.}
    algebra such that
    \[ (x^2 \cdot y) \cdot x = x^2 \cdot (y \cdot x) \]
    for all $x,y \in J$.
    The correspondence with the definition of a quadratic Jordan algebra is given by
    $U_x y = 2x \cdot (x \cdot y) - x^2 \cdot y$ for all $x,y \in J$.
    In particular,
    \begin{equation}\label{eq:JordanV}
        V_{x,y} z = \{ x \ y \ z \} = 2 \bigl( (x \cdot y) \cdot z + (z \cdot y) \cdot x - (z \cdot x) \cdot y \bigr)
    \end{equation}
    for all $x,y,z \in J$.
\end{remark}

\begin{example}\label{ex:abmoufsets}
\begin{compactenum}[(i)]
    \item
        Let $D$ be an associative division algebra over $k$, and let $J=D^+$ be the quadratic Jordan algebra defined by $D$, i.e.\@ we define $U_a b := aba$
        for all $a,b \in D$.
        (When $\Char(k) \neq 2$, this means that we define a Jordan multiplication by the rule $a \cdot b := (ab + ba)/2$.)
        Then $\mouf(J) = \mouf(D)$, where $\mouf(D)$ is as in Example~\ref{ex:PSL2}.
    \item\label{it:Albert}
        Let $J$ be an exceptional quadratic Jordan division algebra (i.e.\@ an Albert division algebra).
        In this case, the corresponding Moufang set arises from a linear algebraic group of type $E_{7,1}^{78}$;
        this follows, for instance, from \cite[14.31]{Springer} (over fields of any characteristic).
\end{compactenum}
\end{example}

\section{Moufang sets from skew-hermitian forms}

A large class of Moufang sets with non-abelian root groups arises from skew-hermitian forms
(see Definition~\ref{def:mouf-skewherm} below).
In this section, we will show that, when the skew-hermitian form is finite-dimensional over a skew field which is finite-dimensional over its center,
then such a Moufang set arises from a classical linear algebraic group,
and conversely, every Moufang set arising from a classical linear algebraic group and with non-abelian root groups is isomorphic to
the Moufang set of a skew-hermitian form.

We will restrict to the case where the center of the underlying skew field has characteristic different from $2$,
although this result can be generalized to arbitrary characteristic.
However, this requires the more general notion of a pseudo-quadratic form, and the resulting unitary groups are not always reductive
but only pseudo-reductive in general.

\begin{definition}
    Let $D$ be a skew field with involution $\sigma$, and let $V$ be a non-trivial right $D$-module.
    Assume that the center $k := Z(D)$ has $\cha(k) \neq 2$.
    \begin{compactenum}[\rm (i)]
        \item
            A {\em skew-hermitian form} on $V$ is a bi-additive map $h \colon V \times V \to D$ such that
            $h(v,w)^\sigma = -h(w,v)$ and
            $h(va, wb) = a^\sigma h(v,w) b$ for all $v,w \in V$ and all $a,b \in D$.
        \item
            We will use the notation
            \[ D_\sigma := \{ a + a^\sigma \mid a \in D \} = \Fix_D(\sigma) , \]
            where the equality holds by our assumption that $\cha(k) \neq 2$.
        \item
            A skew-hermitian form $h$ on $V$ is {\em non-degenerate} if $h(v,V) = 0$ only for $v=0$,
            and is {\em anisotropic} if $h(v,v) = 0$ only for $v = 0$.
            In particular, if $h$ is anisotropic, then $h$ is non-degenerate.
            Notice that when $h$ is anisotropic, then in fact $h(v,v) \not\in D_\sigma$ whenever $v \neq 0$ because $h(v,v)^\sigma = -h(v,v)$ for all $v \in V$.
        \item
            Let $h \colon V \times V \to D$ be a skew-hermitian form on $V$.
            The {\em unitary group} $U(h)$ is defined as
            \[ U(h) := \{ \varphi \in \GL_D(V) \mid h(\varphi(v), \varphi(w)) = h(v,w) \text{ for all } v,w \in V \} . \]
    \end{compactenum}
\end{definition}

\begin{definition}\label{def:mouf-skewherm}
    Let $D$ be a skew field with involution $\sigma$, let $V$ be a non-trivial right $D$-module,
    and let $h \colon V \times V \to D$ be an anisotropic skew-hermitian form on $V$.
    Define a group
    \[ U := \{ (v, a) \in V \times D \mid h(v,v) = a - a^\sigma \} , \]
    with group operation
    \[ (v, a) + (w, b) := \bigl( v + w, a + b + h(w,v) \bigr) , \]
    and define a map $\tau \in \Sym(U^*)$ given by
    \[ (v, a).\tau := (-va^{-1}, a^{-1}) \]
    for all $(v,a) \in U^*$.
    Then $\mouf(U, \tau)$ is a Moufang set, which we call a {\em skew-hermitian Moufang set} or a {\em Moufang set of skew-hermitian type},
    and which we denote by $\mouf(D, \sigma, V, h)$, or simply by $\mouf(h)$.
\end{definition}

\begin{theorem}\label{th:mouf-herm}
    Let $D$ be a skew field with involution $\sigma$, let $V$ be a non-trivial right $D$-module,
    and let $h \colon V \times V \to D$ be an anisotropic skew-hermitian form on $V$.
    Assume that $D$ is finite\dash dimensional over its center $k=Z(D)$, and that $V$ is finite-dimensional over $D$.
    Let $\hat V := D \oplus V \oplus D$, and consider the skew-hermitian form
    \[ \hat h \colon \hat V \times \hat V \to D \colon \bigl((a, v, b), (c, w, d)\bigr) \mapsto a^\sigma d - b^\sigma c + h(v,w) \]
    of Witt index $1$.
    Then the unitary group $U(\hat h)$ is a semisimple linear algebraic group of $k$\dash rank $1$,
    and its Moufang set is isomorphic to $\mouf(h)$.
\end{theorem}
\begin{proof}
	The fact that $G = U(\hat h)$ is a semisimple linear algebraic group of $k$\dash rank~$1$ is well known.
    One possible reference, which also includes the description in characteristic $2$ using pseudo-quadratic forms,
    is \cite[Section~1]{BruhatTitsII}.


	Following the same ideas as in \cite[Example 6.6(3)]{Borel}, we see that the root groups can be described as follows.
    (See also \cite[(10.1.2)]{BruhatTits}.
    We are grateful to T.N.~Venkataramana for an enlightening discussion.)
    Let $w := (1, 0, 0) \in \hat V$, and consider the partial flag
    \[ 0 \leq \langle w \rangle \leq \langle w \rangle^\perp \leq \hat V . \]
    (Notice that $\langle w \rangle^\perp = D \oplus V \oplus 0 \leq \hat V$.)
    The set of elements of $U(\hat h)$ stabilizing this partial flag is a parabolic $k$\dash subgroup $P$ of $G$.
    Its unipotent radical consists of the elements of $P$ acting trivially on $\langle w \rangle$, on $\langle w \rangle^\perp /\langle w \rangle$
    and on $\hat V / \langle w \rangle^\perp$.

    For each element $(z,s) \in U$, consider the map
    \[ \alpha_{z,s} \colon \hat V \to \hat V \colon (a,v,b) \mapsto \bigl( a + h(z,v) + sb, v+zb, b \bigr) . \]
    We claim that the unipotent radical of $P$ consist precisely of these elements $\alpha_{z,s}$.
    First, the elements $\alpha_{z,s}$ are indeed contained in $U(\hat h)$;
    this is an easy computation, using the fact that $h(z,z) = s - s^\sigma$ because $(z,s) \in U$.
    Next, the elements $\alpha_{z,s}$ are contained in the unipotent radical of $P$ because they fix $w$, stabilize $w^\perp$,
    and act trivially on both $\langle w \rangle^\perp /\langle w \rangle$ and $\hat V / \langle w \rangle^\perp$.

    Conversely, any element of $P$ fixing $w$, stabilizing $w^\perp$
    and acting trivially on both $\langle w \rangle^\perp /\langle w \rangle$ and $\hat V / \langle w \rangle^\perp$
    takes the form
    \[ \alpha \colon (a, v, b) \mapsto (a + \varphi(v) + \psi(b),\, v + \chi(b),\, b) \]
    for some $k$\dash linear maps $\varphi \colon V \to D$, $\psi \colon D \to D$ and $\chi \colon D \to V$.
    Expressing that $\alpha$ preserves $\hat h$ then gives rise to the identities
    \begin{align*}
        h(\chi(b), v) &= b^\sigma \varphi(v), \\
        h(\chi(b), \chi(d)) &= b^\sigma \psi(d) - \psi(b)^\sigma d
    \end{align*}
    for all $v \in V$ and all $b,d \in D$.

    Let $z := \chi(1) \in V$ and let $s := \psi(1) \in D$.
    Taking $b=1$ in the first identity already gives $\varphi(v) = h(z, v)$ for all $v \in V$.
    This same identity now takes the form $h(\chi(b), v) = b^\sigma h(z,v) = h(zb, v)$ for all $v \in V$, and since $h$ is non-degenerate,
    this implies $\chi(b) = zb$ for all $b \in D$.
    The second identity now becomes $b^\sigma h(z,z) d = b^\sigma \psi(d) - \psi(b)^\sigma d$ for all $b,d \in D$.
    Substituting $b = 1$ gives $h(z,z) d = \psi(d) - s^\sigma d$ for all $d \in D$.
    Setting in addition $d = 1$ shows that $h(z,z) = s - s^\sigma$,
    and substituting this back in the previous equation then implies that $\psi(d) = sd$ for all $d \in D$.
    We conclude that $(z,s) \in U$ and $\alpha = \alpha_{z,s}$.

    The elements of the Moufang set of $G$, i.e.\@ the minimal $k$\dash parabolics of~$G$, correspond precisely to the $D$-equivalence classes of isotropic elements $w \in \hat V$
    (i.e.\@ those $w \in \hat V$ for which $\hat h(w,w) = 0$);
    so two elements of $\hat V$ are equivalent if and only if they can be obtained from each other by right multiplication by a non-zero element of $D$.
	Let $\infty$ be the element of $X$ represented by $(1, 0, 0) \in \hat V$,
	and for every $(v,a) \in U$, we will denote the corresponding element of $X$ by $[v,a]$,
	represented by
	\[ (a, v, 1) \in \hat V . \]
    (Observe that $\hat h\bigl((a, v, 1), (a, v, 1)\bigr) = a^\sigma - a + h(v,v) = 0$.)

	For every $(z,s) \in U$, we now define a permutation $\alpha_{z,s}$ of $X$ given by
    \begin{equation}\label{eq:alpha_z,s}
	    \alpha_{z,s}(a, v, b) := \bigl( a + h(z,v) + sb,\, v+zb,\, b \bigr) ,
    \end{equation}
	and similarly a permutation $\gamma_{z,s}$ of $X$ given by
    \begin{equation}\label{eq:gamma_z,s}
	    \gamma_{z,s}(a, v, b) := \bigl( a,\, v+za,\, b - h(z,v) - sa \bigr) .
    \end{equation}
	We now set
	\[ U_\infty := \{ \alpha_{z,s} \mid (z,s) \in U \},
		\quad U_0 := \{ \gamma_{z,s} \mid (z,s) \in U \} . \]

	We now show that $U_\infty$ and $U_0$ coincide with the root groups of $\mouf(h)$.
	Fix some $(z,s) \in U$.
	Let $(v,a) \in U$ be arbitrary; then $\alpha_{z,s}$ maps $[v,a]$ to
	\[ [v,a]\alpha_{z,s} = \alpha_{z,s}(a, v, 1) = \bigl( a + h(z,v) + s, v + z, 1 \bigr) = [(v,a)+(z,s)] , \]
    where the sum $(v,a)+(z,s)$ denotes the group operation of $U$.
	(We have slightly abused notation in passing from a class to a representative.)
	Similarly,
	\begin{align*}
		[v,a]\gamma_{z,s}
		&= \gamma_{z,s}(a, v, 1)
		= \bigl( a,\, v+za,\, 1 - h(z,v) - sa \bigr) \\
		&= \bigl[ (v + za)\bigl( 1 - h(z,v) - sa \bigr)^{-1} ,\, a\bigl( 1 - h(z,v) - sa \bigr)^{-1} \bigr] .
	\end{align*}
	On the other hand, the corresponding map in the Moufang set $\mouf(h)$ is given by
	\begin{align*}
		[v,a]\gamma_{z,s}
		&= [v,a] \tau^{-1} \alpha_{z,s} \tau \\
		&= [va^{-1}, -a^{-1}] \alpha_{z,s} \tau \\
		&= \bigl[ va^{-1} + z ,\, - a^{-1} + h(z,va^{-1}) + s \bigr] \tau \\
		&= \bigl[ (va^{-1} + z) a \bigl( 1 - h(z,v) - sa \bigr)^{-1} ,\, a \bigl( 1 - h(z,v) - sa \bigr)^{-1} \bigr] ,
	\end{align*}
    and we see that both formulas coincide.
\end{proof}
\begin{remark}
    Alternatively, we could have computed the $\mu$-map corresponding to the element $(0,1) \in U$ directly from~\eqref{eq:alpha_z,s} and~\eqref{eq:gamma_z,s}.
    (Notice that $(0,1) \in U$ because $1 \in D_\sigma$.)
    By the uniqueness of the $\mu$-maps, it suffices to find elements $(z,s)$ and $(y,t)$ in $U$ such that
    the composition $\gamma_{z,s} \alpha_{0,1} \gamma_{y,t}$ swaps $0$ (i.e., $(0,0,1)$) and $\infty$ (i.e., $(1,0,0)$).
    We get $(z,s) = (y,t) = (0,1)$, so
    \[ \mu_{0,1} = \gamma_{0,1} \alpha_{0,1} \gamma_{0,1} , \]
    and we see that $\mu_{0,1}$ maps an element $(a,v,b) \in \hat V$ to $(b,v,-a)$.
    Hence
    \[ [v,a]\mu_{0,1} = (a,v,1)\mu_{0,1} = (1,v,-a) \sim (-a^{-1}, -va^{-1}, 1) = [-va^{-1}, -a^{-1}] \]
    and we recover the expected formula for $\tau$.
\end{remark}
\begin{corollary}\label{cor:mouf-classical}
    Let $\mouf$ be a Moufang set arising from a classical linear algebraic group of $k$\dash rank one over some field $k$ with $\cha(k) \neq 2$,
    and assume that the root groups of $\mouf$ are non-abelian.
    Then there exists a skew field $D$ with involution $\sigma$, a non-trivial right $D$-module $V$,
    and an anisotropic skew-hermitian form $h \colon V \times V \to D$, such that $\mouf \cong \mouf(D, \sigma, V, h)$.
\end{corollary}
\begin{proof}
    By \cite[Section 1.1]{BruhatTitsII}, every classical linear algebraic group over some field $k$ with $\cha(k) \neq 2$
    is isogenous to either $\SL_n(D)$ for some finite-dimensional skew field $D$ with center $k$,
    or to $\SU(h)$ for some finite-dimensional skew field $D$ with involution $\sigma$, some non-trivial finite\dash dimensional right $D$\dash module~$V$, and some
    sesquilinear hermitian or skew-hermitian form $h$ which is non-degenerate and trace-valued.
    The condition to be trace-valued is always satisfied when $\cha(k) \neq 2$.
    Moreover, by \cite[Remark~1.8(2)]{BruhatTitsII}, we may assume that $h$ is skew-hermitian provided $D \neq k$.
    Notice that when $D = k$, we have $\sigma = 1$ and a hermitian form is just a symmetric bilinear form and hence $\SU(h)$
    is, in fact, an orthogonal group~$\SO(h)$.

    Assume now that $\mouf$ is a Moufang set arising from a classical linear algebraic group $\GG$ of $k$\dash rank one over some field $k$ with $\cha(k) \neq 2$,
    and assume that the root groups of $\mouf$ are non-abelian.
    Since $\SL_2(D)$ and the orthogonal groups have a Moufang set with abelian root groups,
    we may assume that $\GG = \SU(h)$ for some non-degenerate skew-hermitian form $h$ of Witt index $1$.
    Let $h_{\mathrm{an}}$ be the anisotropic part of $h$.
    It now follows from Theorem~\ref{th:mouf-herm} that $\mouf(\GG) \cong \mouf(h_{\mathrm{an}})$.
    Finally, notice that the underlying vector space $V_{\mathrm{an}}$ is non-trivial because the root groups are assumed to be non-abelian.
\end{proof}

%
%
\chapter{Structurable algebras}\label{se:structalg}

Structurable algebras have been introduced by Bruce Allison in 1978 \cite{A78}.
His main motivation was to construct exceptional Lie algebras, by generalizing the Tits--Kantor--Koecher construction, which constructs Lie algebras from Jordan algebras.

\section{Definitions and basic properties}\label{se:struct1}

Structurable algebras are algebras equipped with an involution and triple product.
This triple product is in some sense of more importance than the multiplication of the algebra, in the same spirit as the operators in Jordan algebras are
more important than the actual multiplication in their classical definition.
Structurable algebras can only be defined over fields of characteristic different from $2$ and $3$.
\begin{definition}
        A {\em structurable algebra} over a field $k$ of characteristic not~$2$ or $3$ is a finite-dimensional, unital $k$\dash algebra with involution%
        \footnote{In our context, an involution will always be a $k$\dash linear map of order at most two such that $\overline{xy}=\overline{y}\,\overline{x}$.
        Notice in particular that we do not consider involutions that do not fix the ground field $k$, i.e.\@ our involutions are always ``of the first kind''.}
        $(\A,\bar{\ })$ such that
        \begin{equation}\label{struct id}
                [V_{x,y}, V_{z,w}] = V_{\{x,y,z\},w} - V_{z,\{y,x,w\}}
        \end{equation}
        for $x,y,z,w \in \A$, where the left hand side denotes the Lie bracket of the two operators, and where
        \[ V_{x,y}z := \{x \ y \ z\} := (x\overline{y})z + (z\overline{y})x - (z\overline{x})y . \]
        (This definition of $V_{x,y}z$ should be compared with identity~\eqref{eq:JordanV}.)

        For all $x,y,z \in \A$, we write $U_{x,y}z:=V_{x,z}y$ and $ U_xy:=U_{x,x}y$.
        We will refer to the maps $V_{x,y} \in \End_k(\A)$ as {\em $V$-operators}, and to the maps $U_{x,y} \in \End_k(\A)$ and $U_x \in \End_k(\A)$
        as {\em $U$-operators}.
        The trilinear map $(x,y,z) \mapsto \{ x \ y \ z \}$ is called the {\em triple product} of the structurable algebra.
\end{definition}

Structurable algebras are not necessarily associative nor commutative; they are generalizations of both associative algebras with involution and Jordan algebras.
In particular, the structurable algebras with trivial involution are exactly the Jordan algebras; see section \ref{Jordan} below.

In \cite{A78} and \cite{A79}, a structurable algebra is defined as an algebra with involution such that
\begin{align}\label{eqdef}
        [T_z,V_{x,y}]=&V_{T_z x,y}-V_{x,T_{\overline{z}}y}
\end{align}
for all $x,y,z \in \A$ with $T_x:=V_{x,1}$.
The equivalence of \eqref{struct id} and \eqref{eqdef} follows from \cite[Corollary 5.(v)]{A79}.

\begin{definition}
        Let $(\A,\bar{\ })$ be a structurable algebra; then $\A=\mathcal{H}\oplus\Ss$ for
        \[ \mathcal{H}=\{h\in \mathcal{A}\mid \overline{h}=h\} \quad \text{and} \quad \Ss=\{s\in \mathcal{A}\mid \overline{s}=-s\}. \]
        The elements of $\mathcal{H}$ are called {\em hermitian elements},
        the elements of $\Ss$ are called {\em skew-hermitian elements} or briefly {\em skew elements}.
        The dimension of~$\Ss$ is called the {\em skew-dimension} of $\A$.
\end{definition}
Skew elements tend to behave `nicer' than arbitrary elements of the structurable algebra.

As usual, the commutator and the associator are defined as
\[ [x,y]=xy-yx,\quad [x,y,z]=(xy)z-x(yz),\]
for all $x,y,z\in \A$.
For each $s \in \Ss$, we define the operator $L_s \colon \A\to\A$ by
\[ L_s x := sx . \]
The following map is of crucial importance in the study of structurable algebras:
\[\psi \colon \A\times \A\to\Ss \colon (x,y)\mapsto x\overline{y}-y\overline{x}.\]

A structurable algebra $(\A,\bar{\ })$ is {\em skew-alternative}, i.e.,
\begin{equation}\label{sxy}
    [s,x,y]=-[x,s,y]=[x,y,s]
\end{equation}
for all $s\in\Ss$ and all $x,y\in\A$; see \cite[Proposition 1]{A78}.
This implies that
\begin{align}
&[s,s,x]=[x,s,s]=[s,x,s]=0,\\
&s[t,s,x]=-[s,ts,x],\quad [x,s,t]s=-[x,st,s].\label{struct mouf id}
\end{align}
for all $s,t\in\Ss$ and all $x\in\A$.
The identities \eqref{struct mouf id} are weak versions of two of the well-known Moufang identities for alternative division rings.

We list some more identities. For all $x,y\in \A$ and $s\in \Ss$, we have
\begin{align}
U_{x,y}-U_{y,x}&=L_{\psi(x,y)},\label{Uxy-Uyx}\\
V_{x,sy}-V_{y,sx}&=-L_{\psi(x,y)}L_s,\label{Vxsy-Vysx}\\
s\psi(x,y)s&=-\psi(sx,sy),\label{spsixys}\\
L_sU_{x,y}L_s&=-U_{sx,sy}.\label{LsUxyLs}
\end{align}

Identity \eqref{Uxy-Uyx} follows from the definition of the $U$-operator;
\eqref{Vxsy-Vysx} is shown in \cite[Lemma 2]{A79};
\eqref{spsixys} is \cite[Lemma 11.2]{AH81};
\eqref{LsUxyLs} is \cite[Proposition 11.3]{AH81}.

\begin{definition}
        An {\em ideal} of $\A$ is a two-sided ideal stabilized by $\barop$.
        A structurable algebra $(\A,\barop)$ is {\em simple} if its only ideals are $\{0\}$ and $\A$,
        and it is called {\em central} if its center
        \begin{multline*}
            Z(\A,\barop)=Z(\A)\cap \mathcal{H} \\
                =\{c\in\A\mid [c,\A]=[c,\A,\A]=[\A,c,\A]=[\A,\A,c]=0\}\cap \mathcal{H}
        \end{multline*}
        is equal to $k1$.
        The {\em radical} of $\A$ is the largest solvable%
        \footnote{An ideal $I$ is solvable if there exists a $k\in \N$ such that $I^{2^k}=1$.}
        ideal of $\A$.
        A structurable algebra is {\em semisimple} if its radical is zero.
\end{definition}
If $\cha(k)\neq2,3,5$, a semisimple structurable algebra is the direct sum of simple structurable algebras
(see \cite[Section 2]{Sc} for the characteristic zero case and \cite[Section 2]{S} for the general case).
If $\cha(k)=0$, \cite[Theorem 10]{Sc} states that if $\A$ is a structurable algebra with radical $R$, there exists a semisimple structurable subalgebra $\mathcal{B}\leq \A$ such that $\A=\mathcal{B}\oplus R$.

\section{Conjugate invertibility in structurable algebras}\label{se:struct2}

For structurable algebras there is a notion of invertibility that generalizes the notion of invertibility from Jordan algebras.

\begin{definition}
        Let $(\A,\barop)$ be a structurable algebra. An element $u\in\A$ is said to be {\em conjugate invertible} if there exists an element $\hat{u}\in\A$ such that
        \begin{equation}\label{Vuhatu}
            V_{u,\hat{u}} = \id, \text{ or equivalently, } V_{\hat{u},u} = \id.
        \end{equation}
        If $u$ is conjugate invertible, then the element $\hat{u}$ is uniquely determined, and is called the {\em conjugate inverse} of $u$.
        For long expressions we will also use the notation
        \[ u^\wedge := \hat{u} = \widehat{u} . \]
        If $u$ is conjugate invertible, the operator $U_u$ is invertible and we have
        \begin{equation}\label{hatu}
            \hat{u} = U_u^{-1} u ;
        \end{equation}
        see \cite[Section 6]{AH81}.

        If each element in $\A\setminus\{0\}$ is conjugate invertible, $\A$ is called a {\em conjugate division algebra} or simply a {\em structurable division algebra}.
        Clearly, every structurable division algebra is simple.
\end{definition}
If $\A$ is an associative algebra with involution and $u\in \A$ is invertible, then $\hat{u} = \overline{u}^{-1}$; this motivates the term ``conjugate inverse''.

For skew elements, a more elegant criterion for conjugate invertibility exists:
an element $s\in \Ss$ is conjugate invertible if and only if $L_s$ is invertible.
If $s\in \Ss$ is conjugate invertible, we have
\begin{align}
        \hat{s}&=-L_s^{-1}1\in \Ss;\label{hats}\\
        L_{\hat{s}}L_s&=L_sL_{\hat{s}}=-\id;\label{LhatsLs}
\end{align}
see \cite[Proposition 11.1]{AH81}.
Note that $s\hat{s}=\hat{s}s=-1$ for all $s \in \Ss$.

The following formula (see \cite[Proposition 2.6]{A86_2}) allows to determine the conjugate inverse of any invertible element,
provided that the conjugate inverse of any invertible skew element can be determined.
Let $u\in \A$ and $s\in \Ss$ be both conjugate invertible; then $\psi(u,U_u (su))$ is conjugate invertible and
\begin{equation}\label{hatuskew}
        \widehat{u} = 2\bigl(\psi(u,U_u (su))\bigr)^\wedge U_u (su).
\end{equation}

\section{Examples of structurable algebras}\label{se:struct-ex}

Central simple structurable algebras over fields of characteristic different from $2$, $3$ and $5$ are classified; they consist of six (non-disjoint) classes:
\begin{compactenum}[(1)]
\item central simple associative algebras with involution,
\item central simple Jordan algebras,
\item structurable algebras constructed from a non-degenerate hermitian form over a central simple associative algebra with involution,
\item simple structurable algebras of skew-dimension 1; these are forms of certain algebras constructed from a pair of Jordan algebras,
\item forms of the tensor product of two composition algebras,
\item an exceptional $35$-dimensional case, which can be constructed from an octonion algebra.
\end{compactenum}
For the classes (4) and (5), we know in addition that there always exists a field extension which is at
most quadratic such that
the algebra becomes of the required shape when the scalars are extended to this extension field.
See sections~\ref{ex:skew dim} and~\ref{ex:compalg} below.

When the characteristic is zero the classification was carried out by Allison in \cite{A78}, but in that paper class (6) was overlooked.
The classification was completed and generalized to fields of characteristic different from $2$, $3$ and $5$ by Smirnov in \cite{S}.

Below we describe classes (1)--(5) and give some specific properties of each class;
in particular, we will investigate in which cases we obtain structurable division algebras.
This turns out to be a challenging question, and in fact, for the classes (4) and (5), no complete answer is known.
Since there are no division algebras in class (6), we do not describe it explicitly but refer the interested reader to \cite{S2} and \cite{AF_35dim} instead.
Although we defined structurable algebras to be finite-dimensional algebras, the examples (1), (2), (3) can of course also be defined in the infinite-dimensional case.

\subsection{Associative algebras with involution}\label{ass}

Let $(\A,\barop)$ be an associative algebra with involution.
Then $\A$ is a structurable algebra; see \cite[Example 8.i]{A78} or \cite[p. 411]{Sc}.
An element $u \in \A$ is conjugate invertible if and only if it is invertible in the usual associative sense,
and in this case $\hat{u} = \overline{u}^{-1}$.
In particular, $\A$ is a structurable division algebra if and only if it is a division algebra in the usual sense.

\subsection{Jordan algebras}\label{Jordan}

Let $\A$ be a Jordan algebra and let $\barop=\id$.
In this case the $V$-operator of structurable algebras is equal to the $V$-operator of Jordan algebras; see~\eqref{eq:JordanV}.
The defining identity of structurable algebras is a known identity in the theory of Jordan algebras; see, for example, \cite[p. 202, (FFV)']{Mc}.

Moreover, each structurable algebra with trivial involution (equivalently, of skew-dimension zero) is indeed a Jordan algebra;
see \cite[p.135, Remark (ii)]{A78}).

It is now easy to check that an element $u \in \A$ is conjugate invertible if and only if it is invertible in the Jordan algebra,
and $\hat{u} = u^{-1}$.
In particular, $\A$ is a structurable division algebra if and only if it is a Jordan division algebra.

\subsection{Hermitian structurable algebras}\label{hermitian}

This class of structurable algebras is constructed from a hermitian space.
\begin{definition}\label{def:herm}
Let $(E ,\bar{\ })$ be a unital associative algebra over a field $k$ with a $k$\dash linear involution $\barop$.
Let $W$ be a unital left $E $-module, and let $h \colon W\times W\rightarrow E $ be a hermitian form,
i.e.\@ a $k$\dash bilinear form such that $h(\alpha v, \beta w) = \alpha h(v,w) \overline{\beta}$ and $\overline{h(v,w)} = h(w,v)$ for all $v,w \in W$
and all $\alpha,\beta \in E$.

It is shown in \cite[Section 8.(iii)]{A78} that $\A := E \oplus W$ is a structurable algebra with the following involution and multiplication:
\begin{align*}
	\overline{e+w} &= \overline{e}+w, \\
	(e_1+w_1)(e_2+w_2) &= (e_1e_2+h(w_2,w_1))+(e_2 w_1+\overline{e_1} w_2),
\end{align*}
for all $e,e_1,e_2 \in E $ and all $w,w_1,w_2\in W$.
\end{definition}
It is clear that $\Ss=\{e\in E\mid \overline{e}=-e\}$. After some calculations we find that
\begin{align}
    \psi(e_1+w_1,e_2+w_2)&=(e_1\overline{e_2}-e_2\overline{e_1})-(h(w_1,w_2)-\overline{h(w_1,w_2)}), \label{psi hermitian}\\
    V_{e_1+w_1,e_2+w_2}(e_3+w_3)&=\notag\\
    &(e_1\overline{e_2}+h(w_1,w_2))e_3+(e_1\overline{e_2}+h(w_1,w_2))w_3\notag\\
    + \hspace{.5ex}&(e_3\overline{e_2}+h(w_3,w_2))e_1+(e_3\overline{e_2}+h(w_3,w_2))w_1\notag\\
    + \hspace{.5ex}&(-e_3\overline{e_1}+h(w_3,w_1))e_2+(e_3\overline{e_1}-h(w_3,w_1))w_2 ,\label{V hermitian}
\end{align}
for all $e_1,e_2,e_3\in E$, $w_1,w_2,w_3\in W$.

By \cite[Proposition 4.1]{A86}, the structurable algebra $\A = E\oplus W$ is central simple if and only if $E$ is central simple,
$\dim(E)\dim(W)\neq 1$ and $h$ is non-degenerate;
moreover, $e+w\in E\oplus W$ is conjugate invertible if and only if $e\overline{e}-h(w,w)$ is invertible in $E$, and in this case the conjugate inverse is given by
\begin{equation}\label{inverse hermitian}
    \widehat{e+w} = (e\overline{e}-h(w,w))^{-1}(e-w) .
\end{equation}
In particular, if $e\in \Ss$ is conjugate invertible, then $\hat{e}=-e^{-1}$.
\begin{lemma}
    The hermitian structurable algebra $\A = E\oplus W$ is a structurable division algebra if and only if
    $E$ is a division algebra and for all $0\neq w\in W$ we have $h(w,w)\neq 0,1$.
\end{lemma}
\begin{proof}
    This follows immediately from \eqref{inverse hermitian}, since $h(w,w) = e\overline{e}$ if and only if
    $h(e^{-1}w, e^{-1}w) = 1$ when $e \neq 0$.
\end{proof}

\subsection{Structurable algebras of skew-dimension one}\label{ex:skew dim}

Structurable algebras of skew-dimension one are, among the structurable algebras, closest to Jordan algebras,
and they have interesting connections with so-called Freudenthal triple systems.
In \cite{A90} and \cite{AF84}, this type of structurable algebras has been studied;
we will collect some of these results, which we will need later.

In this section $\A$ is a structurable algebra of skew-dimension one. We fix a non-zero element $s_0\in \Ss$, so $\Ss=ks_0$.
By \cite[Lemma 2.1]{AF84}, $s_0^2=\mu 1$ for some $\mu \in k^*$, therefore $s_0(s_0x)=(xs_0)s_0=\mu x$; it is also shown that a simple  structurable algebra of skew-dimension one is always central.

If $\A$ is simple, then the bilinear map $\psi$ is non-degenerate (see \cite[Lemma~2.2]{AF84}). Define the quartic form $\nu \colon \A\to k$ given by
\begin{align}
\nu(x)=\frac{1}{6\mu}\psi(x,U_{x}(s_0x))s_0,\label{eq:defmu}
\end{align}
by identifying $k1$ with $k$. It is easy to see that $\nu(1)=1$ and that $\nu$ is independent of the choice of $s_0$.
This map plays the role of the norm of $\A$,
in a way that can be made explicit; see \cite[Proposition~5.4]{AF92}.

By \cite[Proposition 2.11]{AF84}, an element $x\in \A$ is conjugate invertible if and only if $\nu(x)\neq0$.
When this is the case, we have
\begin{equation}\label{inv skew dim 1}
        \hat{x} = -\frac{1}{3\mu\nu(x)}s_0 U_x(s_0x).
\end{equation}
In particular, $\A$ is a structurable division algebra if and only if the quartic form $\nu$ is anisotropic.

We will now discuss an important class of structurable algebras of skew-dimension one, which we will call {\em structurable matrix algebras}. The multiplication of these algebras is similar to the multiplication of split octonions when represented by Zorn matrices.

\begin{definition}\label{def:structmatalg}
Let $J$ be a Jordan algebra over a field $k$, let $T \colon J\times J\to k$ be a symmetric bilinear form, let $\times \colon J\times J\to J$ be a symmetric bilinear map, and let $N \colon J\to k$ be a cubic form such that one of the following holds:
\begin{compactenum}[(i)]
    \item
        $J$ is a cubic Jordan algebra with a non-degenerate admissible form $N$, with basepoint $1$,
        trace form $T$, and Freudenthal cross product $\times$; see, for instance, \cite[\S 38]{KMRT}.
    \item $J$ is a Jordan algebra of a non-degenerate quadratic form $q$ with basepoint $1$, and $T$ is the linearization of $q$.
    In this case, $N$ and $\times$ are the zero maps.
    \item $J=0$, and the maps $N$, $T$ and $\times$ are the zero maps. In this case, $J$ is not unital.
\end{compactenum}
Fix a constant $\eta\in k$.
We now define the {\em structurable matrix algebra} $M(J, \eta)$ as follows.
Let
\[\mathcal{A}= \left\{ \begin{pmatrix}k_1&j_1\\j_2&k_2\end{pmatrix} \Bigm\vert k_1,k_2\in k, \, j_1, j_2 \in J\right\},\]
and define the involution and multiplication by the formulae
\[\overline{\begin{pmatrix}k_1&j_1\\j_2&k_2\end{pmatrix}}=\begin{pmatrix}k_2&j_1\\j_2&k_1\end{pmatrix},\]
\[\begin{pmatrix}
k_1&j_1\\
j_2&k_2
\end{pmatrix}
\begin{pmatrix}
k'_1&j'_1\\
j'_2&k'_2
\end{pmatrix}=
\begin{pmatrix}
k_1k'_1+\eta T(j_1,j'_2)&k_1j'_1+k'_2j_1+\eta(j_2\times j'_2)\\
k'_1j_2+k_2j'_2+j_1\times j'_1& k_2k'_2+\eta T(j_2,j'_1)\end{pmatrix},\]
for all $k_1,k_2,k'_1,k'_2\in k$, $j_1,j_2,j'_1,j'_2\in J$.
It is shown in \cite[Section 8.v]{A78} and \cite[Section 4]{AF84} that $M(J,\eta)$ is a simple structurable algebra.
\end{definition}
The following proposition explains the importance of structurable matrix algebras.
\begin{proposition}[{\cite[Prop.\@~4.5]{AF84}}]\label{prop:mat}
Let $\A$  be a structurable algebra of skew-dimension one with $s_0^2=\mu1$.
Then $\A$ is isomorphic to a structurable matrix algebra $M(J,\eta)$ if and only if $\mu$ is a square in $k$.
\end{proposition}
\begin{corollary}\label{cor:mat}
        Let $\A$  be a structurable algebra of skew-dimension one.
        Then there exists a field extension $E/k$ of degree at most $2$ such that
        $\A \otimes_k E$ is isomorphic to a structurable matrix algebra over $E$.
\end{corollary}

The algebras in classes (1) and (3) on page \pageref{se:struct-ex} have skew-dimension 1 if and only if the associative algebra with involution has skew-dimension 1.

\begin{remark}\label{rem:N=0}
	It follows by combining Proposition 4.4 and Theorem 4.11 from \cite{A90} that $\A$ is a form of a structurable matrix algebra with $N=0$
	if and only if $\A$ is isomorphic to a structurable algebra of a non-degenerate hermitian form over a $2$-dimensional composition algebra;
	in this case the conjugate degree of the algebra is equal to $2$.

	In all other cases, the conjugate degree of $\A$ is 4 and the quartic form $\nu$ defined in \eqref{eq:defmu} is the conjugate norm of $\A$.
\end{remark}

It is an open problem to determine explicitly all structurable algebras of skew-dimension one;
see also Remark~\ref{rem:skewdim1} below.
Examples of structurable algebras of skew-dimension one that are not isomorphic to structurable matrix algebras can be obtained by applying a Cayley--Dickson process
to a certain class of Jordan algebras.
We will now briefly explain this process, and we refer to~\cite{AF84} for more details.

\subsubsection*{The Cayley--Dickson process for structurable algebras}

In order to obtain a structurable algebra, we start from a Jordan algebra equipped with a Jordan norm of degree~$4$.

\begin{definition}
Let $J$ be a Jordan algebra over $k$.
A form $Q \colon J\rightarrow k$ of degree $4$ is a {\em Jordan norm of degree~$4$} if:
\begin{compactenum}[(i)]
\item $1\in J$ is a basepoint of $Q$, i.e. $Q(1)=1$;
\item $Q(U_j j')=Q(j)^2Q(j')$ for all $j,j'\in J\otimes_k K$ for all field extensions $K/k$;
\item The trace form
\[ T \colon J\times J\to k\colon (j,j')\mapsto Q(1;j)Q(1;j')-Q(1;j,j') \]
is a $k$\dash bilinear non-degenerate form.
\end{compactenum}
\end{definition}

The main examples of Jordan algebras with a Jordan norm of degree~$4$ are separable Jordan algebras of degree~$4$ with their generic norm and separable Jordan algebras of degree~$2$ with the square of their generic norm.
In \cite[Proposition 5.1]{A90} Jordan norms of degree~$4$ are classified.

If $J$ is a separable Jordan algebra of degree~$4$, it is shown in \cite[Theorem 5.4]{AF84} that the space $J_0=\{x\in J\mid T(x,1)=0\}$ can be given the structure of a separable Jordan algebra of degree~$3$ and thus of a cubic Jordan algebra.

\begin{definition}\label{def:CD}
 Let $J$ be a Jordan algebra with $Q$ a Jordan norm of degree~$4$ with trace $T$.
Consider the $k$\dash linear bijection $\theta$ on $J$ given by
\[ b^\theta=-b+\tfrac{1}{2}T(b,1)1,  \]
for all $b\in J$; observe that $\theta^2 = 1$.

Let $\mu \in k^*$, and define the algebra $\CD(J, Q,\mu) := J\oplus s_0 J$,
with multiplication and involution given by
\begin{align*}
    (j_1 + s_0j_2)(j_3 + s_0j_4) &= j_1j_3 + \mu{(j_2j_4^{\theta})}^{\theta} + s_0 \bigl( j_1^\theta j_4 + (j_2^\theta j_3^\theta)^\theta \bigr), \\
	\overline{j_1+s_0j_2} &= j_1-s_0j_2^\theta,
\end{align*}
for all $j_1,j_2,j_3,j_4\in J$. By \cite[Theorem 6.6]{AF84}, this is a simple structurable algebra with skew-dimension one and $\Ss=ks_0$.
\end{definition}
Since $(ts_0)^2=t^2\mu$, it follows from Proposition \ref{prop:mat} that the structurable algebra $\CD(J,Q,\mu)$ is isomorphic to a structurable matrix algebra if and only if $\mu$ is a square in $k$.

This procedure can be used to construct structurable division algebras of skew-dimension one.
\begin{lemma}[{\cite[Theorem 7.1]{AF84}}]\label{div skew dim one}Let $J$ be a Jordan division $k$\dash algebra with Jordan norm of degree $4$.
        Define the field $E=k(\xi)$ with $\xi$ transcendental over $k$, let $J'=J\otimes_k E$ and let $Q'$ be the extension of $Q$ to $J'$.
        Then $\CD(J',Q',\xi)$ is a central simple structurable division algebra over $k$.
\end{lemma}
\begin{remark}\label{rem:skewdim1}
        It is a major open problem whether every structurable division algebra of skew-dimension one is either a hermitian structurable algebra or
        is obtained from a Cayley--Dickson process on a Jordan division algebra with a Jordan norm of degree $4$.
        To the best of our knowledge, this is not known, and we have been told by Skip Garibaldi
        that there seem to be strong indications that {\em not} all structurable division algebras of skew-dimension one are of this form.
        See also section~\ref{ss:skewdim1} for the relevance of this question to our results.
\end{remark}

\subsection{Forms of the tensor product of two composition algebras}\label{ex:compalg}

For the first part of this section we allow the characteristic of $k$ to be equal to $3$, but we still demand that it is different from $2$.

\begin{definition}
        Let $C_1$ and $C_2$ be two composition algebras over $k$ (possibly of different dimension) with involution $\sigma_1$ and $\sigma_2$, respectively.
        Let $m_1:=\dim_k(C_1)$ and $m_2:=\dim_k(C_2)$.
        Let $\A=C_1\otimes_k C_2$, equipped with the involution
        \[ \barop=\sigma:=\si_1\otimes \si_2. \]
        Then we call $(\A, \barop)$ an {\em $(m_1,m_2)$-product algebra}.
\end{definition}
If $\cha(k)\neq 2,3$, it is shown in \cite[Section 8.(iv)]{A78} that $(\A,\barop)$ is a structurable algebra.

\begin{definition}
        A $k$\dash algebra $\A$ is a {\em form of the tensor product of two composition algebras} if there exists a field extension $E/k$ and two composition algebras $C_1,C_2$ over $E$
        such that $\A \otimes_k E\cong C_1 \otimes_E C_2$.
\end{definition}

These algebras were first studied in \cite{A88} by Allison, under the assumption that the base field has characteristic zero.
As pointed out in \cite{AF92}, most results remain valid in general characteristic different from $2$ (although {\em loc.\@ cit.\@} still excludes characteristic $3$),
and in particular, there is a complete classification of the forms of the tensor products of two composition algebras;
see Proposition~\ref{pr:formtensor} below.
We first give the definition which is required to describe the structure of these forms.

\begin{definition}[{\cite[p. 671]{A88}}]
        Let $C$ be a composition algebra $C$ of dimension $m$ over a quadratic extension $E/k$.
        Let $\Gal(E/k)=\langle \si\rangle$ and define the composition algebra $^\si C$ over $E$ that has the same addition, multiplication and involution as $C$,
        but with scalar multiplication $t\cdot x:=\si(t)x$ for all $t\in E$ and $x\in \prescript{\si}{}C$.

        Let $\mathcal{B}$ be the $E$-algebra $\mathcal{B}=C\otimes_E \prescript{\si}{}C$, equipped with the usual involution $\barop := \barop \otimes \barop$.
        Now define an automorphism $\tau \colon \mathcal{B}\to\mathcal{B} \colon x\otimes y\mapsto y\otimes x$,
        and let $\A=\{x\in \mathcal{B}\mid \tau(x)=x\}$,
        equipped with the restriction of $\barop$ from~$\mathcal{B}$.
        Then $(\A, \barop)$ is a form of $(\mathcal{B}, \barop)$, and is called a {\em twisted $(m,m)$-product algebra}, denoted by $c_{E/k}(C,\barop)$.
        By \cite[p. 671]{A88}, such an algebra is indeed structurable (in characteristic different from $2$ and $3$).
\end{definition}

By \cite[Proposition 2.2]{A88}, the $(m_1,m_2)$-product algebras and the twisted $(m,m)$-product algebras are central simple, except for $(m_1,m_2)=(2,2)$ and $(m,m)=(2,2)$.
Each (twisted) $(m_1,m_2)$-product algebra for $m_1,m_2\leq4$ is associative and each $(8,1)$- or $(8,2)$-product algebra is alternative.
Notice that an $(8,2)$-product is essentially an octonion algebra over a quadratic extension $E/k$, but equipped with an involution that is not $E$-linear.

\begin{proposition}[{\cite[Th. 2.1]{A88} and \cite[Proposition 7.9]{AF92}}]\label{pr:formtensor}
	If a $k$\dash algebra $\A$ is a form of the tensor product of two composition algebras,
	then either $\A$ is isomorphic to the tensor product of two composition algebras over~$k$,
	or $\A$ is isomorphic to a twisted $(m,m)$-product algebra for some $m>1$,
	and in this case, there is a quadratic extension $E/k$ such that $\A \otimes_k E$ is isomorphic to
	the tensor product of two composition algebras over $E$.
\end{proposition}

Let $C_1$ and $C_2$ be two composition algebras with involution $\si_1$ and $\si_2$ and norm form $q_1$ and $q_2$, respectively.
Let $S_i$ be the set of skew elements in $C_i$, i.e.
\[ S_i = \{ x \in C_i \mid x^{\si_i} = -x \}. \]

Let $C_1 \otimes_k C_2$ be a $(m_1,m_2)$-product algebra equipped with the involution  $\barop=\si:=\sigma_1 \otimes \sigma_2 $.
It follows that the set of skew elements in $C_1 \otimes_k C_2$ is equal to
\[ \Ss = \{ x \in C_1 \otimes_k C_2 \mid \overline{x}:=x^\si = -x \} = (S_1 \otimes 1) \oplus (1 \otimes S_2); \]
observe that $\dim_k \Ss = \dim_k C_1 + \dim_k C_2 - 2$.

\begin{definition}\label{def:invsharp}
\begin{compactenum}[(i)]
    \item
	We will associate a quadratic form $q_A$ to $C_1 \otimes_k C_2$, called the {\em Albert form}, by setting
	\[ q_A \colon \Ss \to k \colon (s_1 \otimes 1) + (1 \otimes s_2) \mapsto q_1(s_1) - q_2(s_2) \]
	for all $s_1 \in S_1$ and $s_2 \in S_2$. We have $q_A\perp \Hh=q_1\perp (-1)q_2$
	and when we denote $q'_i=q_i|_{S_i}$ for the pure part of the Pfister form $q_i$, we have $q_A=q'_1\perp \langle-1\rangle q'_2$.
    \item
	For each $s=s_1\otimes 1+1\otimes s_2\in S$, we define
	\[(s_1\otimes 1+1\otimes s_2)^\natural=s_1\otimes 1-1\otimes s_2.\]
    \item\label{it:s-inv}
	Let $s\in \Ss$ such that $q_A(s)\neq 0$. Then we say that $s$ is {\em invertible} and define the {\em inverse} of $s$ by
	\[s^{-1}:=-\frac{1}{q_A(s)}s^\natural.\]
\end{compactenum}
\end{definition}
The Albert form was introduced by A.~A.~Albert in the case where $C_1$ and $C_2$ are both quaternion algebras.

Tensor products of two composition algebras are far from being associative or alternative, but (as usual) the skew elements behave more nicely than arbitrary elements:

\begin{lemma}\label{ident bioctonion}
	For all $x\in C_1\otimes_kC_2$ and all $s_1,s_2,s\in \Ss$, we have
	\begin{compactenum}[\rm (i)]
	\item $s_1(s_2s_1)=(s_1s_2)s_1$;
	\item $(s_1s_2s_1)x=s_1(s_2(s_1x))$;
	\item If $s$ is invertible then $s(s^{-1}x)=s^{-1}(sx)=x$.
	\end{compactenum}
\end{lemma}
\begin{proof}
	These identities can be easily verified using the well-known properties of composition algebras.
\end{proof}

In the case that  both $C_1$ and $C_2$ are quaternion algebras, $C_1\otimes_k C_2$ is associative; Albert proved that $C_1\otimes_k C_2$ is a division algebra if and only if its Albert form is anisotropic (see \cite[Theorem  III.4.8]{L}.)
It is not obvious to generalize this result to arbitrary composition algebras.
The proof of the following theorem makes use of the Lie algebra associated to the structurable algebra constructed in section \ref{se:constr lie alg}.
\begin{theorem}[{\cite[Theorem 5.1]{A86} }]
   Let $\cha(k)=0$ and $\A$ a (twisted) $(8,m_2)$-product algebra. Then $\A$ is a structurable division algebra if and only if the Albert form on $\Ss$ is anisotropic.
\end{theorem}

\begin{remark}\label{rem:product-div}
    It seems plausible that this result can be generalized to arbitrary fields of characteristic different from $2$ and $3$,
    but the method used in Allison's proof cannot be generalized.
    One implication is easy: if $\A$ is a structurable division algebra, then in particular all elements of $\Ss$ are conjugate invertible,
    and Lemma~\ref{ident bioctonion} together with Definition~\ref{def:invsharp}\eqref{it:s-inv} implies that $q_A$ is anisotropic.

    For $(8,m_2)$-product algebras with $m_2 \in \{ 1,2,4 \}$, the converse does indeed hold; this can be shown by explicitly constructing
    an element $s_0 = 1 \otimes s_2 \in \Ss \setminus \{ 0 \}$ such that $q_A(\psi(x, xs_0)) \neq 0$ and invoking \cite[Theorem 8.7]{AF92}.
    The question remains open for $(8,8)$-product algebras and for twisted $(8,8)$-product algebras.
\end{remark}

\begin{remark}
\begin{compactenum}[(i)]
    \item
	In \cite{AF92}, it is shown that the conjugate degree (i.e.\@ the degree of the conjugate norm) of a (twisted) $(8,m)$-product algebra is equal to $2$, $4$, $4$ or $8$
	for $m=1,2,4,8$, respectively.
    \item
	In an entirely different context, the authors of \cite{HT} study Albert forms of the (twisted) tensor product of two octonion algebras.
	In \cite[Theorem 2.1]{HT} it is mentioned that every anisotropic 14\dash dimensional quadratic form with trivial discriminant and trivial Clifford invariant
	is the Albert form of a twisted tensor product of two octonion algebras.
	They also show the existence of quadratic forms that are the Albert form of a twisted tensor product of two octonion algebras,
	but that are not similar to the Albert form of a non-twisted tensor product of two octonion algebras;
	these are certain quadratic forms over fields that do not satisfy a condition called ``$D(14)$'' in \cite{HT}.
\end{compactenum}
\end{remark}

\section{Construction of Lie algebras from structurable algebras}\label{se:constr lie alg}

In this section, we will recall the Tits--Kantor--Koecher construction of the Lie algebra $K(\A)$ associated to any structurable algebra $\A$.
This Lie algebra will play a crucial role to make the connection with linear algebraic groups (see section~\ref{se:SSA and SAG} below).

In order to describe the construction, we need to introduce some more concepts; we give a brief summary of \cite{A79}.
We continue to assume that $\A$ is a structurable algebra with involution $\sigma \colon \A \to \A \colon x \mapsto \overline{x}$,
and we let $\End(\A)$ be the ring of $k$\dash linear maps from $\A$ to $\A$.
For each $A\in \End(\A)$, we define new $k$\dash linear maps
\begin{align*}
A^\eps&=A-L_{A(1)+\overline{A(1)}},\\
A^\delta&=A+R_{\overline{A(1)}},
\end{align*}
where $L_x$ and $R_x$ denote left and right multiplication by an element $x \in \A$, respectively.
One can verify that
\begin{align}
    V_{x,y}^\eps&=-V_{y,x},\label{Veps}\\
    V_{x,y}^\delta(s)&=-\psi(x,sy)\label{Vdel},
\end{align}
for all $x,y\in \A$ and $s\in \Ss$.
Define the Lie subalgebra $\Strl(\A,\barop)$ of $\End(\A)$ as
\begin{equation}\label{defstrl}
    \Strl(\A,\barop)=\{A\in \End(\A)\mid [A, V_{x,y}]=V_{Ax,y}+V_{x,A^\eps y}\}.
\end{equation}
(This definition follows from \cite[Corollary 5]{A78}.)
It follows from the definition of structurable algebras that $V_{x,y}\in \Strl(\A,\barop)$, so we can define the Lie subalgebra
\[\Instrl(\A, \barop)=\Span \{V_{x,y}\mid x,y\in \A\},\]
which is, in fact, an ideal of $\Strl(\A,\barop)$.
Notice that for all $s,t\in \Ss$, we have $L_sL_t\in \Instrl(\A)$, since it follows from \eqref{sxy} that
\begin{align}\label{eq:LsLt}
L_sL_t=\half( V_{st,1}-V_{s,t})=\half(V_{1,ts}-V_{s,t}).
\end{align}
It follows from skew-alternativity and the definition of $\delta$ and $\eps$ that
\begin{align}
    (L_rL_t)^\eps&=-L_tL_r,\label{LsLteps}\\
    (L_rL_t)^\delta(s)&=s(tr)+r(ts)\label{LsLtdel}
\end{align}
for all $r,t,s\in\Ss$.
For all $A\in \Strl(\A,\barop)$ we have a version of triality for endomorphisms:
\begin{align}\label{triality}
A(sx)=A^\delta(s)x+sA^\eps(x) \quad \text{for all } x\in\A, s\in \Ss.
\end{align}
By \cite[Lemma 1]{A79}, we have for all $A\in\Strl(\A,\barop)$ and $x,y\in\A$ that
\begin{align}\label{deltapsi}
A^\delta\psi(x,y)=\psi(A x,y)+\psi(x,Ay).
\end{align}
For all $A\in \Strl(\A,\barop)$, the map $A\mapsto A^\eps$ is a Lie algebra automorphism of $\Strl(\A,\barop)$ of order 2; the map $A\mapsto A^\delta|_\Ss$ is a Lie algebra homomorphism from $\Strl(\A,\barop)$ into $\End_k(\Ss)$.

It follows that $\A\oplus\Ss$ is a $\Strl(\A,\barop)$-module under the action $A(x,s)=(Ax,A^\delta s)$ for all $A\in \Strl(\A,\barop)$ and $(x,s)\in \A\oplus \Ss$.

\begin{definition}\label{def:Lie alg}
Consider two copies $\A_+$ and $\A_-$ of $\A$ with corresponding isomorphisms $\A \to \A_+ \colon x \mapsto x_+$
and $\A \to \A_- \colon x \mapsto x_-$, and let $\Ss_+\subset A_+$ and $\Ss_-\subset A_-$ be the corresponding subspaces of skew elements.
Let
\[ K(\A)=\Ss_-\oplus \A_-\oplus \Instrl(\A) \oplus \A_+ \oplus \Ss_+ \]
as a vector space; as in \cite[\S 3]{A79}, we make $K(\A)$ into a Lie algebra by extending the Lie algebra structure of $\Instrl(\A)$ as follows:
\begin{alignat*}{2}
\intertext{$\bt\ [\Instrl,K(\A)]$}
[V_{a,b},V_{a',b'}]&=V_{\{a,b,a'\},b'}-V_{a',\{b,a,b'\}} &\in \Instrl(\A),&\\
[V_{a,b},x_+] &:= (V_{a,b}x)_+ \in \A_+ , \quad & [V_{a,b},y_-] &:= (V_{a,b}^\eps y)_- \\*
&&&\ =(-V_{b,a}y)_-\in \A_-,\\
[V_{a,b},s_+] &:= (V_{a,b}^\delta s)_+ & [V_{a,b},t_-] &:= (V_{a,b}^{\eps\delta} t)_- \\*
&\ = -\psi(a,sb)_+\in \Ss_+,&&\ =\psi(b,ta)_-\in \Ss_-,\\
\intertext{$\bt\ [\Ss_\pm,\A_\pm]$}
[s_+,x_+] &:= 0, & [t_-,y_-] &:= 0,\\
[s_+,y_-] &:= (sy)_+\in \A_+,& [t_-,x_+] &:= (tx)_-\in \A_-,\\
\intertext{$\bt\  [\A_\pm,\A_\pm]$}
[x_+,y_-] &:= V_{x,y}\in \Instrl(\A),\\
[x_+,x_+'] &:= \psi(x,x')_+\in \Ss_+,&[y_-,y_-'] &:= \psi(y,y')_-\in \Ss_-,\\
\intertext{$\bt \ [\Ss_\pm,\Ss_\pm]$}
[s_+,s_+']&:=0,&[t_-,t_-']&:=0,\\
[s_+,t_-]&:=L_{s}L_{t}\in \Instrl(\A),
\end{alignat*}
for all $x,x',y,y'\in \A$, all $s,s',t,t'\in \Ss$, and all $V_{a,b},V_{a',b'}\in \Instrl(\A)$.
\end{definition}

From the definition of the Lie bracket we clearly see that the Lie algebra $K(\A)$ has a $5$\dash grading given by $K(\A)_j=0$ for all $|j|>2$ and
\begin{multline*}
	K(\A)_{-2}=\Ss_-,\quad K(\A)_{-1}=\A_-,\quad K(\A)_{0}=\Instrl(\A),\\
		K(\A)_{1}=\A_+,\quad K(\A)_{2}=\Ss_+.
\end{multline*}
In the case where $\A$ is a Jordan algebra, we have $\Ss=0$, and thus the Lie algebra $K(\A)$ has a $3$-grading;
in this case $K(\A)$ is exactly the Tits--Kantor--Koecher construction of a Lie algebra from a Jordan algebra (see, for example, \cite[Section~VIII.5]{J_Jordan}).

It is shown in \cite[\S 5]{A79} that the structurable algebra $\A$ is simple if and only if $K(\A)$ is a simple Lie algebra,
and that $\A$ is central if and only if $K(\A)$ is central.
The following strong result motivates the construction of structurable algebras.
\begin{theorem}[{\cite[Theorem 4 and 10]{A79}}]\label{th:kantorkoecher}
Let $\mathcal{L}$ be a simple Lie algebra over a field $k$ of characteristic different from $2$, $3$ and $5$. We define the following condition on $\LL$:
\begin{compactitem}[\rm ($\star$)] \item
	$\LL$ contains an $sl_2$-triple%
	\footnote{A triple $\{ e,f,h \}$ is called an $sl_2$-triple if $[e,f]=h,\ [h,e]=2e,\  [h,f]=-2f$; such a triple spans an $sl_2$ Lie subalgebra.}
	$\{e,f,h\}$ such that $\LL$ is the direct sum of an arbitrary number of irreducible $sl_2$-modules over this triple of highest weight $0$, $2$ or $4$.
\end{compactitem}
Then $\LL\cong K(\A)$ for some simple structurable algebra $\A$ if and only if $\LL$ satisfies~\textup{($\star$)}.

When $\cha(k)=0$, the condition \textup{($\star$)} on $\LL$ is fulfilled if and only if $\LL$ is isotropic,
i.e.\@ contains a non-trivial split toral subalgebra (see {\cite{Seligman}}).
\end{theorem}

The following result is particularly relevant for our purposes.
\begin{theorem}[{\cite[Theorem 3.1]{A86}}]\label{th:divalg and liealg}
Let $\cha(k)=0$, and let $\A$ be a central simple structurable algebra. Then $K(\A)$ has relative rank 1 if and only if $\A$ is a structurable division algebra.

Moreover, each central simple Lie algebra of relative rank 1 can be obtained in this way.
\end{theorem}
One of the goals of the current paper is to prove a similar result for linear algebraic groups (over fields of arbitrary characteristic different from $2$ or $3$).
We refer to Theorems~\ref{thm:AGadjoint} and~\ref{thm:AG1} below.

We now give a brief overview of the various types of simple Lie algebras that can be obtained starting from central simple structurable algebras.
When $\cha(k)=0$, we mention, in view of the previous theorem, the Tits index (see Table~\ref{ta:exc} on page~\pageref{ta:exc}) when the structurable algebra is division.
Since our main goal is to understand the exceptional cases, we do not go into detail about the classical types that are obtained.

\begin{compactenum}[(1)]
    \item
	If $\A$ is an associative algebra, then the corresponding Lie algebra is classical.
    \item
	Let $\A$ be a central simple Jordan algebra.
	Then $K(\A)$ is classical unless the Jordan algebra is an exceptional Jordan algebra (i.e.\@ an Albert algebra);
	in this case, it is of type $E_7$.
	If $J$ is an Albert division algebra, then the corresponding Lie algebra has Tits index $E_{7,1}^{78}$;
	see also Example~\ref{ex:abmoufsets}\eqref{it:Albert}.
    \item
	If $\A$ is a structurable algebra arising from a hermitian form, then the corresponding Lie algebra is classical.
    \item\label{KKstruct dim 1}
	Let $\A$ be of skew-dimension $1$, thus $\A$ is a form of the algebra $M(J,\eta)$ for several possibilities of the Jordan algebra $J$
	as in Definition \ref{def:structmatalg}.
	If $N=0$, then $\A$ is isomorphic to a structurable algebra of hermitian form type by Remark~\ref{rem:N=0}, and hence $K(\A)$ is classical.
	So assume that $N\neq 0$, and $J$ is a Jordan algebra of a non-degenerate cubic norm as in Definition \ref{def:structmatalg}(i).
	It follows from \cite[Theorem V.4,V.8 and V.9]{J_Jordan}
	that $\dim_k J \in \{ 1,3,6,9,15,27 \}$.
	We get the following types of Lie algebras (at least in characteristic $0$):
	\[ \begin{array}{c|c|c}
	\dim_k(J)& \dim_k(\A)& \text{type of } K(\A)\\
	\hline
	1  & 4 & G_2\\
	3  & 8 & D_4\\
	6  & 14 & F_4\\
	9  & 20 & E_6\\
	15 & 32 & E_7\\
	27 & 56 & E_8
	\end{array}\]
	The known examples of structurable division algebras of this type arise from the Cayley--Dickson process described in Definition~\ref{def:CD},
	starting from a central simple Jordan division algebra of degree~$4$.
	Such an algebra has dimension $10$, $16$ or $28$, and the respective Tits indices of the Lie algebra $K(\A)$ are $^2\!E^{35}_{6,1}$, $E^{66}_{7,1}$
	and $E^{133}_{8,1}$, respectively (see  \cite[Example 7.2]{AF84}), again assuming that $\cha(k) = 0$.
     \item
	When $C_1,C_2$ are composition algebras, the Lie algebra $K(C_1\otimes_k C_2)$ coincides with the Lie algebra constructed from $C_1$ and $C_2$
	using Tits' second Lie algebra construction, which is visualized in Freudenthal's magic square (see \cite[p. 672]{A88}).

	In the case where $\A$ is a non-associative (twisted) product algebra, we have
	\[\begin{array}{c|c}
	\text{(twisted) } (i,j)\text{-product algebra}& \text{type of } K(\A)\\
	\hline
	(8,1)& F_4\\
	(8,2)& E_6\\
	(8,4)& E_7\\
	(8,8)& E_8
	\end{array}\]
	If $\cha(k)=0$, it is shown in \cite[Theorem 6.22]{A88} that if the algebras in the above list are structurable division algebras,
	then $K(\A)$ has Tits index $F_{4,1}^{21},$ $^2\!E^{29}_{6,1},$ $ E^{48}_{7,1}$ or $E^{91}_{8,1}$, respectively.
    \item
	The exceptional 35-dimensional structurable algebras give rise to split Lie algebras of type $E_7$; see \cite{AF_35dim}.
\end{compactenum}

\section{Isotopies of structurable algebras}\label{se:struct5}

Although isomorphisms of structurable algebras are well defined, it turns out that it is better
to allow the unit element $1$ to be mapped to a different element.
This idea is encapsulated in the notion of an isotopy.
\begin{definition}[{\cite[Section 8]{AH81}}]\label{def:isotopy}
Two structurable algebras $(\A,\bar{\ })$ and $(\A',\bar{\ })$ over a field $k$ are {\em isotopic}
if there exists a $k$\dash vector space isomorphism $\psi \colon \A\rightarrow \A'$ such that there exists a $\chi\in \Hom_k(\A,\A')$ such that
\[\psi(V_{x,y}z)=V_{\psi(x),\chi(y)}\psi(z)\quad\forall x,y,z\in \A.\]
The map $\psi$ is then called an {\em isotopy} between $(\A,\bar{\ })$ and $(\A',\bar{\ })$.
The map $\chi$ is completely determined by the map $\psi$; we call $\chi$ the {\em inverse dual} of $\psi$ and denote it by $\chi:=\hat{\psi}$.

In this case, the map $\chi$ is again an isotopy, with inverse dual $\psi$; in particular, $\doublehat{\psi}=\psi$.
Isotopy defines an equivalence relation on structurable algebras.

If $\psi$ maps the identity of $\A$ to the identity of $\A'$, then $\psi$ is an isomorphism of structurable algebras.

By \cite[Lemma 1.20]{AF84}, if $\A$ and $\A'$ are isotopic, then $\A$ is (central) simple if and only if $\A'$ is (central) simple.
\end{definition}

The following theorem indicates why isotopy is a useful equivalence relation on structurable algebras.
\begin{theorem}[{\cite[Proposition 12.3]{AH81}}]\label{th:isotopic_isomor}
	Two structurable algebras $\A$ and $\A'$ are isotopic if and only if $K(\A)$ and $K(\A')$ are graded-isomorphic graded Lie algebras.
\end{theorem}

We collect some useful facts about isotopies.
From \cite[Proposition~11.3]{AH81}, it follows that
\begin{align}
L_s\{x,y,z\}&=\{L_s x, L_{\hat{s}} y, L_s z\}\label{Lsxyz}
\end{align}
for all $x,y,z\in\A$ and $s\in \Ss$ conjugate invertible. Therefore $L_s$ is an isotopy with $\widehat{L_s}=L_{\hat{s}}$.
Let $u\in\A$ be conjugate invertible and let $\alpha$ be an isotopy; then by \cite[Proposition 8.2]{AH81},
\begin{align}\label{hatalphau}
\widehat{\alpha u}=\hat{\al} \hat{u},
\end{align}
and in particular if $s\in \Ss$ is conjugate invertible, then
\begin{align}
\widehat{su}=\hat{s}\hat{u}. \label{hatsx}
\end{align}

If $(\A',\bar{\ })$ is isotopic to $(\A,\bar{\ })$, there exists a conjugate invertible $u\in \A$ such that $(\A',\bar{\ })$ is isomorphic to a certain isotope of $(\A,\bar{\ })$, denoted by $(\A,\bar{\ })^{\langle u\rangle}$, which we describe below.
\begin{construction}\label{constr:isot}
Let $u\in \A$ be a conjugate invertible element. We give the definition of the $u$-conjugate isotope $\A^{\langle u\rangle}$ of $\A$ following the approach in \cite[p.\@ 364]{A86_2}; see \cite[Section 7]{AH81} for the original definition.

The algebra $\A^{\langle u\rangle}$ will be a structurable algebra with underlying vector space $\A$.
Its involution is defined by
\begin{align}
\tau^{\langle u\rangle} x=\bar{x}^{\langle u\rangle}=2x-\{x,u,\hat{u}\}=x-\psi(x,\hat{u})u.\label{tauu}
\end{align}
We have ${\tau^{\langle u\rangle}}^2=\id$, and we define $\Ss^{\langle u\rangle}$ and $\mathcal{H}^{\langle u\rangle}$ as the $(-1)$- and $1$\dash eigenspace, respectively, for $\tau^{\langle u\rangle}$. Then
\[\A=\Ss^{\langle u\rangle} \oplus \mathcal{H}^{\langle u\rangle};\]
moreover, one can show that $\Ss^{\langle u\rangle}=\Ss u$.

Next, we define the operator $P_u$ given by
\begin{align}\label{defPu}
P_u x=\tfrac{1}{3}U_u(2\tau^{\langle u\rangle}+\id)x=\frac{1}{3}U_u(5  x-2V_{x,u}\hat{u}).
\end{align}
This operator is invertible and has the following nice properties:
\begin{align}
P_uP_{\hat{u}}&=P_{\hat{u}}P_u=\id\label{PuPhatu},\\
P_u\hat{u}&=u\label{Puhatu},\\
P_u(su)&=-\frac{1}{3}U_u(su)\quad \text{ for all } s\in \Ss\label{Pusu},\\
P_u\{x,y,z\}&=\{P_u x, P_{\hat{u}} y, P_u z\}\quad \text{ for all } x,y,z\in \A\label{Puxyz}.
\end{align}
This last identity says that $P_u$ is an isotopy on $\A$ with $\widehat{P_u}=P_{\hat{u}}$.

Finally, if $x,y\in \A$ we can write $x=su+a$ where $s\in \Ss$ and $a\in \mathcal{H}^{\langle u\rangle}$, and we define
\begin{align}
x_{\langle u\rangle}y=(su+a)_{\langle u\rangle}y=sP_u y+ V_{a,u}y.\label{mult isotop}
\end{align}

This defines a product on the vector space $\A$.
Then
\begin{equation}\label{1u}
	1^{\langle u\rangle}=\hat{u}
\end{equation}
is a unit for the product and $\tau^{\langle u\rangle}$ is an involution for this product.
We denote the algebra with this product and involution $\tau^{\langle u\rangle}$ by $\A^{\langle u\rangle}$; this algebra is again structurable.
\end{construction}

The $V$-operator of the algebra $\A^{\langle u\rangle}$ is given by
\begin{align}
V^{\langle u\rangle}_{x,y}z=\{x,y,z\}^{\langle u\rangle}=\{x,P_u y,z\}, \label{xyzu}
\end{align}
for all $x,y,z\in\A$. We denote by $L^{\langle u\rangle}_x$ the left multiplication with $x$ in $\A^{\langle u\rangle}$ and by  $\psi^{\langle u\rangle}(x,y):=L^{\langle u\rangle}_x \overline{y}^{\langle u\rangle}-L^{\langle u\rangle}_y \overline{x}^{\langle u\rangle}$.
Notice that
\begin{align}
\psi^{\langle u\rangle}(x,y)&=\psi(x,y)u\label{psiuxy},\\
L_{su}^{\langle u\rangle}&=L_sP_u, \label{Lusu}
\end{align}
for all $x,y\in \A$ and all $s\in \Ss$.

Let $x$ be conjugate invertible in $\A$. It follows from the identity $V^{\langle u\rangle}_{x,P_{\hat{u}} \widehat{x}}=V_{x,\widehat{x}}=\id$ that $x$ is also conjugate invertible in $\A^{\langle u\rangle}$, with
\begin{align}\label{isotopeinv}
\widehat{x}^{\langle u\rangle}=P_{\hat{u}}\widehat{x}
\end{align}
where $\widehat{\phantom{x}}^{\langle u\rangle}$ denotes the conjugate inverse in $\A^{\langle u\rangle}$.

We give an overview of some interesting facts about isotopies for the different classes of central simple algebras.
\begin{compactenum}[(1)]
    \item
	Let $\A$ be an associative algebra with involution.
	It is shown in \cite{A86_2} that two associative algebras are isotopic if and only if they are isomorphic. If $u\in \A$ is invertible, $P_u=L_{u\overline{u}}$.
    \item
	Let $\A$ be a Jordan algebra.
	It follows from \eqref{tauu} and \eqref{mult isotop} that two Jordan algebras are isotopic as Jordan algebras (see \cite[Section 12]{J_Jordan})
	if and only if they are isotopic as structurable algebras.
	If $u\in \A$ is invertible, $P_u=U_{u}$.
    \setcounter{enumi}{3}
    \item
	Let $\A$ be of skew-dimension $1$.
	In \cite{A90}, various interesting properties of the isotope $(\A,\barop)^{\langle u \rangle}$ for $u\in k1\oplus k s_0$ have been obtained.
    \item\label{pg:isotopic}
	If $\cha(k)=0$, \cite[Theorem 5.4]{A88} states that two structurable algebras that are forms of an $(m_1,m_2)$-product algebra are isotopic
	if and only if their respective Albert forms are similar.

	In \cite{A86_2} it is shown that two alternative algebras are isotopic if and only if they are isomorphic;
	in this case, $P_u=L_{u\overline{u}}$ for every invertible element $u\in \A$.
    \item
	Two exceptional $35$-dimensional structurable algebras are always isotopic; see \cite{AF_35dim}.
\end{compactenum}


\chapter{One-invertibility for structurable algebras}\label{se:one-inv}

As we will see, the description of the Moufang sets arising from structurable division algebras will rely on
the notion of {\em one-invertibility}, introduced by Bruce Allison and John Faulkner in the context of Kantor pairs \cite{AF99},
generalizing the identically named concept for Jordan pairs.
They explicitly proved a criterion for one-invertibility (Theorem~\ref{th:1inv});
this criterion allows us to explicitly compute the left and right inverse of arbitrary non-zero elements in $\A\times\Ss$
when $\A$ is a division algebra (Theorem~\ref{th:formula-oneinverse} and Corollary~\ref{cor:oneinverse}).

In order to include the case $\Char(k) = 5$, we have to deal with the rather subtle notion of {\em algebraicity} for structurable algebras
and their corresponding $5$\dash graded Lie algebras.
This is what we do in section~\ref{se:alg-gr}; see also Remark~\ref{rem:loos} below.
However, we will see later that every structurable {\em division} algebra is automatically algebraic (see Theorem~\ref{thm:div-alg}), and
this will allow us to state our results in section~\ref{se:MS-SDA} for arbitrary structurable division algebras over fields $k$ with $\Char(k) \neq 2,3$,
without having to be concerned about algebraicity any longer.

In section~\ref{se:oneinverse} we then define one-invertibility for elements in $\A\times \Ss$,
and in section~\ref{se:divtooneinv} we show that if $\A$ is an (algebraic) structurable division algebra,
then each element in $(x,s)\in \A\times\Ss\setminus\{(0,0)\}$ is one-invertible.
We are indebted to John Faulkner for providing the main idea for the proof of this fact.

The results from this section will be used later in section~\ref{se:MS-SDA} when we describe the construction
of Moufang sets from structurable division algebras and determine the group $U$ and permutation $\tau$;
see Theorem \ref{mainth:moufset} for the main result.


\section{Algebraicity of $5$-graded Lie algebras}\label{se:alg-gr}

In this section, we introduce the notion of an {\em algebraic} $5$\dash graded Lie algebra;
we will then define an {\em algebraic} structurable algebra as an algebra such that the associated graded Lie algebra $K(\A)$ is algebraic.
It turns out that a simple structurable algebra is associated to an algebraic $k$\dash group only
if it satisfies the algebraicity condition. This concept, as well as other constructions employed in the present
section, go back to~\cite[Chapter 2]{St-thes}, which, actually, deals with a more general case of $(2n+1)$-graded Lie algebras over
commutative rings, and due to that is much more technical.

\begin{notation}
Let $\LL$ be a finite-dimensional $5$\dash graded Lie algebra over a field $k$, $\Char k\neq 2,3$.
For any commutative $k$\dash algebra $R$ and any
$(x,s)\in (\LL\otimes_k R)_\sigma\oplus (\LL\otimes_k R)_{2\sigma}=(\LL_\sigma\oplus \LL_{2\sigma})\otimes_k R$
we set
\[
e_\si(x,s)=\sum\limits_{i=0}^4\frac 1{i!}\ad(x+s)^i\in\End_R(\LL\otimes_k R).
\]
\end{notation}

\begin{definition}
Let $\LL$ be a finite-dimensional $5$\dash graded Lie algebra over $k$. The \emph{grading derivation}
on $\LL$ is the derivation $\zeta\in\Der_k(\LL)$ such that for any $-2\le i\le 2$ and any $x\in\LL_i$
one has
\[
\zeta(x)=i\cdot x\quad\mbox{for any $-2\le i\le 2$ and any $x\in\LL_i$}.
\]
If $\LL$ contains an element $\zeta$ such that $\ad_\zeta$ is the grading derivation, we call $\zeta$
a grading derivation of $\LL$ by abuse of language.
\end{definition}

\begin{lemma}\label{lem:esigma'-prod}
Let $\LL$ be a finite-dimensional $5$\dash graded Lie algebra over a field~$k$, $\Char k\neq 2,3$.
\begin{compactenum}[\rm (i)]
\item\label{esigma':prod}
Let $R$ be a commutative associative unital $k$\dash algebra. For all $x,y\in \LL_\si\otimes_k R$ and $s,t\in \LL_{2\si}\otimes_k R$
we have
\begin{align}\label{eq:esi-operations}
    e_\si(x,s)\,e_\si (y,t) &= e_\si(x+y,s+t+\frac 12 [x,y])\quad\text{and} \\
    e_\si(x,s)^{-1} &= e_\si(-x,-s).
\end{align}
\item Let $\GU_\sigma$ denote $\LL_\sigma\oplus\LL_{2\sigma}$ considered as an algebraic $k$\dash group
with respect to the operation
\[
(x,s)*(y,t)=(x+y,s+t+\frac 12[x,y]).
\]
Then
\[
e_\si \colon \GU_\sigma\to \GL(\LL),\quad (x,s)\mapsto e_\si(x,s),
\]
is a
homomorphism of algebraic $k$\dash groups. If
there is a grading derivation $\zeta\in\LL$, then this homomorphism is a closed embedding.
\end{compactenum}
\end{lemma}
\begin{proof}
\begin{compactenum}[(i)]
    \item The statement follows immediately from a version of the Campbell--Baker--Hausdorff formula
	 over an arbitrary commutative ring in which $6$ is invertible \cite[Theorem 9]{AF99}, and the fact that
    $\LL$ is $5$\dash graded.
    \item Clearly, $e_\sigma \colon \GU_\sigma\to\End(\LL)$ is a morphism of $k$\dash varieties. By~\eqref{esigma':prod}
we have $e_\sigma(\GU_\sigma)\subseteq\GL(\LL)$ and $e_\sigma$ is
homomorphism of $k$\dash groups. If $\zeta$ exists, then $e_\sigma(x,s)(\zeta)=\zeta-\sigma x-2\sigma s$,
and hence $e_\sigma$ is universally injective. Then by~\cite[II, \S 5, Proposition 5.1]{DeGa} $e_\sigma$ is a
closed embedding.
\qedhere
\end{compactenum}
\end{proof}

\begin{definition}\label{def:algebraic}
Let $\LL$ be a finite-dimensional $5$\dash graded Lie algebra over a field $k$, $\Char k\neq 2,3$.
We say that $\LL$ is \emph{algebraic},
if for any $(x,s)\in \LL_\si\oplus\LL_{2\sigma}$
the endomorphism $e_\si(x,s)$ of $\LL$ is a Lie algebra automorphism. We say that a
(finite-dimensional) structurable algebra $\A$ over $k$ is \emph{algebraic}, if $K(\A)$ is algebraic
in the above sense.
\end{definition}

\begin{remark}\label{rem:algebraic}
Lemma~\ref{lem:esigma'-alg} below shows that any Jordan algebra over a field of characteristic $\neq 2,3$
is algebraic, and any structurable algebra over a field of characteristic $\neq 2,3,5$ is algebraic.
We will be able to prove later that every structurable {\em division} algebra is algebraic,
but this will require much more preparation; see Theorem~\ref{thm:div-alg} below.

In fact, we do not know whether there exist central simple structurable algebras that are not algebraic.
\end{remark}

\begin{lemma}\label{lem:ext-alg}
Let $\LL$ be a finite-dimensional $5$\dash graded Lie algebra over a field $k$, $\Char k\neq 2,3$.
If $\LL$ is algebraic, then
for any commutative associative unital $k$\dash algebra $R$ and any
$(x,s)\in (\LL\otimes_k R)_\sigma\oplus (\LL\otimes_k R)_{2\sigma}=(\LL_\sigma\oplus \LL_{2\sigma})\otimes_k R$
the endomorphism
\[
e_\si(x,s)=\sum\limits_{i=0}^4\frac 1{i!}\ad(x+s)^i
\]
of $\LL\otimes_k R$ is an $R$-Lie algebra automorphism.
\end{lemma}
\begin{proof}
Take any $a\in (\LL_\sigma\oplus \LL_{2\sigma})\otimes_k R$.
Since $e_\si(a)$ is an $R$-linear endomorphism of $\LL\otimes_k R$, in order to establish the claim it is
enough to prove that
\begin{equation}\label{eq:esi-aut}
e_\si(a)([b,c])=[e_\si(a)(b),e_\si(a)(c)]
\end{equation}
for any $b\in\LL_i$, $c\in\LL_j$, $-2\le i,j\le 2$. The formulas~\eqref{eq:esi-operations} of Lemma~\ref{lem:esigma'-prod}
imply
that, moreover, it is enough to prove~\eqref{eq:esi-aut} only for elements $a=\lambda a_0$, where
$a_0\in\LL_\sigma$ or $a_0\in\LL_{2\si}$ and $0\neq\lambda\in R$.

Assume first $R$ is a field extension of $k$. Then $\lambda$ is invertible. Consider the Lie algebra automorphism
$\phi_\lambda$ of $\LL\otimes_k R$ defined by the formula $\phi_\lambda(x)=\lambda^n x$ for any
$x\in\LL_n\otimes_k R$,
$-2\le n\le 2$. Assume that $a_0\in\LL_1$. Then, clearly,
$$
\phi_\lambda^{-1}\circ e_\si(a)\circ\phi_\lambda=e_\si(\phi_\lambda^{-1}(a))=e_\si(a_0).
$$
Therefore,
\begin{multline*}
e_\si(a)([b,c])=\phi_\lambda\circ e_\si(a_0)\circ\phi_\lambda^{-1}([b,c])=\\
\phi_\lambda\bigl([e_\si(a_0)(\lambda^{-i}b),e_\si(a_0)(\lambda^{-j}c)]\bigr)=
[\phi_\lambda e_\si(a_0)\phi_\lambda^{-1}(b),\phi_\lambda e_\si(a_0)\phi_\lambda^{-1}(c)]=\\
[e_\si(a)(b),e_\si(a)(c)],
\end{multline*}
which settles the case $a_0\in\LL_1$. If $a_0\in\LL_{-1}$, one should run the same argument for $\lambda^{-1}$ instead of $\lambda$.
If $a_0\in\LL_{2\si}$, then one can run the same argument with the element $\sqrt\lambda$ of the
quadratic field extension $R\sqrt\lambda$ of $R$. Since the map $\LL\otimes_k R\to\LL\otimes_k R\sqrt\lambda$
is injective, the fact that~\eqref{eq:esi-aut} holds in $\LL\otimes_k R\sqrt\lambda$ implies that
it holds in $\LL\otimes_k R$.

Next, assume that $R$ is a domain, and let $K$ be its fraction field. By the previous
case the equality~\eqref{eq:esi-aut} holds in $\LL\otimes_k K$. Since the natural map $\LL\otimes_k R\to \LL\otimes_k K$
is injective, we conclude that~\eqref{eq:esi-aut} holds in $\LL\otimes_k R$.

Now let $R$ be arbitrary. Since $a,b,c$ are finite $R$-linear combinations of elements in $\LL$,
we can assume without loss of generality that $R$ is a finitely generated $k$\dash algebra. Then
$R\cong k[x_1,\ldots,x_n]/I$, where $k[x_1,\ldots,x_n]$ is the polynomial ring in $n$ variables over $k$.
Let $\tilde a$, $\tilde b$ and $\tilde c$ be any preimages of $a,b,c$ contained in the
corresponding graded summands of $\LL\otimes_k k[x_1,\ldots,x_n]$.
Since $k[x_1,\ldots,x_n]$ is a domain, the equality~\eqref{eq:esi-aut} holds for $\tilde a,\tilde b,\tilde c$.
Then it holds for $a,b,c$, since the natural projection $\LL\otimes_k k[x_1,\ldots,x_n]\to\LL\otimes_k R$
preserves the Lie bracket.
\end{proof}

\begin{lemma}\label{lem:esigma'-alg}
Let $\LL$ be a finite-dimensional $5$\dash graded Lie algebra over a field~$k$, $\Char k\neq 2,3$.
If, moreover, $\Char k\neq 5$ or $\LL_2\oplus\LL_{-2}=\{0\}$, then $\LL$ is algebraic.
\end{lemma}
\begin{proof}
Since by Lemma~\ref{lem:esigma'-prod} we have $e_\si(x)e_\si(s)=e_\si(x,s)$, it is enough to prove the claim
for all endomorphisms $e_\si(a)$, where $a\in\LL_i$, $i\neq 0$.
If $\LL_2\oplus\LL_{-2}=\{0\}$, we can consider $\LL$ as a $3$-graded Lie algebra. Thus,
     we are given a $(2n+1)$-graded Lie algebra $\LL$ over $k$ such that $(3n)!\in k^\times$, where $n=1$ or $2$.
For any derivation $D\in\End(\LL)$, any
$b,c\in \LL$, and any $3n\geq m\geq 0$ it follows by induction on $m$ that
\begin{equation}\label{eq:Dm}
\frac 1{m!}D^m([b,c])=\sum\limits_{p+q=m}\left[\frac 1{p!}D^p(b),\frac 1{q!}D^q(c)\right].
\end{equation}
For any $a\in \LL_i$, $b\in \LL_j$, $c\in \LL_l$, where $-n\leq i,j,l\leq n$ and $i\neq 0$, all homogeneous graded components of $[e_\si(a)(b),e_\si(a)(c)]$
are of the form
\begin{equation}\label{eq:exp-homog}
\sum\limits_{\substack{p+q=m \\[.6ex] p,q\leq 2n}}
\left[\frac 1{p!}\ad_a^p(b),\frac 1{q!}\ad_a^q(c)\right]\in\LL_{im+j+l}.
\end{equation}
If $3n\geq m\geq 0$, such a component is equal to the corresponding component of $e_\si(a)[b,c]$ by~\eqref{eq:Dm}.
If $4n\geq m>3n$, then $|im|>3n$ while $|j+l|\leq 2n$, which implies that $\LL_{im+j+l}=0$.
\end{proof}

\begin{lemma}\label{lem:univ-alg}
Let $\LL$, $\LL'$ be two finite-dimensional $5$\dash graded Lie algebras over a field $k$, $\Char k\neq 2,3$.
Let $f \colon \LL\to\LL'$ be a graded $k$\dash homomorphism of Lie algebras, such that $f|_{\LL_i} \colon \LL_i\to\LL'_i$
is a bijection for all $i\in\{\pm 1,\pm 2\}$. If $\LL$ is algebraic, then $\LL'$ is algebraic.
\end{lemma}
\begin{proof}
We use the same idea as in the proof of Lemma~\ref{lem:esigma'-alg}.
In order to show that $\LL'$ is algebraic, it is enough to show that
\begin{equation}\label{eq:exp-tilde}
e_\sigma(a)([b,c])=[e_\sigma(a)(b),e_\sigma(a)(c)]
\end{equation}
for any $a\in \LL'_i=f(\LL_i)$ with $i\in\{\pm 1,\pm 2\}$, any
$b\in\LL'_j$ with $-2\leq j\leq 2$, and any $c\in \LL'_l$ with $l=0$.
The equality~\eqref{eq:Dm} now holds for $4\geq m\geq 0$,
since $\Char k\neq 2,3$ (but may be equal to $5$) and hence $4!$ is invertible.
It remains to note that if $8\geq m> 4$, then
the degree $m$ homogeneous component~\eqref{eq:exp-homog} of $[e_\sigma(a)(b),e_\sigma(a)(c)]$
is trivial, since $l=0$. This implies the equality~\eqref{eq:exp-tilde}.
\end{proof}


\section{One-invertibility in $\A\times \Ss$}\label{se:oneinverse}

In \cite{AF99} the notion of $n$-invertibility for Kantor pairs is introduced.
Kantor pairs are generalizations of Jordan pairs; an example of a Kantor pair is a pair of structurable algebras,
in the same way as every Jordan algebra gives rise to a Jordan pair.

Since we will only apply the results of \cite{AF99} in the context of a pair of structurable algebras,
 we only explain the necessary terminology and results of \cite{AF99} in this context. This makes the
exposition less technical, and in Remark~\ref{rem:AF99} we explain why our point of view is the same
as in \cite{AF99}.

\begin{quote}
        {\em Throughout section~\ref{se:oneinverse}, we will assume that $k$ is a field of characteristic
     different from $2$ and $3$.}
\end{quote}

Let $\A$ be an arbitrary structurable $k$\dash algebra. In section \ref{se:constr lie alg}, we described the $5$\dash graded
Lie algebra $K(\A)$ constructed from $\A$.
In \cite{AF99}, a slightly different (but isomorphic) Lie algebra is used.
We will adopt this Lie algebra; this will make the formulas for one-invertibility more elegant.
\begin{definition}\label{def:KacA}
Let $\A$ be a structurable algebra,
and let $K(\A)$ be the Lie algebra as introduced in Definition~\ref{def:Lie alg} above.
\begin{compactenum}[(i)]
    \item\label{it:KK'}
        Consider two copies $\A_+$ and $\A_-$ of $\A$ with corresponding isomorphisms $\A \to \A_+ \colon x \mapsto x_+$
        and $\A \to \A_- \colon x \mapsto x_-$, and let $\Ss_+\subset A_+$ and $\Ss_-\subset A_-$ be the corresponding subspaces of skew elements.
        We define a new Lie algebra $K'(\A)$ with the same underlying vector space and the same grading as $K(\A)$, i.e.
        \[K'(\A)=\Ss_-\oplus \A_-\oplus \Instrl(\A) \oplus \A_+ \oplus \Ss_+,\]
        and we keep the Lie bracket of $K(\A)$ except that we modify the formulas for $[\A_\pm,\A_\pm]$ by a factor $-2$ as follows:
        \begin{align*}
            [x_+,y_-] &:= -2 V_{x,y}\in \Instrl(\A),\\
            [x_+,x_+'] &:= -2\psi(x,x')_+\in \Ss_+, \\
            [y_-,y_-'] &:= -2\psi(y,y')_-\in \Ss_-,
        \end{align*}
        for all $x,x',y,y'\in \A$.

        It is straightforward to verify that the following map from $K(\A)$ to $K'(\A)$ is a Lie algebra isomorphism:
        \[ \begin{cases}
             \Instrl(\A) \to \Instrl(\A) :& V_{a,b}  \mapsto V_{a,b},\\
             \A_+ \to \A_{+} :&x \mapsto x,\\
             \A_{-} \to \A_{-} :&y \mapsto-\tfrac{1}{2} y,\\
             \Ss_+ \to \Ss_{+} :&s \mapsto -2s,\\
             \Ss_{-} \to \Ss_{-} :&t \mapsto -\tfrac{1}{2}t.
        \end{cases} \]
    \item\label{KacA:frakg}
	Let $\zeta\in\End_k(K'(\A))$ be the grading derivation of $K'(\A)$.
        In what follows, it is convenient to consider $\zeta$ as an element of the Lie algebra;
        we thus define
	\begin{equation}\label{eq:frakg}
        	\G:=\Ss_-\oplus \A_-\oplus (\Instrl(\A)+k\zeta) \oplus \A_+ \oplus \Ss_+
	\end{equation}
        with the same Lie bracket as $K'(\A)$ and with $[\zeta,x_i]=\zeta(x_i)=ix_i$ for all $x_i\in K'(A)_i$ with $i\in [-2,2]$.
        It follows that also $\G$ has a $5$\dash grading with
        \begin{align*}
                \G_0 &= K'(\A)_0+k\zeta=\Instrl(\A)+k\zeta, \\
                \G_{\pm1} &= K'(\A)_{\pm1}=\A_{\pm}, \\
                \G_{\pm2} &= K'(\A)_{\pm2}=\Ss_{\pm}.
        \end{align*}
        Observe that $K'(\A)$ is an ideal of $\G$ and that $[\G,\G]=K'(\A)$.
\end{compactenum}
\end{definition}

The Lie algebra $\G$ defined above is the same algebra as in \cite{AF99} in case the Kantor pair in \cite{AF99} is a pair of structurable algebras,
as we now briefly explain.

\begin{definition}[\cite{AF99}]\label{def:Kantor}
A \emph{Kantor pair} is a pair of vector spaces $(K_+,K_-)$ over $k$
equipped with a trilinear product
\[
\lK\cdot,\cdot,\cdot\rK \colon K_\si\times K_{-\si}\times K_\si\to K_\si,\quad \si\in\{-1,1\},
\]
satisfying the following two identities:
        \begin{compactitem}
           \item[(KP1)]  $[\VK_{x,y}, \VK_{z,w}] = \VK_{\lK x,y,z\rK,w} - \VK_{z,\lK y,x,w \rK}$;
           \item[(KP2)]  $\KK_{a,b}\VK_{x,y}+\VK_{y,x}\KK_{a,b}=\KK_{\KK_{a,b}x,y}$;
        \end{compactitem}
where $\VK_{x,y}z:=\lK xyz\rK$ and $\KK_{a,b}z:=\lK azb \rK -\lK bza \rK$.
\end{definition}
There is a tight connection between Kantor pairs and Lie triple systems in the sense of~\cite{J_Jordan}.
In~\cite{AF99}, a $\ZZ$-graded Lie triple system $\TT=\bigoplus\limits_{i\in\ZZ}\TT_i$ is called \emph{sign-graded}, if
$\TT_i=0$ for all $i\neq\pm 1$.

\begin{theorem}[{\cite[Theorem 7]{AF99}}]\label{thm:Kan-triple}
Let $(K_+,K_-)$ be a pair of vector spaces over $k$ equipped
with a trilinear product
\[
\lK\cdot,\cdot,\cdot\rK \colon K_\si\times K_{-\si}\times K_\si\to K_\si,\quad \si\in\{-1,1\}.
\]
Then $(K_+,K_-)$ is a Kantor pair if and only if
$\TT=K_+\oplus K_-$ is a sign-graded Lie triple system with two
non-zero graded components $\TT_1=K_+$ and $\TT_{-1}=K_-$ and the
triple product $[\,,\,,\,]\colon\TT\times\TT\times\TT\to\TT$ defined as follows
\begin{equation}\label{eq:tripleprod}
\begin{aligned}
& [x_\sigma,y_\sigma,z_\sigma]=0;\\
& [x_\sigma,y_{-\sigma},z_\sigma]=-\lK x_\sigma,y_{-\sigma},z_\sigma \rK;\\
& [x_{-\sigma},y_{\sigma},z_\sigma]=\lK y_\sigma,x_{-\sigma},z_\sigma \rK;\\
& [x_{\sigma},y_{\sigma},z_{-\sigma}]=\lK y_\sigma,z_{-\sigma},x_\sigma \rK-\lK x_\sigma,z_{-\sigma},y_\sigma \rK.
\end{aligned}
\end{equation}
\end{theorem}

Let $\TT=\TT_1\oplus\TT_{-1}$ be a sign-graded Lie triple system over $k$.
In \cite[p.\@ 532]{AF99} the $5$\dash graded Lie algebra
$\G(\TT)=\bigoplus\limits_{i=-2}^2\G(\TT)_i$ is defined,
which is called the standard graded embedding of $\TT$.
Recall that $\G(\TT)$ is the Lie subalgebra of the Lie algebra
\[
\Der(\TT)\oplus \TT,
\]
where $\Der(\TT)$ is the Lie algebra of $k$\dash derivations of the Lie triple system $\TT$, with the following graded
components:
\[
\G(\TT)_\sigma=\TT_\sigma,
\]
\[
\G(\TT)_0=k\delta+[\TT_1,\TT_{-1},-]\subseteq \Der(\TT_1\oplus \TT_{-1}),
\]
where $\delta$ is the grading derivation of $\G(\TT)$, and
\[
\G(\TT)_{2\sigma}=[\TT_\sigma,\TT_\sigma,-]\subseteq \Der(\TT_1\oplus \TT_{-1}).
\]
If $\TT=K_+\oplus K_-$ is the Lie triple system corresponding to a Kantor pair $(K_+,K_-)$,
the $5$\dash graded Lie algebra $\G(K_+\oplus K_-)$ is also called the standard graded embedding of this Kantor pair.

\begin{remark}\label{rem:AF99}
Let $\A$ be a structurable algebra over $k$ and let $\A_+,\  \A_-$ be two isomorphic copies of $\A$;
then the pair $(\A_+,\A_-)$ with the triple product
\[
\lK x,y,z \rK:=2V_{x,y}z=2\{x,y,z\}
\]
for $x,z\in \A_\si$ and $y\in \A_{-\si}$ is a Kantor pair, and the standard graded embedding $\G(\A_+\oplus\A_-)$
coincides with the Lie algebra $\G$ of~\eqref{eq:frakg}.
To see this, we identify $L_s\in \G(\A_+\oplus \A_-)_{\pm2}$ with $s\in \G_{\pm2}$ and represent the elements of
$\G(\A_+\oplus \A_-)_0$ with their action on $\G(\A_+\oplus \A_-)_1$.
Then the algebras $\G(\A_+\oplus \A_-)$ and $\G$ are identical;
this can be verified using $K_{a,b}=2L_{\psi(a,b)}$ and the identity (KP2) of Definition~\ref{def:Kantor}.
The Lie algebra $\G(\A_+\oplus \A_-)$ is described more explicitly in \cite[p.\@ 535]{AF99}.
\end{remark}

We will now define some subgroups of $\End_k(\G)$. In \cite{AF99} the action of $\End_k(\G)$ on $\G$ is denoted on the left, whereas we need an action on the right in order to be compatible with the conventions in the theory of Moufang sets.
This is why some formulas differ slightly from \cite{AF99}.
\begin{definition}
Let $\si\in \{-1,+1\}$, $x\in \G_\si, s\in \G_{2\si}$; we define
\[e_\si(x,s)=\text{exp}(\ad(x+s))=\sum_{i=0}^4\frac{1}{i!} (\ad(x+s))^i\in \End_k(\G).\]
Define the set
\[U_\si=\{e_\si(x,s)\mid x\in \G_\si, s\in \G_{2\si}\}.\]
\end{definition}

\begin{remark}\label{rem:loos}
Throughout~\cite{AF99}, one uses the fact that the elements $e_\si(x,s)$ are Lie algebra automorphisms
of $\G$, i.e. belong to $\Aut_k(\G)$. However, the proof of this statement in~\cite[Theorem 8]{AF99}
requires the (missing) assumption \mbox{$\cha(k)\neq 5$}. We thank Ottmar Loos for bringing this to our attention.
In place of the assumption $\cha(k)\neq 5$, in what follows we impose the weaker condition that $\A$ is {\em algebraic}
(see Definition~\ref{def:algebraic}).
\end{remark}

For future reference, we reproduce here the basic properties of $e_\si(x,s)\in\End_k(\G)$
claimed in~\cite[Theorem 8]{AF99} together with a complete proof based on our results from section~\ref{se:alg-gr}.

\begin{lemma}[{\cite[Theorem 8]{AF99}}]\label{le:esigma}
Let $\A$ be an algebraic structurable algebra over $k$.
Let $\si\in \{-1,+1\}$, then we have the following properties for all $x,y\in \G_\si$ and $s,t\in \G_{2\si}$:
\begin{compactenum}[\rm (i)]
\item $e_\si (x,s)$ is an automorphism of the Lie algebra $\G$,\\i.e. $[a,b]. e_\si(x,s)=[a. e_\si(x,s), b. e_\si(x,s)]$ for all $a,b\in \G$.
\item $e_\si(x,s)e_\si (y,t)=e_\si(x+y,s+t+ \psi(x,y))$,
\item $e_\si(x,s)^{-1}=e_\si(-x,-s)$.
\item The map $e_\si \colon \G_\si\times \G_{2\si}\to U_\si \colon (x,s)\mapsto e_\si(x,s)$ is a bijection.
\end{compactenum}
\end{lemma}
\begin{proof}
Since $\A$ is algebraic, the $5$\dash graded Lie algebra $\G$ is algebraic
by Lemma~\ref{lem:univ-alg} applied to the natural homomorphism of $5$\dash graded Lie algebras
\[ K(\A)\cong K'(\A)\to\G.  \]
This implies (i). All other statements follow from Lemma~\ref{lem:esigma'-prod} applied to the $5$\dash graded Lie algebra $\G$;
in particular, (iv) follows from Lemma~\ref{lem:esigma'-prod}(ii).
\end{proof}

\begin{quote}
        {\em From now on until the end of section~\ref{se:oneinverse}, we will assume that $\A$
is an algebraic structurable algebra over $k$.}
\end{quote}

\begin{definition}\label{def:G}
\begin{compactenum}[(i)]
    \item
        The {\em elementary group} of the algebraic structurable algebra $\A$ is defined as\[G:=\langle U_+,U_-\rangle \leq \Aut(\G).\]
    \item
        As in \cite{AF99}, we define $H_+$ as the subgroup of all automorphisms in $G$ that preserve the gradation of $\G$,
        and $H_-$ as the subset of all automorphisms in $G$ that reverse the gradation of $\G$, i.e.
        \begin{align*}
            H_+ &:= \{h\in G\mid \zeta.h = \zeta\}\\
            &\ = \{h\in G\mid \G_i.h=\G_{i} \text{ for all } i\in \{-2,-1,0,1,2\}\}, \\
            H_- &:= \{h\in G\mid \zeta.h = -\zeta\}\\
            &\ = \{h\in G\mid \G_i.h=\G_{-i} \text{ for all } i\in \{-2,-1,0,1,2\}\}.
        \end{align*}
    \item
        Define $\varphi\in \End_k(\G)$, which reverses the gradation of $\G$, as follows:
          \begin{alignat*}{2}
         \G_0&\to \G_0  \colon & \ V_{a,b} &\mapsto -V_{b,a},\\
         &&\zeta&\mapsto -\zeta,\\
         \A_1&\to \A_{-1}  \colon &x&\mapsto x,\\
         \A_{-1}&\to \A_{1}  \colon &x&\mapsto x,\\
         \A_2&\to \A_{-2}  \colon &s&\mapsto s,\\
         \A_{-2}&\to \A_{2}  \colon &s&\mapsto s.
        \end{alignat*}
        We will observe in Lemma~\ref{toevoeging}(ii) that $\varphi$ is an automorphism of $\G$,
        and in Lemma \ref{le:e-alpha} below, we will show that it is, in fact, an element of $H_-$.
\end{compactenum}
\end{definition}
In Theorem \ref{th:Gisrank1group} we will show that if all elements in $\A\times \Ss$ are one-invertible, then the group $G=\langle U_+, U_-\rangle$ is an abstract rank one group.

\begin{lemma}\label{toevoeging} Let $\si\in\{-1,+1\}$ and $(x,s)\in \G_\si\times \G_{2\si}$.
\begin{compactenum}[\rm (i)]
    \item For all $h\in H_-$, we have $e_\si(x,s)^h=e_{-\si}(x.h,s.h)$ and ${U_\si}^h=U_{-\si}$.
    \item We have $\varphi\in \Aut(\G)$, $e_\si(x,s)^\varphi=e_{-\si}(x,s)$ and ${H_-}^\varphi=H_-$.
\end{compactenum}
\end{lemma}
\begin{proof}
\begin{compactenum}[(i)]
    \item
        Let $h\in H_-$.
        Then
        \begin{equation}\label{a.ad}
            a.\ad(x+s)^h=a.(h^{-1} \ad(x+s) h)=a.\ad(x.h+s.h)
        \end{equation}
        for all $a\in \G$, since $h$ is an automorphism of $\G$.
        Since $h$ reverses the grading, we have $(x.h,s.h)\in \G_{-\sigma}\times \G_{-2\sigma }$, and hence
        it follows from~\eqref{a.ad} that $e_\si(x,s)^h=e_{-\si}(x.h,s.h)$.
        Since $h$ is an isomorphism, we conclude that ${U_\si}^h=U_{-\si}$.
    \item
        It follows from the definition of the Lie bracket of $\G$ that $\varphi\in \Aut_k(\G)$,
        using \eqref{eq:LsLt} to show that $[s,t].\varphi=[s.\varphi, t.\varphi]$ for $s\in \G_{2}$ and $t\in \G_{-2}$.
        Next, since $\ad(x+s)^\varphi=\ad(x.\varphi+s.\varphi)$, we have $e_\si(x,s)^\varphi=e_{-\si}(x,s)$.
        To show the final statement, notice that $\G_i.\varphi^{-1} h\varphi=\G_{-i}.h\varphi=\G_{i}.\varphi=\G_{-i}$.
        Since $G^\varphi=G$, we conclude that $H_-^\varphi=H_-$.
    \qedhere
\end{compactenum}
\end{proof}
It is not clear a priori that $\varphi$ is actually contained in $H_-$.
In Lemma~\ref{le:e-alpha} below, however, we will show that this is indeed the case.

\begin{notation}
Let $(x,s)\in \A\times\Ss$. As $\A_+$ and $\A_-$ are copies of $\A$ and $\Ss_-$ and $\Ss_+$ copies of $\Ss$, we can write $e_\si(x,s)$  without causing any confusion:
If we write $e_+(x,s)$ we consider $(x,s)$ as an element of $\A_+\times \Ss_+$, whereas if we write $e_-(x,s)$ we consider $(x,s)$ as an element of $\A_-\times \Ss_-$.
\end{notation}

We have now collected enough background information to define the notion of one-invertibility.
\begin{definition}\label{def:one inv}
\begin{compactenum}[(i)]
    \item
        Let $(x,s)\in \A\times \Ss$.
        We say that $(x,s)$ is {\em one-invertible} if there exist $(y,t),(z,r)\in \A\times\Ss$ such that
        \[e_{-}(z,r)e_+(x,s)e_{-}(y,t)\in H_-.\]

        Using Lemma \ref{toevoeging}(ii), we see that this condition is equivalent with
        \[e_{+}(z,r)e_-(x,s)e_{+}(y,t)\in H_-.\]
        \item If $(x,s)$ is one-invertible, we say that $(x,s)$ has {\em left
        \footnote{Recall that the action of $G$ is on the right.}
        inverse} $(y,t)$ and {\em right inverse} $(z,r)$.
        By \cite[Lemma 11]{AF99}, the left and right inverses are unique.
    \item
        Let $(x,s)\in \A\times\Ss$ be one-invertible with right inverse $(z,r)$.
        Define the linear map $P_{(x,s)} \colon \A\to\A$ given by
        \[P_{(x,s)}a := U_x \Bigl( a+\frac{2}{3}\psi(z,a)x \Bigr) + s \bigl( a+2\psi(z,a)x \bigr) \quad \text{for all } a\in \A.\]
\end{compactenum}
\end{definition}

\begin{remark}
        In \cite[Section 5]{AF99}, $n$-invertibility for an $n$-tuple in
        \[ (\G_\si\times \G_{2\si})\times(\G_{-\si}\times \G_{-2\si})\times\dots\times (\G_{(-1)^{n-1}\si}\times \G_{(-1)^{n-1}2\si}) \]
        is defined in a similar way, but we will not need this more general notion.
\end{remark}

 The following theorem gives us a very useful characterization of one-invertibility.
\begin{theorem}[{\cite[Theorem 13]{AF99}}]\label{th:1inv}
\begin{compactenum}[\rm (i)]
    \item
        An element $(x,s)\in \A\times\Ss$ is one-invertible if and only if there exists $(u,t)\in \A\times\Ss$ such that
        \begin{gather}\label{inverse}
        \begin{aligned}
        V_{x,u}&=\id+L_sL_t,\\
        su&=-\frac{1}{3}U_x (tx),\\
        \psi(x,s(tx))&=0.
        \end{aligned}
        \end{gather}
        This system of equations has either no solutions or exactly one solution.
    \item\label{it:l-r-inv}
        Let $(x,s)\in \A\times \Ss$ be one-invertible with $(u,t)$ the solution of the system of equations \eqref{inverse}.
        Then the left inverse of $(x,s)$ is $(u-tx,t)$ and its right inverse is $(u+tx,t)$.
    \item
        Let $(x,s)\in \A\times \Ss$ be one-invertible with $(u,t)$ the solution of the system of equations \eqref{inverse}.
        For each $\si\in \{-1,1\}$, let
        \[h_\si := e_{-\si}(u+tx,t) \, e_\si(x,s) \, e_{-\si}(u-tx,t)\in H_-.\]
        Then $h_\si|_{\G_\si}=P_{(u-tx,t)}=P_{(u+tx,t)}$ and $h_\si|_{\G_{-\si}}=P_{(x,s)}$.
\end{compactenum}
\end{theorem}
\begin{proof} We transfer \cite[Theorem 13]{AF99} to our setup using Remark \ref{rem:AF99}.

 It is shown in the proof in \cite{AF99} that $(x,s)\in \G_\si\times\G_{2\si}$ is one-invertible if and only if there exist $(u,t)\in  \G_{-\si}\times\G_{-2\si}$ such that
\begin{gather}\label{inverse lie bracket}
\begin{aligned}
2[s,t]+[x,u]+2\si\zeta&=0,\\
[s,u]-\frac{1}{6}[x,[x,[x,t]]]&=0,\\
-\frac{1}{3}[x,[x,[s,t]]]&=0,
\end{aligned}
\end{gather}
and that in this case the left inverse of $(x,s)$ is $(u-tx,t)$ and the right inverse is $(u+tx,t)$.

The last two equations of \eqref{inverse lie bracket} immediately give the last two equations of \eqref{inverse},
but the first equation needs some more explanation.
Recall that the identification between $\G(\LL(\A))$ and $\G$ in Remark \ref{rem:AF99} was made in such a way that the elements of $\G(\LL(\A))_0$
are identified with their action on $\A_+$.
So let $w\in \A_+$; if $\si=+1$, we get
\begin{equation}\label{eq:ad+}
    \ad(2[s,t]+[x,u]+2\si\zeta)(w)=2L_sL_tw-2V_{x,u}w+2w=0,
\end{equation}
and if $\si=-1$, we get
\begin{equation}\label{eq:ad-}
    \ad(2[s,t]+[x,u]+2\si\zeta)(w)=-2L_tL_sw+2V_{u,x}w-2w=0.
\end{equation}
Thus we have two conditions $V_{x,u}=\id+L_sL_t$ and $V_{u,x}=\id+L_tL_s$.
These two identities are equivalent when we consider $x,u\in\A$ and $s,t\in\Ss$,
since $V_{u,x}^\eps =\id^\eps+(L_tL_s)^\eps$ gives $-V_{x,u}=-\id-L_sL_t$ with $\eps$ as defined on page \pageref{Veps}.

Since the left and right inverses of an one-invertible element are uniquely determined,
the system of equations \eqref{inverse} has either no solutions or it has a unique solution.
\end{proof}

\begin{remark}
        \begin{compactenum}[(i)]
            \item
                In the setup of \cite{AF99}, the equations~\eqref{eq:ad+} and~\eqref{eq:ad-} are not necessarily equivalent, since they deal with
                the more general situation of Kantor pairs.
            \item
                In \cite{AF99}, the identity that should be satisfied is given by $V_{x,u}=2(\id +L_sL_t)$;
                notice that the difference with the first equation of \eqref{inverse} is caused by the fact that the $V$-operator in \cite{AF99}
                is the double of the $V$-operator of the structurable algebra.
        \end{compactenum}
\end{remark}

The following lemma shows that one-invertibility in $\A\times\Ss$ is a generalization of conjugate invertibility in $\A$.
\begin{lemma}\label{le:1inv to div}
    \begin{compactenum}[\rm (i)]
        \item
            Let $x\in \A$.
            Then $(x,0)\in\A\times\Ss$ is one-invertible if and only if $x$ is conjugate invertible in $\A$.
            The left and right inverse of $(x,0)$ are both equal to $(\hx,0)$.
        \item
            Let $s\in\Ss$.
            Then $(0,s)\in\A\times \Ss$ is one-invertible if and only if $s$ is conjugate invertible in $\A$.
            The left and right inverse of $(0,s)$ are both equal to $(0,\hat{s})$.
    \end{compactenum}
\end{lemma}
\begin{proof}
\begin{compactenum}[(i)]
    \item
        If we want to determine the one-invertibility of $(x,0)$, the system of equations \eqref{inverse} reduces to
        \begin{align*}
        V_{x,u}&=\id,\\
        0&=-\frac{1}{3}U_xtx,\\
        0&=0.
        \end{align*}
        From the first equation it follows that this system of equations can only have a solution if $x$ is conjugate invertible.
        In this case, by \eqref{Vuhatu}, $u=\hat{x}$ and $t=0$ is the solution.
    \item
        If we want to determine the one-invertibility of $(0,s)$, the system of equations \eqref{inverse} reduces to
        \begin{align*}
        0&=\id+L_sL_t,\\
        su&=0,\\
        0&=0.
        \end{align*}
        From the first equation it follows that this system of equations can only have a solution if $s$ is conjugate invertible.
        In this case, by \eqref{LhatsLs}, $u=0$ and $t=\hat{s}$ is the solution.
    \qedhere
\end{compactenum}
\end{proof}

The map $P_{(x,s)}$ defined in Definition \ref{def:one inv} is a generalization of the map $P_x$ on the structurable algebra $\A$ defined in \eqref{defPu}.
Indeed, using \eqref{Uxy-Uyx}, we obtain
\begin{multline*}
        P_{(x,0)}a = U_x \bigl( a+\frac{2}{3}\psi(\hx,a)x \bigr) = U_x \bigl( a+\frac{2}{3}(V_{\hx,x}a-V_{a,x}\hx) \bigr) \\
            = \frac{1}{3}U_x(5a-2V_{a,x}\hx) = P_xa .
\end{multline*}


\section{One-invertibility for structurable division algebras}\label{se:divtooneinv}

\begin{quote}
    \em Throughout section~\ref{se:divtooneinv}, we will continue to assume that $k$ is a field of characteristic
    different from $2$ and $3$ and that $\A$ is an algebraic structurable algebra over $k$.
\end{quote}

It follows from Lemma \ref{le:1inv to div} that if $\A$ is a structurable algebra such that each element in $\A\times\Ss\setminus{(0,0)}$ is one-invertible,
then $\A$ is a structurable division algebra.
In Theorem \ref{th:formula-oneinverse}, we will show that also the converse is true:
if $\A$ is a structurable division algebra, then each element of $\A\times\Ss\setminus{(0,0)}$ is one-invertible, and we determine the left and right inverse.

We start by showing that $(1,s)\in \A\times\Ss$ is always one-invertible if $\A$ is a structurable division algebra.

\begin{lemma}[J. Faulkner]\label{le:x=1inv}
Let $\A$ be a structurable division algebra, let $s\neq 0\in \Ss$.
Then $s+\hat{s}\neq 0$, and
\[u:=-\hat{s}(\widehat{s+\hat{s}})=\widehat{1-s^2},\quad  t:=\widehat{s+\hat{s}}\]
is the solution of
\begin{gather}\label{inversex=1}
\begin{aligned}
V_{1,u}&=\id+L_sL_t,\\
su&=-\frac{1}{3}U_1t,\\
\psi(1,st)&=0.
\end{aligned}
\end{gather}
In particular, $(1,s)$ is one-invertible.
\end{lemma}
\begin{proof}
It follows from \eqref{hatsx} and \eqref{LhatsLs} that $-\hat{s}(\widehat{s+\hat{s}})=\widehat{1-s^2}$.
It follows from \eqref{LhatsLs} that $s+\hat{s}=0$ if and only if $s^2=1$.

We suppose that $s^2=1$ and deduce a contradiction.
Define $a := 1+s$.
It follows from \eqref{sxy} that $L_sL_s=L_{s^2}=\id$; therefore $(xa)\overline{a}=(x\overline{a})a=0$ for all $x\in \A$.
Since $a\neq 0$, it is conjugate invertible, we find that
\[a=V_{a,\hat{a}}a=2(a \overline{\hat{a}})a-(a\overline{a})\hat{a}=2(a \overline{\hat{a}})a\]
and
\[a=V_{\hat{a},a}a=(\hat{a} \overline{a})a +(a \overline{a})\hat{a}-(a \overline{\hat{a}}) a=-(a \overline{\hat{a}}) a,\]
a contradiction. It follows that $s+\widehat{s}\neq0$.

Next, we claim that, for $t := \widehat{s+\hat{s}}$,
\begin{align}\label{eq:Faulkner}
L_{st}=L_sL_t=L_tL_s.
\end{align}
By \eqref{LhatsLs} and the fact that $L_s^2=L_{s^2}$ we have $[L_s,L_{s+\hat{s}}]=0$, therefore $[L_s,-L^{-1}_{s+\hat{s}}]=[L_s,L_t]=0$.
Hence $[s,t]=[L_s,L_t](1)=0$ and from \eqref{sxy} it follows that for all $y\in \A$
\begin{align*}
        2(L_{st}-L_sL_t)y&=2[s,t,y]=[s,t,y]-[t,s,y]\\
        &=(L_{[s,t]}-[L_s,L_t])y =0,
\end{align*}
and \eqref{eq:Faulkner} follows.

We now verify that $u=-\hat{s}(\widehat{s+\hat{s}})$ and $t=\widehat{s+\hat{s}}$ are solutions of \eqref{inversex=1}.
\begin{compactenum}[(1)]
\item Since the conjugate inverse of a skew element is again a skew element and $\hat{s}$ and $t$ commute, we have $\overline{u}=u$; hence $V_{1,u}=L_{u}$. From \eqref{eq:Faulkner} it follows that $\id+L_sL_t=\id+L_tL_s=L_{1+st}$, the first equation of \eqref{inversex=1} is satisfied since
\[1+st=-(s+\hat{s})t+st=-\hat{s}t=u.\]
\item
We have $\frac{1}{3}U_1t=-t=L_sL_{\hat{s}}t=-su$.
\item
We have $\psi(1,st)=\overline{st}-st=ts-st=0$ by \eqref{eq:Faulkner}.
\end{compactenum}
This shows that $(u,t)$ is indeed a solution of \eqref{inversex=1}.
By Theorem \ref{th:1inv}, we conclude that $(1,s)$ is one-invertible.
\end{proof}

From now on let $\A$ be a structurable division algebra and let $(x,s)\in \A\times\Ss$ with $x\neq0$ and $s\neq 0$.
Our aim is to determine the solution of \eqref{inverse}.
It follows from \eqref{1u} that $x=\doublehat{x}$ is the unity in the structurable algebra $A^{\langle \hat{x} \rangle}$,
an isotope of $\A$ as described in Construction \ref{constr:isot}.
From the previous lemma we know that $(x,s)$ is one-invertible in the algebra  $A^{\langle \hat{x} \rangle}$.
We will show that this implies that $(x,s)$ is one-invertible also in $\A$.
We are indebted to John Faulkner for bringing this method to our attention.

\begin{definition}
Let $x\in \A \setminus \{ 0 \}$, and define
\[ \alpha_x \colon \A\to \A^{\langle \hat{x} \rangle} \colon y\mapsto y. \]
Note that $\alpha_x(1)$ is not the unit in $\A^{\langle \hat{x} \rangle}$, but that $x$ is the unit.
The map $\alpha_x$ is an isotopy (see Definition \ref{def:isotopy}) with $\hat{\alpha}_x=P_x$.
Indeed, using \eqref{xyzu} we get
\[\alpha_x\{x,y,z\}=\{\alpha_x x, \hat{\al}_x y, \alpha_x z\}^{\langle \hat{x} \rangle}
\iff  \{x,y,z\}=\{ x, P_{\hat{x}}P_x y, z\},\]
which holds by \eqref{PuPhatu}.
\end{definition}

We need to determine a map from $\Ss$ onto $\Ss\is=\Ss\hx$ that is compatible with $\alpha_x$ and $\hat{\alpha}_x$.

\begin{definition}\label{def:qx}
Let $x\in \A \setminus \{ 0 \}$.
Define the $k$\dash linear maps
\begin{align*}
        & q_x \colon \Ss\to\Ss \colon s\mapsto \frac{1}{6}\psi(x,U_x(sx)),\\[1ex]
        & \be_x\colon \Ss\mapsto \Ss^{\langle \hat{x} \rangle} \colon s\mapsto s\hat{x},\\
        & \hat{\be}_x\colon \Ss\mapsto \Ss^{\langle \hat{x} \rangle} \colon s\mapsto q_x(s)\hx.
 \end{align*}
\end{definition}
\begin{lemma}\label{le:be}
Let $\A$ be a structurable division algebra, and let $x \in \A \setminus \{ 0 \}$.
Then
\begin{compactenum}[\rm (i)]
    \item\label{qx}
        $q_x(s) \hx=P_x(sx)=-\frac{1}{3}U_x sx=\half \psi(P_x s,P_x 1)\hat{x}$ for all $s \in \Ss$;
    \item
        $(q_x)^{-1}=q_\hx$, hence $\hat{\be}_x$ is a bijection;
    \item
        \label{qxh} $\widehat{q_x(s)}=q_{\hx}(\hat{s})$ for all $s \in \Ss \setminus \{ 0 \}$;
    \item
        we have
        \[ \alpha_x(L_sy)=L_{\beta_x(s)}\is\hat{\alpha}_x y\quad \text{ and }\quad \hat{\alpha}_x(L_sy)=L_{\hat{\beta}_x(s)}\is\alpha_x y \]
        for all $s \in \Ss$ and all $y\in \A$.
\end{compactenum}
\end{lemma}

\begin{proof}
\begin{compactenum}[(i)]
    \item
        By \eqref{Pusu}, we have $P_x(sx)=-\frac{1}{3}U_x (sx)$.
        It follows from \eqref{hatuskew} that $U_x sx=-\frac{1}{2}\psi(x,U_x sx) \hx$.
        Combining \eqref{Puxyz} and \eqref{Uxy-Uyx} yields
        \begin{align}\label{Pupsi}
        P_uL_{\psi(y,z)}=L_{\psi(P_u y,P_u z)} P_{\hat{u}},
        \end{align}
        for all $u,y,z\in \A$, and by \eqref{Puhatu} we have
        \[\half \psi(P_x s,P_x 1)\hat{x}=\half(P_x(\psi(s,1)x))=P_x(sx).\]
    \item
        Using (i) we find
        \[q_\hx(q_x(s))x=P_\hx(q_x(s)\hx)=P_\hx(P_x(sx))=sx,\]
        which shows that $q_\hx(q_x(s))x=s x$; this implies that $q_\hx\circ q_x=\id$.
        Similarly, $q_x\circ q_\hx=\id$.
    \item
        By \eqref{hatsx} and \eqref{hatalphau} ,
        \[\widehat{(q_x(s))} x=\widehat{(q_x(s)\hx)}=\widehat{P_x(sx)}=P_\hx(\hat{s}\hx)=q_{\hx}(\hat{s})x;\]
        it follows that $\widehat{q_x(s)}=q_{\hx}(\hat{s})$.
    \item
        The first equality follows by \eqref{Lusu}:
        \begin{align*}
        L_{\beta_x(s)}\is\hat{\alpha}_x y=L_{s\hx}\is P_xy
        =L_s P_{\hx}P_x y
        =L_s y.
        \end{align*}
        The second equality follows by (i), \eqref{Lusu} and \eqref{Pupsi}:
        \begin{align*}
        L\is_{\hat{\be}_x (s)}\al_x y&=L\is_{ \half \psi(P_x s,P_x 1)\hat{x}} y
        =L_{ \half \psi(P_x s,P_x 1)}P_{\hat{x}} y\\
        &=P_x(L_{\half\psi(s,1)} y)
        =P_xL_s y. \qedhere
        \end{align*}
\end{compactenum}
\end{proof}

\begin{lemma}\label{le:isotopic1inv}
	Let  $(x,s)\in \A\times \Ss\setminus\{0,0\}$ and $(u,t)\in \A\times\Ss\setminus\{0,0\}$.
	Then $(u,t)$ is the solution of the equations \eqref{inverse} w.r.t. $(x,s)$ in $\A$ if and only if
	$(\hat{\al}_x u, \hat{\be}_x t)\in \A\is\times\Ss\is$ is the solution of the equations \eqref{inverse} w.r.t. $(\al_x x,\be_x s)$ in $\A\is$.
\end{lemma}
\begin{proof}
Let  $(x,s)\in \A\times \Ss\setminus\{0,0\}$ and let $\al:=\al_x, \hat{\al}:=\hat{\al}_x, \be:=\be_x, \hat{\be}:=\hat{\be}_x$.
Suppose that the following equations hold in $\A$:
\begin{align*}
	V_{x,u}&=\id+L_sL_t,\\
	su&=-\frac{1}{3}U_xtx,\\
	{\psi(x,s(tx))}&=0.
\end{align*}
Applying the isotopy $\al$ we obtain, using Lemma \ref{le:be}, that
\begin{align*}
V\is_{\al x,\hat{\al} u}\al&=\al+L\is_{\be s}L\is_{\hat{\be} t}\al,\\
L\is_{\be s}\hat{\al} u&=-\frac{1}{3}U\is_{\al x} L\is_{\hat{\be} t}\al x,\\
{\be (\psi(x,s(tx)))}&=0.
\end{align*}
We determine ${\be \psi(x,s(tx))}$ in terms of the multiplication of $\A\is$.
By \eqref{psiuxy}, we have
\[\be \psi(x,s(tx))=\psi(x,s(tx))\hat{x} =\psi\is(\al x,\al s(tx))=\psi\is(\al x, L\is_{\be s}L\is_{\hat{\be} t}\al x).\]
Since $\al,\hat{\al},\be,\hat{\be}$ are bijections we conclude that $(u,t)$ is the solution of the equations for $(x,s)$ in $\A$ if and only if $(\hat{\al} u, \hat{\be} t)$ is the solution of the equations for $(\al x,\be s)$ in $\A\is.$
\end{proof}

\begin{remark}
Define the Lie algebra $\G\is$ as the Lie algebra obtained by applying Definition \ref{def:KacA} to $\A\is$. Define the graded bijection $\gamma_x \colon \G\to \G\is$ given by $\ga_x \colon \G_i\to\G\is_i$ for $i\in[-2,2]$ such that
\begin{alignat*}{2}
\ga_x|_{\G_1}&=\alpha_x,
&\ga_x|_{\G_{-1}}&=\widehat{\alpha}_x,\\
\ga_x|_{\G_2}&=\be_x,
&\ga_x|_{\G_{-2}}&=\widehat{\be}_x,\\
\ga_x(V_{a,b})&=V_{a,P_xb}\is =V_{a,b},\quad
&\ga_x(\zeta)&=\zeta.
\end{alignat*}
Using \eqref{psiuxy}, \eqref{Lusu}, \eqref{Pupsi} and Lemma \ref{le:be}(iii), it can be verified that $\ga_x$ is a Lie algebra isomorphism.
This can be used to give a different proof of Lemma~\ref{le:isotopic1inv}, but the verification of the fact that $\ga_x$ is a Lie algebra isomorphism
requires more effort than we needed in our proof of Lemma~\ref{le:isotopic1inv}.

In \cite[Section 12]{AH81} it is shown that each isomorphism from $K(A)$ to $K(A')$ can be obtained in a similar way from an isotopy between $\A$ and $\A'$.
\end{remark}

We are now ready to prove the main theorem of this section.
\begin{theorem}\label{th:formula-oneinverse}
	Let $\A$ be a structurable division algebra, and let $0\neq x\in \A$ and $0\neq s\in \Ss$.
	Then $(x,s)$ is one-invertible.
	Moreover, the solution of the system of equations \eqref{inverse} is
	\[ u=\big(x-s(q_{\hx}(s)x)\big)^\wedge=-\hat{s}\big((q_{\hx}(s)+\hat{s})^\wedge\hx\big)
		\quad\text{and}\quad t=\big(s+q_x(\hat{s})\big)^\wedge, \]
	where the expressions of which the conjugate inverse is taken are different from zero.
\end{theorem}
\begin{proof}
Let $0\neq x\in \A$ and $0\neq s\in \Ss$, and let $\al:=\al_x, \hat{\al}:=\hat{\al}_x, \be:=\be_x, \hat{\be}:=\hat{\be}_x$.
We have $(\al x,\be s)=(x,s\hat{x})$.
Now $x=\widehat{\hat{x}}=1\is$, by \eqref{1u}.
We can apply Lemma \ref{le:x=1inv} to find the solution of the equations \eqref{inverse} for $(x,s\hat{x})$ in $\A\is$;
we then use Lemma \ref{le:isotopic1inv} to translate this solution back to $\A$.

In Lemma \ref{le:x=1inv}, we found that for $(1,s)$ we have $u=-L_{\hat{s}} t=\widehat{1-s^2}$ and $t=\widehat{s+\hat{s}}$.
It also follows from this lemma that the expressions of which the conjugate inverses are considered, are never zero.

By \eqref{isotopeinv}, Lemma \ref{le:be}\eqref{qx} and \eqref{hatsx}, $\hat{\be}(t) \in \A\is$ is equal to
\begin{multline*}
	\hat{\be}(t) = \big(\be(s)+(\be(s))^{\wedge\langle \hx\rangle}\big)^{\wedge\langle \hx\rangle} = P_x (\big(s\hx+P_x (\widehat{s \hx})\big)^{\wedge})\\
	= P_x (\big(s\hx+P_x (\hat{s} x)\big)^{\wedge})
	= P_x (\big(s\hx+q_x(\hat{s})\hx\big)^{\wedge})\\
	= P_x(\big(s+q_x(\hat{s})\big)^{\wedge}x)
	= q_x(\big(s+q_x(\hat{s})\big)^{\wedge})\hx.
\end{multline*}
It follows that $t=\big(s+q_x(\hat{s})\big)^{\wedge}$.
In order to determine $\hat{\alpha}_x(u)$ in $\A\is$, we first observe that
\[ {L_{\be s}\is}^2\al{x}={L_{s\hx}\is}^2x =L_{s\hx}\is s P_\hx x=L_{s\hx}\is s \hx=sP_{\hx}(s \hx)=s(q_{\hx}(s)x). \]
Then
\[ \hat{\al}(u)=\big(\al{x}-{L_{\be s}\is}^2\al{x}\big)^{\wedge\langle \hx\rangle}=P_x(\big(x-{L_{s\hx}\is}^2x\big)^\wedge)=P_x(\big(x-s(q_{\hx}(s)x)\big)^\wedge), \]
and hence
\[
u=\big(x-s(q_{\hx}(s)x)\big)^\wedge=\big(-s(\hat{s}x)-s(q_{\hx}(s)x)\big)^\wedge=-\hat{s}(\big(q_{\hx}(s)+\hat{s}\big)^\wedge \hx). \qedhere
\]
\end{proof}

Combining Theorem \ref{th:1inv}, Lemma \ref{le:1inv to div} and the previous theorem, we obtain an expression for the left and right inverses of elements in $\A\times\Ss$.
\begin{corollary}\label{cor:oneinverse}
	Let $\A$ be a structurable division algebra. Then all elements in $\A\times\Ss\setminus\{(0,0)\}$ are one-invertible.
	Let $0\neq x\in \A$ and $0\neq s\in \Ss$.
	Then:
	\begin{compactitem}
	\item the left and right inverse of $(x,0)$ are both equal to $(\hx,0)$;
	\item the left and right inverse of $(0,s)$ are both equal to $(0,\hat{s})$;
	\item the right inverse of $(x,s)$ is
	\[ \Bigl( -\hat{s}\big((q_{\hx}(s)+\hat{s})^\wedge\hx\big)+ \big(s+q_x(\hat{s})\big)^\wedge x,\ \big(s+q_x(\hat{s})\big)^\wedge \Bigr) , \]
	and the left inverse of $(x,s)$ is
	\[ \Bigl( -\hat{s}\big((q_{\hx}(s)+\hat{s})^\wedge\hx\big)- \big(s+q_x(\hat{s})\big)^\wedge x,\ \big(s+q_x(\hat{s})\big)^\wedge \Bigr). \]
	\end{compactitem}
\end{corollary}

%
%
\chapter{Simple structurable algebras and simple algebraic groups}\label{se:SSA and SAG}

In this chapter, we will study the connection between simple structurable algebras and simple algebraic groups.
In one direction, we will show in section~\ref{ss:SAG-SSA}, in full generality, that every algebraic central simple structurable algebra $\A$ over a field $k$
with $\Char(k) \neq 2,3$ gives rise to an adjoint simple algebraic $k$-group.
For the converse, we restrict to the rank one case; in section~\ref{ss:SSA-SAG}, we will show that,
for any adjoint simple algebraic group $\GG$ of $k$\dash rank 1  over a field~$k$ with $\Char(k) \neq 2,3$,
there is a structurable division algebra $\A$ over $k$ such that $\GG$ is precisely the algebraic group arising from $\A$.

Again, we will pay special attention to the case $\Char(k) = 5$ and the related notion of algebraicity.
More precisely, we will show in section~\ref{ss:dedalg} that every Lie algebra of an adjoint simple algebraic $k$-group is indeed algebraic,
and we will deduce from this that any structurable division algebra is algebraic.
This will allow us, from section~\ref{ss:SSA-SAG} on, to part with the algebraicity condition when we deal with structurable division algebras.

\section{Simple algebraic groups from simple structurable algebras}\label{ss:SAG-SSA}

In this section, we show how to associate in a natural way to an algebraic
simple structurable algebra $\A$ over a field $k$ of characteristic
$\neq 2,3$ a simple algebraic $k$\dash group $\GG$. Eventually, this will serve to prove that
the Moufang sets we will construct in Theorem~\ref{mainth:moufset} are, in fact,
Moufang sets arising from linear algebraic groups of $k$\dash rank one (as described in Theorem~\ref{th:M(G)}),
but our results in section~\ref{ss:SAG-SSA} are more general.

All commutative rings and algebras mentioned in this section are assumed to be unital.

The main result of the present section is the following theorem.

\begin{theorem}\label{thm:AGadjoint}
Let $\A$  be an algebraic central simple structurable algebra over a field $k$ of characteristic different from $2,3$.
Then the algebraic $k$\dash group $\GG=\Aut(K(\A))^\circ$
is an adjoint absolutely simple group of $k$\dash rank $\geq 1$, and $K(\A)=[\Lie(\GG),\Lie(\GG)]$.
\end{theorem}

In the proof of Theorem~\ref{thm:AGadjoint} we will employ the following Lie-theoretic gadget.
Assume that $R$ be a commutative ring, and let $\LL=\bigoplus\nolimits_{i\in\ZZ}\LL_i$ be a $\ZZ$-graded Lie algebra over $R$.
Recall that $\LL$ is called \emph{$(2n+1)$-graded}, if $\LL_i=0$ for all $i\in\ZZ$ such that $|i|>n$.

\begin{construction}\label{constr:newL} For any $(2n+1)$-graded Lie algebra $\LL$ over $R$ we define the $(2n+1)$-graded Lie algebra
\[
\newl \LL=\bigoplus\limits_{i\in\ZZ}\newl \LL_i
\]
over $R$ as follows.
For any $i\neq 0$, $\newl \LL_i=\LL_i$.  For $i=0$, we define $\newl \LL_0$ to be the set of all
$\phi=(\phi_i)\in\prod\limits_{i\in\ZZ\setminus\{0\}}\End_R(\LL_i)$ satisfying the following conditions:
\begin{align}\label{eq:newL}
	\phi_{i+j}([a,b]) &= [\phi_i(a),b]+[a,\phi_j(b)] \notag \\
		&\text{for all} -n\leq i,j\leq n,\ i,j\neq 0,\ i\neq -j;\ a\in \LL_i,\ b\in \LL_j; \notag \\[1ex]
	\phi_j([[a,b],c]) &= [[\phi_i(a),b],c]+[[a,\phi_{-i}(b)],c]+[[a,b],\phi_{j}(c)] \notag \\
		&\text{for all} -n\leq i,j\leq n,\ i,j\neq 0, \ a\in \LL_i,\ b\in \LL_{-i},\ c\in \LL_j.
\end{align}
In other words, $\phi\in\newl \LL_0$ behaves as a derivation of $\LL$ that preserves grading, except that $\phi$ is not defined
on $\LL_0$. Clearly, $\newl \LL_0$ is an $R$-module.

We define the Lie bracket $[-,-]_{\newl \LL}$ on $\newl \LL$ in terms of the
Lie bracket $[-,-]$ on $\LL$ as follows.
For any $u\in \newl \LL_i,v\in \newl \LL_j$, $i,j\in\ZZ$, we let
\begin{equation}\label{eq:[]}
[u,v]_{\newl \LL}=
\begin{cases}
	[u,v] & \text{if } i\neq -j,\ i,j\neq 0; \\
	\bigl(\ad([u,v])|_{\LL_i}\bigr)_{i\in\ZZ\setminus\{0\}} & \text{if } i= -j,\ i\neq 0; \\
	u(v) & \text{if } i=0,\ j\neq 0;\\
	-v(u) & \text{if } i\neq 0,\ j=0;\\
	uv-vu & \text{if } i=j=0.
\end{cases}
\end{equation}
It is straightforward to check that $\newl \LL$ is indeed a $(2n+1)$-graded Lie algebra over~$R$.
Since it is not likely to provoke confusion, in what follows
we denote the Lie bracket on $\newl \LL$ simply by $[-,-]$.

Note that the Lie algebra $\newl \LL$ by construction contains an element $\zeta\in\newl \LL_0$ which acts
as a {\it grading derivation}:
\[
[\zeta,x]=ix\quad\mbox{for any $i\in\ZZ$ and $x\in\newl \LL_i$.}
\]
It is also easy to see that there is a natural graded Lie algebra homomorphism $\LL\to\newl \LL$ that sends
any element $x\in \LL_0$ to $\bigl(\ad(x)|_{\LL_i}\bigr)_{i\in\ZZ\setminus\{0\}}$.
\end{construction}

\begin{lemma}\label{lem:40}
Let $\LL$ be a $(2n+1)$-graded Lie algebra over a commutative ring $R$ such that $(2n)!\in R^\times$, and
let $\newl \LL=\newl \LL_0\oplus\bigoplus\limits_{i\neq 0}\LL_i$ be the Lie algebra of~Construction~\ref{constr:newL}. Then
the natural homomorphism
$\ad \colon \newl \LL\to\Der_R(\newl \LL)$, $x\mapsto\ad(x)$, is an isomorphism of $\ZZ$-graded $R$-Lie algebras.
\end{lemma}
\begin{proof}
We first show that $\ad$ is injective.
Let $a \in \ker(\ad)$, and write $a = \sum_{i\in\ZZ} a_i$ with $a_i\in\newl \LL_i$ for all $i$.
Since $\newl \LL$ contains the grading derivation $\zeta$, this already implies $a_i = 0$ for all $i \neq 0$,
so that $a \in \newl \LL_0$, and now $\ad a$ coincides with $a$ by definition.
Hence $a = 0$, and we conclude that $\ad$ is injective.

We now show that $\ad$ is surjective; so let $D\in\Der_R(\newl \LL)$ be any derivation.
Write

Write
$
D(\zeta)=\sum\limits_{i\in\ZZ}a_i,
$
where $a_i\in\newl \LL_i$ for all $i$. Note that $a_i=0$ for $|i|>n$.
Set
\[
D' = D + \hspace{-1ex}\sum_{\substack{-n \leq i \leq n \\[.7ex] i\neq 0}} \tfrac 1i \ad(a_i).
\]
We show that $D'\in\ad(\newl \LL)$.
Indeed, take any $b\in\newl \LL_j$, $-n\le j\le n$, and write
$D(b)=\sum\limits_{i\in\ZZ} b_i$, $b_i\in \newl \LL_i$. Then
the $(i+j)$-th graded component of
\begin{equation*}
jD(b)=D([\zeta,b])=[D(\zeta),b]+[\zeta,D(b)]
\end{equation*}
is equal to $jb_{i+j}=[a_i,b]+(i+j)b_{i+j}$, which implies $-ib_{i+j}=[a_i,b]$ for all $i\in\ZZ$.
Note that $[a_i,b]=0$ as soon as $|i|>n$, since $a_i=0$.
Thus we have
\[
D'(b)=D(b)+ \hspace{-1ex}\sum_{\substack{-2n \leq i \leq 2n \\[.7ex] i\neq 0}}\frac 1i[a_i,b]=
D(b)- \hspace{-1ex}\sum_{\substack{-2n \leq i \leq 2n \\[.7ex] i\neq 0}}b_{i+j}=b_j.
\]
Hence
$D'$ preserves $\newl \LL_i$, $i\in\ZZ$. Then, since $D'$ is
a derivation, by the very definition of $\newl \LL$ there is an element $x\in\newl \LL_0$ such that
for any $u\in\newl \LL_i$, $i\neq 0$, we have $D'(u)=\ad(x)(u)$. Then for any
$y\in\newl \LL_0$ and $u\in\newl \LL_i$, $i\neq 0$,
we have
\begin{align*}
	\ad([x,y])(u) &= \ad(x)\bigl(\ad(y)(u)\bigr)-\ad(y)\bigl(\ad(x)(u)\bigr) \\
	&= D'\bigl(\ad(y)(u)\bigr)-\ad(y)\bigl(D'(u)\bigr) \\
	&= D'([y,u])-[y,D'(u)]=[D'(y),u]=D'(y)(u),
\end{align*}
which implies that $[x,y]=D'(y)$ as elements of $\newl \LL_0$. Consequently, $D'$ acts on the whole of
$\newl \LL$ as $\ad(x)$, that is, $D'\in\ad(\newl \LL)$, and therefore $D \in\ad(\newl \LL)$ as well.
\end{proof}

\begin{lemma}\label{lem:basechange-newL}
Let $\LL$ be a $(2n+1)$-graded finite-dimensional Lie algebra over a field $k$. Then for any
commutative $k$\dash algebra $R$, there is a natural isomorphism
of $(2n+1)$-graded $R$-Lie algebras
\[
\newl{\LL}\otimes_k R\cong (\LL\otimes_k R)^{\newl{}}.
\]
\end{lemma}
\begin{proof}
By construction, $\left({(\LL\otimes_k R)^{\newl{}}}\,\right)_0$ is the
$R$-submodule of the free $R$-module
\[
\prod\limits_{i\in\ZZ\setminus\{0\}}\End_R(\LL_i\otimes_k R)\cong \prod\limits_{i\in\ZZ\setminus\{0\}}
\End_k(\LL_i)\otimes_k R
\]
given by a finite set of linear equations
with coefficients in $k$, obtained by substituting the bases of $\LL_i$, $i\in\ZZ\setminus\{0\}$, into
the equations~\eqref{eq:newL}. Since $R$ is free, and hence flat as a $k$\dash module, we have
$\left({(\LL\otimes_k R)^{\newl{}}}\,\right)_0\cong \newl{\LL}_0\otimes_k R$ as $R$-modules. It remains to note that,
clearly, the Lie brackets on $\newl{\LL}\otimes_k R$ and $(\LL\otimes_k R)^{\newl{}}$ are compatible.
\end{proof}

To simplify the notation, under the assumptions of Lemma~\ref{lem:basechange-newL} we will consider
the Lie algebras $\LL$ and $\newl{\LL}$ as functors on the category of $k$\dash algebras $R$, so that, by definition,
\[
\newl{\LL}(R)=\newl{\LL}\otimes_k R\cong (\LL\otimes_k R)^{\newl{}}.
\]
Similarly, $\Aut(\LL)$ and $\Der(\LL)$ will stand for the functors of Lie automorphisms and
Lie derivations of $\LL$ respectively:
\[
\Aut(\LL)(R)=\Aut_R(\LL\otimes_k R),\qquad \Der(\LL)(R)=\Der_R(\LL\otimes_k R).
\]
Note that $\Aut(\LL)$ and $\Der(\LL)$ are naturally represented by closed $k$\dash subschemes of the affine $k$\dash scheme
of linear endomorphisms of $\LL$.
In particular, Lemmas~\ref{lem:40} and~\ref{lem:basechange-newL} imply that there is an isomorphism of functors
\begin{equation}\label{eq:L-Der}
\newl{\LL}\cong\Der(\newl{\LL}).
\end{equation}

\begin{lemma}\label{lem:Gseparable}
Let $\LL$  be a central simple finite-dimensional $(2n+1)$-graded Lie algebra over a field $k$,
and let $K$ be a field extension of $k$.

\begin{compactenum}[\rm (i)]
\item\label{gsep:i} Let $I_i\subseteq\LL(K)_i$, $i\in\ZZ\setminus\{0\}$, be $K$-subspaces
that satisfy the following conditions: for all $i,j\in\ZZ\setminus\{0\}$ such that $i\neq -j$
we have
\begin{gather}
	\label{eq:Iij} [\LL(K)_i,I_j]\subseteq I_{i+j};\\
	\label{eq:Iiii-1} [[\LL(K)_i,I_{-i}],\LL(K)_i]\subseteq I_i;\\
	\label{eq:Iiii-2} [[\LL(K)_i,I_{-i}],\LL(K)_{-i}]\subseteq I_{-i}.
\end{gather}
Then either $I_i=0$ for all $i\in\ZZ\setminus\{0\}$, or $I_i=\LL(K)_i$ for all $i\in\ZZ\setminus\{0\}$.

\item\label{gsep:ii} If $f\in\Aut(\newl\LL)(K)$ satisfies
$f|_{\LL(K)_i}=\id_{\LL(K)_i}$ for all $i\in\ZZ\setminus\{0\}$, then
$f=\id_{\newl\LL(K)}$.
\end{compactenum}
\end{lemma}
\begin{proof}
Since $\LL$ is central simple over $k$, the $K$-Lie algebra $\LL(K)=\LL\otimes_k K$ is also
central simple. In particular, we have
\begin{equation}\label{eq:0=[]}
\LL(K)_0=\sum\limits_{i\in\ZZ\setminus\{0\}}[\LL(K)_i,\LL(K)_{-i}].
\end{equation}
(Indeed, the sum of the right hand side of~\eqref{eq:0=[]} and $\bigoplus\limits_{i \in \ZZ \setminus \{0\}} \LL(K)_i$ is an ideal of $\LL(K)$.)
We claim that
\[
J=\bigoplus\limits_{i\in\ZZ\setminus\{0\}}I_i\oplus\sum\limits_{i\in\ZZ\setminus\{0\}}[\LL(K)_i,I_{-i}]
\]
is a Lie ideal
in $\LL(K)$. By~\eqref{eq:0=[]}, it is enough to check that  for any $i,j\in\ZZ\setminus\{0\}$ we have
\begin{equation}\label{eq:ijJ}
[I_i,\LL(K)_j]\subseteq J,
\end{equation}
and
\begin{equation}\label{eq:iijJ}
[[\LL(K)_i,I_{-i}],\LL(K)_j]\subseteq J.
\end{equation}
The inclusion~\eqref{eq:ijJ} holds by the definition of $J$ if $j=-i$, and by~\eqref{eq:Iij} if $j\neq -i$.
The inclusion~\eqref{eq:iijJ} holds by~\eqref{eq:Iiii-1} and~\eqref{eq:Iiii-2} if $j=\pm i$.
If $j\neq\pm i$, then~\eqref{eq:iijJ} follows from Jacobi identity and~\eqref{eq:Iij}:
\begin{multline*}
[[\LL(K)_i,I_{-i}],\LL(K)_j]\\
\subseteq
[\LL(K)_i,[I_{-i},\LL(K)_j]]+[I_{-i},[\LL(K)_i,\LL(K)_j]]\\
\subseteq [\LL(K)_i,I_{j-i}]+[I_{-i},\LL(K)_{i+j}]\subseteq I_j.
\end{multline*}
Thus, $J$ is an ideal. Then $J=0$ or $J=\LL(K)$, which implies the claim~\eqref{gsep:i}.

Now we prove~\eqref{gsep:ii}. Take any $d\in \newl\LL(K)_0$ and write
\[
f(d)=\sum\limits_{i\in\ZZ} v_i,\quad v_i\in \newl\LL(K)_i,\ i\in\ZZ.
\]
For any $j\in\ZZ\setminus\{0\}$ and any $u\in \newl\LL(K)_j=\LL(K)_j$ we have
\[
[f(d), u]=[f(d),f(u)]=f([d,u])=[d,u]\in \LL(K)_j.
\]
Therefore,
\[
[f(d),u]=[d,u]=[v_0,u],
\]
and for any $i\in\ZZ\setminus\{0\}$ we have $[v_i,u]=0$.
By the very definition of $\newl\LL(K)_0$, the equality $[d,u]=[v_0,u]$ for all $u$ as above implies
that $d=v_0$.

It remains to show that $v_i=0$ for all $i\in\ZZ\setminus\{0\}$. For any $i\in\ZZ\setminus\{0\}$ set
\[
I_i=\{v\in \LL(K)_i \mid [v,\LL(K)_j]=0\ \mbox{for any}\ j\in\ZZ\setminus\{0\}\}.
\]
Clearly, $v_i\in I_i$ for all $i\in\ZZ\setminus\{0\}$.
Moreover, $[I_i, \LL(K)_j] = 0$ for all $j \in \ZZ\setminus\{0\}$, and hence the $I_i$ trivially satisfy the assumptions of~\eqref{gsep:i}.
Therefore, either $I_i=0$ for all $i\in\ZZ\setminus\{0\}$, or $I_i=\LL(K)_i$ for all $i\in\ZZ\setminus\{0\}$.
The latter is not possible since $\LL(K)$ is a simple Lie algebra.
We conclude that $v_i=0$ for all $i\in\ZZ\setminus\{0\}$.
\end{proof}

In the following lemma, we collect some classical facts about the connection between adjoint simple algebraic groups and their Lie algebras.

\begin{lemma}\label{lem:liead}
Let $k$ be a field, $\Char k\neq 2,3$. Let $\GG$ be an adjoint simple algebraic group over $k$.
Let $\LL=\Lie(\GG)$ be its Lie algebra. Then
\begin{compactenum}[\rm (i)]
\item\label{liead:L'} The $k$\dash Lie algebra $[\LL,\LL]$
is central simple. If $G$ is of type $A_n$ and $\Char k$ divides
$n+1$, then $\dim([\LL,\LL])=\dim\LL-1$; otherwise $[\LL,\LL]=\LL$.
\item\label{liead:iso} There are natural isomorphisms of $k$\dash group schemes
\[
\GG\xrightarrow{\cong}\Aut(\LL)^\circ\xrightarrow{\cong}\Aut([\LL,\LL])^\circ.
\]
In particular, $\LL\cong\Der(\LL)\cong\Der([\LL,\LL])$.
\end{compactenum}
\end{lemma}
\begin{proof}
First we prove~\eqref{liead:L'}.
Let $\pi \colon \GG^{sc}\to \GG$ be the simply connected cover of $\GG$ over $k$. Denote $\Lie(\GG^{sc})$ by $\LL^{sc}$.
One has $\ker(\pi)=\Cent(\GG^{sc})$. Note that, since $\GG$ and $\GG^{sc}$ are smooth and $\ker(\pi)$
is finite, we have
\begin{equation}\label{eq:dimG}
\dim\LL=\dim\GG=\dim\GG^{sc}=\dim\LL^{sc}.
\end{equation}
Set
\[
\LL'=\pi(\LL^{sc})=\LL^{sc}/\Lie(\Cent(\GG^{sc}))\subseteq\LL.
\]
Let $\bar k$ be the algebraic closure of $k$. It is well-known that
$\Lie(\GG^{sc}\times_k \bar k)=\LL^{sc}\otimes_k \bar k$ is a split Chevalley Lie algebra over $\bar k$.

Assume that $\Char k=0$. Then $\Lie(\Cent(\GG^{sc}))=0$, and the Lie algebra $\LL^{sc}=\LL$ is simple,
since $\LL^{sc}\otimes_k \bar k$ is simple. Hence
\[
\LL=[\LL,\LL]=\Der(\LL).
\]

Assume that $\Char k> 3$. If $\GG$ is not of type $A_n$ or $\Char k$ does not divide $n+1$, then
the $k$\dash group $\Cent(\GG^{sc})$ is smooth, and hence satisfies $\Lie(\Cent(\GG^{sc}))=0$.
It is also known that in this case the Chevalley Lie algebra $\LL^{sc}\otimes_k \bar k$ is simple
(see e.g.~\cite[p. 29]{Sel67}).
Hence $\LL^{sc}\cong\LL'$ is simple. By~\eqref{eq:dimG} we have $\LL'=\LL$. Therefore, $\LL$
is simple and $[\LL,\LL]=\LL$. It is also known that
we have
$\LL'\otimes_k \bar k=\Der(\LL'\otimes_k \bar k)$~\cite{block62},~\cite[Theorem 3.6]{winter}, and
hence
\[
\LL=\Der(\LL).
\]

Assume that $\GG$ is of type $A_n$ and $\Char k$ divides $n+1$. In this case
$\LL^{sc}\otimes_k\bar k=\Lie(\SL_{n+1,\bar k})$ is the Lie algebra of $(n+1)\times(n+1)$ matrices
of trace 0 over $\bar k$. The Lie algebra $\LL'\otimes_k\bar k$ is the quotient of $\LL^{sc}\otimes_k\bar k$
by its 1-dimensional center $Z$ consisting of diagonal matrices of trace $0$.
It is known that $\LL'\otimes_k\bar k$
is simple and of codimension 1 in its Lie algebra of derivations that coincides with the Lie algebra
of all $(n+1)\times(n+1)$ matrices over $\bar k$ modulo diagonal matrices~\cite[pp. 59--60]{Zas}. Consequently,
\[
\Der(\LL'\otimes_k\bar k)=\Lie(\PGL_{n+1,\bar k})=\LL\otimes_k\bar k.
\]
Clearly, we also have
\[
[\LL\otimes_k\bar k,\LL\otimes_k\bar k]=\LL'\otimes_k\bar k.
\]
We claim that these two facts imply
\begin{equation}\label{eq:DerLL-An}
\Der(\LL\otimes_k\bar k)=\LL\otimes_k\bar k.
\end{equation}
Indeed, let $D$ be a $\bar k$-derivation of $\LL\otimes_k\bar k$. Then $D$ induces a derivation
on $[\LL\otimes_k\bar k,\LL\otimes_k\bar k]$. Hence there is $x\in \LL\otimes_k\bar k$ such that
$D|_{\LL'\otimes_k\bar k}=\ad x|_{\LL'\otimes_k\bar k}$. Then for any $y\in \LL\otimes_k\bar k$ and any
$u\in \LL'\otimes_k\bar k$ we have
\[
[D(y),u]=D([y,u])-[y,D(u)]=[x,[y,u]]-[y,[x,u]]=[[x,y],u].
\]
Since $\LL'\otimes_k\bar k$ has trivial center, we deduce that $D(y)=[x,y]$ for any $y\in \LL\otimes_k\bar k$, which means that $D=\ad x$.
This proves~\eqref{eq:DerLL-An}.
Therefore, over $k$ we conclude that $\LL'=[\LL,\LL]$ is simple, of codimension 1 in $\LL$, and
\[
\LL=\Der([\LL,\LL])=\Der(\LL).
\]

It remains to show that in all these cases, $\LL'=[\LL,\LL]$ is, moreover, central simple. Let $l$ be any field extension of $k$.
Then $\LL'\otimes_k l=[\LL\otimes_k l,\LL\otimes_k l]$. Since $\LL\otimes_k l=\Lie(\GG\times_k l)$,
and $\GG\times_k l$ is an adjoint simple $l$-group, the above argument implies that $\LL'\otimes_k l$ is simple.
Hence $\LL'$ is central simple.

Now we finish the proof of~\eqref{liead:iso}.
Let $\mathcal{M}$ denote any of the Lie algebras $\LL$, $\LL'$. Let
$\Ad \colon \GG\to\Aut(\mathcal{M})$ be the natural homomorphism of $k$\dash groups. By~\cite[Ch.\@~II, \S 4, 2.3]{DeGa}
we have $\Lie(\Aut(\mathcal{M}))\cong\Der(\mathcal{M})$, and by~\cite[Ch.\@~II, \S 4, 4.2]{DeGa} we know that
$\Lie(\Ad)=\ad \colon \LL\to\Der(\mathcal{M})$ is the natural map. We have seen above that the map $\ad$ is bijective. Then by~\cite[Ch.\@~II, \S 5, Corollaire 5.5(a)]{DeGa}
$\ker(\Ad)$ is \'etale. Since $\GG$ is adjoint, this implies that $\ker(\Ad)=1$.
Then by~\cite[Ch.\@~II, \S 5, Proposition 5.1 and Corollaire 5.6]{DeGa} $\Ad$ is an open and closed immersion.
Since $\GG$ is connected, this implies that $\GG\cong\Aut(\mathcal{M})^\circ$.
\end{proof}

\begin{definition}\label{def:gradingtorus}
Let $\LL=\bigoplus\limits_{i\in\ZZ}\LL_i$ be a $\ZZ$-graded Lie algebra over a field $k$.
Consider  the 1-dimensional split $k$\dash subtorus $\GS\cong\Gm$ of $\Aut(\LL)$
defined as follows: for any $k$\dash algebra $R$, any $t\in\Gm(R)$,
and any $i\in\ZZ$, $v\in\LL_i\otimes_{k} R$, we set $t\cdot v=t^iv$. We call $\GS$
\emph{the grading torus of $\LL$}.
\end{definition}

We deduce Theorem~\ref{thm:AGadjoint} from the following more general result.

\begin{theorem}\label{thm:LGadjoint}
Let $\LL$ be an algebraic central simple $5$\dash graded Lie algebra over
a field $k$ of characteristic different from $2,3$, such that $\LL\neq\LL_0$.
Then the algebraic $k$\dash group $\GG=\Aut(\LL)^\circ$
is an adjoint absolutely simple group of $k$\dash rank $\geq 1$, satisfying $\LL=[\Lie(\GG),\Lie(\GG)]$ and
$\Lie(\GG)\cong\newl{\LL}$.
\end{theorem}
\begin{proof}
Instead of proving the claim of the theorem directly, we prove that
$\GG=\Aut\bigl(\newl{\LL}\,\bigr)^\circ$ is an adjoint simple algebraic $k$\dash group of $k$\dash rank $\geq 1$, where
$\newl{\LL}$ is the Lie algebra associated to the $5$\dash graded algebra $\LL$ in Construction~\ref{constr:newL}.
We first explain how this result implies the theorem.
By~\cite[II, \S 4, 2.3]{DeGa} for any finite-dimensional Lie algebra $\LL$ over $k$ there is a natural isomorphism of functors
\begin{equation}\label{eq:derLL}
	\Der(\LL)\cong\Lie(\Aut(\LL)).
\end{equation}
Therefore,
\[
\Lie(\GG)=\Lie\bigl(\Aut\bigl(\newl{\LL}\,\bigr)\bigr)\cong\Der\bigl(\newl{\LL}\,\bigr).
\]
Combining this fact with~\eqref{eq:L-Der}, we conclude that $\Lie(\GG)\cong\newl{\LL}$.
Since $\LL$ is simple, the natural
homomorphism of Lie algebras $\LL\to\newl{\LL}$
is injective, and $\LL$ is a simple ideal of $\newl{\LL}$.
By Lemma~\ref{lem:liead}, also $[\Lie(\GG),\Lie(\GG)]$ is a simple ideal of $\Lie(\GG)\cong\newl{\LL}$.
Since $[\LL,\LL]\neq 0$, we thus have
\[
\LL\cap [\newl{\LL},\newl{\LL}]=\LL=[\newl{\LL},\newl{\LL}].
\]
Applying Lemma~\ref{lem:liead} again, we deduce that $\GG\cong\Aut(\LL)^\circ$.

Now we proceed to establish that $\GG=\Aut\bigl(\newl{\LL}\,\bigr)^\circ$ is an adjoint simple algebraic
$k$\dash group of $k$\dash rank $\geq 1$. Let $\GS\subseteq\Aut\bigl(\newl{\LL}\,\bigr)$ be the grading torus of $\newl{\LL}$.
Since $\GS$ is connected, we have $\GS\subseteq\GG$. By Lemma~\ref{lem:40} we have
$\Der\bigl(\newl{\LL}\,\bigr)=\newl{\LL}$,
and~\eqref{eq:derLL} then implies $\Lie(\GG)=\newl{\LL}$;
moreover, the adjoint action of $\GS$ on $\Lie(\GG)$ is precisely the one described in Definition~\ref{def:gradingtorus}.

It remains to show that $\GG$ is an adjoint simple algebraic $k$\dash group. In order to do that, it is
enough to establish the same for
\[
\GG\times_k {\bar k}=\Aut\bigl(\newl{\LL}\otimes_k\bar k\bigr)^\circ,
\]
where $\bar k$ is an algebraic closure of $k$.
Note that by Lemma~\ref{lem:basechange-newL} we have
\[
\newl{\LL}\otimes_k\bar k=\bigl(\LL\otimes_k\bar k\bigr)^{\newl{}}.
\]
Since $\LL$ is algebraic, the $5$\dash graded Lie algebra $\LL\otimes_k\bar k$ is also algebraic by Lemma~\ref{lem:ext-alg}. Since
$\LL$ is central simple, the Lie algebra $\LL\otimes_k\bar k$ is also central simple.
Summing up, we can assume that $k=\bar k$ from now on.

Note that by Lemma~\ref{lem:univ-alg} the $5$\dash graded Lie algebra $\newl{\LL}$ is algebraic.
By Lemma~\ref{lem:esigma'-prod} applied to $\newl{\LL}$, there are
homomorphic closed embeddings
\[
e_\si \colon \GU_\sigma\to \GL\bigl(\newl{\LL}\,\bigr),\quad (x,s)\mapsto e_\si(x,s),
\]
where $\GU_\sigma$ is the affine $k$\dash variety
$\newl{\LL}_\sigma\oplus \newl{\LL}_{2\sigma}=\LL_\si\oplus\LL_{2\si}$ considered as a $k$\dash group with respect to the operation
\[
(x,s)+(y,t)=\bigl(x+y,s+t+\frac 12[x,y]\bigr).
\]
We identify $\GU_\sigma$ with its image under $e_\si$. Since $\newl{\LL}$ is algebraic,
we have $\GU_\sigma\subseteq\Aut(\newl{\LL})$.
Note that since $\GU_\sigma$ as a variety is isomorphic to an affine $k$\dash space, it is connected and smooth.
By the very definition of a Lie algebra of an algebraic group we have
\begin{equation}\label{eq:Lie-U}
\Lie(\GU_\sigma)\cong \newl{\LL}_\sigma\oplus \newl{\LL}_{2\sigma}
\end{equation}
as $k$\dash Lie algebras.

Let $\GZ$ be the closed $k$\dash subgroup of $\Aut\bigl(\newl{\LL}\,\bigr)$ generated by $\GU_+$ and $\GU_-$ in the sense
of~\cite[Exp.\@~$\mathrm{VI_B}$ Prop. 7.1]{SGA3}.  Since $\LL\neq\LL_0$,
the group $\GZ$ is non-trivial. By loc. cit. $\GZ$ is smooth.
Since $\GU_+$ and $\GU_-$ are connected,
$\GZ$ is also connected,
 and hence $\GZ\subseteq\GG$. Since $\GS$ normalizes $\GU_\sigma$, it normalizes
$\GZ$. Therefore, $\Lie(\GZ)$ is a $\ZZ$-graded Lie subalgebra
of $\newl{\LL}$. Moreover,~\eqref{eq:Lie-U} implies that
\begin{equation}\label{eq:Lie(Z)}
\Lie(\GZ)_i=\newl{\LL}_i=\LL_i\quad\mbox{for all}\ i\in\ZZ\setminus\{0\},
\end{equation}
and $\LL\subseteq\Lie(\GZ)$.

We show that $\GZ$ is an adjoint simple $k$\dash group.
Let $\GU$ be the last non-trivial member of the algebraic lower central series of the unipotent radical
$\radu(\GZ)$.
Then $\GU$ is a commutative algebraic $k$\dash group which is a characteristic closed $k$\dash subgroup of $\GZ$.
Since $\GS$ acts on $\GZ$, we conclude that
$\Lie(\GU)$ is a commutative  $\ZZ$-graded Lie ideal in $\Lie(\GZ)$.
Note that the $k$\dash subspaces $\Lie(\GU)_i \subseteq \newl{\LL}_i = \LL_i$,
$i\in\ZZ\setminus\{0\}$, satisfy the assumptions of Lemma~\ref{lem:Gseparable}(i),
and recall that $\LL$ is central simple.
Hence $\Lie(\GU)_i=0$ or $\Lie(\GU)_i=\newl{\LL}_i$ for all $i\in\ZZ\setminus\{0\}$.
The latter is not possible, since $\Lie(\GU)$ is a commutative Lie algebra.
Hence $\Lie(\GU)_i=0$ for all $i\in\ZZ\setminus\{0\}$.
Then also $\Lie(\GU)_0=0$, since
\[
[\Lie(\GU)_0,\newl{\LL}_i]=[\Lie(\GU)_0,\Lie(\GZ)_i]\subseteq\Lie(\GU)_i=0
\]
for any $i\in\ZZ\setminus\{0\}$, and on the other hand the adjoint action of
$\Lie(\GU)_0\subseteq\newl{\LL}_0$
on $\bigoplus\limits_{i\in\ZZ\setminus\{0\}}\newl{\LL}_i$ is faithful by the definition of
$\newl{\LL}$.
Thus, $\Lie(\GU)=0$, and hence $\GU$ is trivial. This proves that $\GZ$ is reductive.
Moreover, $\GZ$ acts
faithfully on $\Lie(\GZ)$, since $\GZ\subseteq\GG$ and $\GG=\Aut\bigl(\newl{\LL}\,\bigr)^\circ$ acts faithfully on
$\bigoplus\limits_{i\in\ZZ\setminus\{0\}}\newl{\LL}_i$ by Lemma~\ref{lem:Gseparable}(ii). Hence $\GZ$
is semisimple and adjoint. Assume $\GZ$ is not simple. Then $\GZ\cong \GZ_1\times \GZ_2$ is a product of two non-trivial
semisimple adjoint $k$\dash groups, and $\Lie(\GZ)=\Lie(\GZ_1)\oplus\Lie(\GZ_2)$, where $\Lie(\GZ_1)$ and $\Lie(\GZ_2)$
are non-trivial Lie ideals of $\Lie(\GZ)$. However, this contradicts Lemma~\ref{lem:Gseparable}(i).
Therefore, $\GZ$ is a simple adjoint $k$\dash group.

Now we apply Lemma~\ref{lem:liead} to $\GZ$. Since $\LL\subseteq\Lie(\GZ)$,
and because both $\LL$ and $[\Lie(\GZ),\Lie(\GZ)]$ are simple Lie algebras, we see that $\LL=[\Lie(\GZ),\Lie(\GZ)]$ and
$\Lie(\GZ)=\Der(\LL)$. Since $\Lie(\GZ)\subseteq \newl{\LL}$, and the natural Lie homomorphism
\[
\newl{\LL}\to\Der(\LL)
\]
is injective, we have $\Lie(\GZ)=\newl{\LL}$.
Then, again by Lemma~\ref{lem:liead}, we have
$\GZ=\Aut\bigl(\newl{\LL}\,\bigr)^\circ$. This finishes the proof.
\end{proof}

\begin{proof}[Proof of Theorem~\ref{thm:AGadjoint}]
Since $\A$ is central simple, the Lie algebra $K(\A)$ is also central simple by~\cite[\S 5]{A79}.
Since $\A$ is algebraic, $K(\A)$ is also algebraic.
Thus the claim follows from Theorem~\ref{thm:LGadjoint}.
\end{proof}

The following statement readily follows from the proof of Theorem~\ref{thm:LGadjoint}.

\begin{corollary}\label{cor:newl-K}
Let $\A$  be an algebraic central simple structurable algebra over a field $k$ of characteristic different from $2,3$,
and let $\GG=\Aut(K(\A))^\circ$.
Then $\Lie(\GG)\cong \Der(K(\A))\cong\newl{K(\A)}$.
\qed
\end{corollary}

\section{Deducing algebraicity}\label{ss:dedalg}

In this section we show that if a $5$-graded Lie algebra over a field $k$ of characteristic $\neq 2,3$
is the tangent Lie algebra of an adjoint simple
algebraic $k$-group, then it is algebraic. Of course, this is automatically true if $\Char k\neq 2,3,5$,
thanks to Lemma~\ref{lem:esigma'-alg}. However, since our proof
is independent of characteristic and involves a general technical lemma that is potentially useful in other contexts,
we do not restrict ourselves to the characteristic $5$ case.
As a consequence, we also show that any structurable division algebra over a field of characteristic $5$ is automatically algebraic.

\begin{definition}\label{def:rootsys}
Let $\GG$ be an algebraic group over a field $k$ and $\GT\subseteq\GG$ be a split $n$-dimensional
$k$\dash subtorus of $\GG$.
Let $X^*(\GT)\cong \ZZ^n$ be the group of characters of $\GT$, and let
\[
\Lie(\GG)=\bigoplus_{\alpha\in X^*(\GT)}\Lie(\GG)_\alpha
\]
be the $\ZZ^n$-grading on $\Lie(\GG)$ induced by the adjoint action of $\GT$. We call
\[
\Phi(\GT,\GG)=\{\alpha\in X^*(\GT) \mid \Lie(\GG)_\alpha\neq 0\}
\]
the set of \emph{roots of $\GG$ with respect to $\GT$}.
\end{definition}

If $\GG$ is a reductive algebraic group over $k$ and $\GT$ is a maximal split $k$\dash subtorus of $\GG$, then
$\Phi(\GT,\GG)\setminus \{0\}$ is a root system in the sense of Bourbaki~\cite{Bou} (see~\cite{BorelTits}).
By abuse of language, we call $\Phi(\GT,\GG)$ a root system of $\GG$.

Let $\Phi$ be a root system and $\Pi\subseteq\Phi$ be a system of simple roots. For any $\alpha\in\Phi$
we write
\[
\alpha=\sum\limits_{\beta\in\Pi}m_\beta(\alpha)\beta,
\]
where the coefficients $m_\beta(\alpha)$ are either all non-negative, or all non-positive. Once $\Pi$
is fixed, we denote the corresponding sets of positive and negative roots by $\Phi^\sigma$.

\begin{lemma}\label{lem:5-parab}
Let $k$ be a field, $\Char k\neq 2,3$. Let $\GG$ be an adjoint simple algebraic group over $k$.
Let $\LL=\Lie(\GG)$ be its Lie algebra.
Let $\LL=\bigoplus\limits_{i=-2}^2\LL_i$ be any $5$\dash grading on $\LL$
such that $\LL_1\oplus\LL_{-1}\neq 0$.

\begin{compactenum}[\rm (i)]
\item\label{5-parab:PU} There is a unique pair of opposite parabolic
subgroups $\GP_{\pm }$ of $\GG$ with unipotent radicals $\GU_{\pm}$, and a 1-dimensional split
$k$\dash torus $\GS\le\Cent(\GP_+\cap\GP_-)$, such that
\[
\Lie(\GP_\sigma)=\LL_0\oplus\LL_\sigma\oplus\LL_{2\sigma},\quad \Lie(\GU_\sigma)=\LL_\sigma\oplus\LL_{2\sigma},\quad
\Lie(\GP_+\cap\GP_-)=\LL_0,
\]
and the adjoint action of $\GS$ on $\Lie(\GG)$ induces the given $5$\dash grading.

\item\label{5-parab:roots} There is a maximal split $k$\dash subtorus $\GT$ of $\GG$ such that
\[
\GS\subseteq\GT\subseteq\GP_+\cap\GP_-.
\]
Set $\Phi=\Phi(\GT,\GG)$, then  there is a system of simple roots
$\Pi\subseteq\Phi$ and a non-empty subset $J\subseteq\Pi$ such that
\begin{equation}\label{eq:lieP}
\begin{aligned}
&\Phi(\GT,\GP_\sigma)=\Phi^{\sigma}\cup\bigl(\Phi\cap\ZZ(\Pi\setminus J)\bigr),\\
&\Phi(\GT,\GU_\sigma)=\Phi^{\sigma}\setminus\ZZ (\Pi\setminus J),\\
&\Phi(\GT,\GP_+\cap\GP_-)=\Phi\cap\ZZ (\Pi\setminus J).
\end{aligned}
\end{equation}

\item\label{5-parab:height} For any $-2\leq i\leq 2$ we have
\begin{equation}\label{eq:Li}
\LL_i = \hspace{-2ex}\bigoplus\limits_{\substack{\alpha\in\Phi \colon \\[.6ex] \sum\limits_{\beta\in J}m_\beta(\alpha)=i}}
\hspace*{-2ex}\Lie(G)_\alpha.
\end{equation}

\item\label{5-parab:type} The type of the root system $\Phi\cap\ZZ J$
is $A_1$, $BC_1$, $A_2$, or $A_1\times A_1$.

\end{compactenum}
\end{lemma}
\begin{proof}
First we prove~\eqref{5-parab:PU}. By Lemma~\ref{lem:liead}, we have $\GG=\Aut(\LL)^\circ$ and hence it is a closed $k$\dash subscheme
of $\End(\LL)$. 
 Let $\GS\cong\Gm$ be the grading torus of $\LL$.
Since $\GS$ is connected, we have $\GS\subseteq\GG$.

By~\cite[Exp.\@~XXVI, Prop. 6.1]{SGA3} there is a unique pair of opposite parabolic
subgroups $\GP_{\pm }$ of $\GG$ with a common Levi subgroup
\[
\GLe=\Cent_{\GG}(\GS)=\GP_+\cap\GP_-,
\]
such that
\[
\Lie(\GP_\sigma)=\LL_0\oplus\LL_\sigma\oplus\LL_{2\sigma},\quad \Lie(\GLe)=\LL_0.
\]
Since $\GP_{\sigma}$ is a semidirect product of $\GLe$ and the unipotent radical $\GU_{\sigma}=\radu(\GP_\sigma)$, we have
\[
\Lie(\GU_\sigma)=\LL_\sigma\oplus\LL_{2\sigma}.
\]

Now we prove~\eqref{5-parab:roots}. Let $\GT$ be a maximal $k$\dash split torus of $\GG$ containing $\GS$.
Let $\Phi=\Phi(\GT,\GG)$ be the root system of $\GG$ with respect to $\GT$.
Since $\GS\subseteq\GT$, there is a natural surjective homomorphism
\[
\pi \colon X^*(\GT)\to X^*(\GS)\cong\ZZ.
\]
Since $\Lie(\GP_\sigma)=\LL_0\oplus\LL_\sigma\oplus\LL_{2\sigma}$, we have
\[
\Phi(\GT,\GP_\sigma)=\bigl\{\alpha\in\Phi \mid \pi(\alpha)\in\{0,\sigma,2\sigma\}\bigr\};
\]
in particular, $\Phi(\GT,\GP_+)=-\Phi(\GT,\GP_-)$.
Then $\Phi(\GT,\GP_\sigma)\setminus\{0\}$ is a parabolic set of roots in the
sense of~\cite[Ch.\@ VI, \S 1, D\'ef.\@ 4]{Bou}. Then by~\cite[Ch.\@~VI, \S 1, Prop.\@ 20]{Bou} there is
a system of simple roots $\Pi$ in $\Phi$ and a subset $J$ of $\Pi$ such that $\Phi(\GT,\GP_\sigma)$
satisfies~\eqref{eq:lieP}. Since
\[
\Phi(\GT,\GLe)=\Phi(\GT,\GP_+\cap\GP_-)=\Phi(\GT,\GP_+)\cap\Phi(\GT,\GP_-)
\]
and
\[
\Phi(\GT,\GU_\sigma)=\Phi(\GT,\GP_\sigma)\setminus \Phi(\GT,\GP_+\cap\GP_-),
\]
the rest of~\eqref{eq:lieP} holds as well.

Now we prove~\eqref{5-parab:height}. Note that by~\eqref{eq:lieP} we have
$J\subseteq\Phi(\GT,\GP_+)\setminus\Phi(\GT,\GLe)$ and $\Pi\setminus J\subseteq\Phi(\GT,\GLe)$.
Therefore,
$\pi(\alpha)\in\{1,2\}$ for any $\alpha\in J$ and $\pi(\alpha)=0$ for any $\alpha\in\Pi\setminus J$.
Clearly, there is $\beta\in J$ such that $\pi(\beta)=1$, since otherwise $\LL_1\oplus\LL_{-1}=0$.
Let $\gamma\in\Phi$ denote the sum of all simple roots.
Then
\[
2\geq \pi(\gamma)\ge\sum\limits_{\beta\in J}\pi(\beta)
\]
implies that $|J|\leq 2$ and $\pi(\beta)=1$ for all $\beta\in J$.
Then for any $\alpha\in\Phi$ we have
\[
\pi(\alpha)=\sum\limits_{\beta\in J}m_\beta(\alpha).
\]
This implies~\eqref{eq:Li}.

In order to prove~\eqref{5-parab:type}, note that the root system $\Psi=\ZZ J\cap\Phi$  has $J$ as its system of simple roots.
Clearly, if $|J|=1$, then $\Psi$ has type $A_1$ or $BC_1$. If $|J|=2$,
then $\pi(\beta)\leq 2$ for any $\beta\in\Psi$ readily implies that $\Psi$ is of type $A_1\times A_1$ or $A_2$.
\end{proof}

We will need the following technical recognition lemma, which uses the {\em algebra of distributions} $\Dist(\GG)$ associated to an algebraic $k$\dash group $\GG$
(also known as the {\em hyperalgebra of $\GG$}).
This algebra is very well known to people working in representation theory.
Over a field of characteristic $0$, the representation theory of~$\GG$ is essentially identical to the representation theory of its Lie algebra $\Lie(\GG)$,
but this is no longer true for fields of positive characteristic.
The algebra $\Dist(\GG)$ is a suitable replacement for $\Lie(\GG)$ in this case; for fields of characteristic $0$, it coincides with the universal enveloping
algebra of $\Lie(\GG)$.
We refer to \cite{Tak}, \cite[Chapter II, \S 4]{DeGa} or \cite[Chapter~7]{Jantzen}.
The idea to use the algebra of distributions in our context came from the work of
J.~Faulkner on $3$-graded Lie algebras~\cite{Fau, Fau2}, and was previously applied to the study of $(2n+1)$-graded Lie algebras
of algebraic groups in the third author's thesis~\cite{St-thes}.

\begin{lemma}\label{lem:H1H2}
	Let $k$ be a field such that $n!\in k^\times$. Let $\GG$ be a smooth algebraic group over $k$ equipped
	with an action of $\Gm$  by $k$\dash automorphisms. Assume that the induced $\ZZ$-grading on $\Lie(\GG)$ is a $(2n+1)$-grading.
	Let $\GH_1,\GH_2$ be two $\Gm$-invariant smooth connected closed $k$\dash subgroups of $\GG$ such that
	\[
	\Lie(\GH_1)=\Lie(\GH_2)\subseteq\bigoplus_{i>0}\Lie(\GG)_i.
	\]
	Then $\GH_1=\GH_2$.
\end{lemma}
\begin{proof}
Let $\GH$ be an algebraic $k$\dash group,
let $\Dist(\GH)$ be its $k$\dash Hopf algebra of distributions, and denote the comultiplication in $\Dist(\GH)$ by $\Delta$.
Then $\Lie(\GH)$ identifies with the Lie algebra of primitive elements of $\Dist(\GH)$~\cite[Ch.\@~3, Proposition~3.1.8]{Tak}.
Recall that an {\em infinite divided power sequence} (abbreviated to {\em d.p.s.}) in $\Dist(\GH)$
is a sequence of elements $\{a_i\}_{i=0}^\infty$ such that $a_0=1$ and for any $i\geq 1$,
\[
\Delta(a_i)=\sum_{p+q=i}a_p\otimes a_q.
\]
If the field $k$ is perfect, then by~\cite[Ch.\@~II, \S 5, 2.3]{DeGa} $\GH_{red}$ is a smooth $k$\dash subgroup
of $\GH$, and by~\cite[Ch.\@~3, Corollary 3.3.12]{Tak} the Lie algebra
$\Lie(\GH_{red})$ is the set of elements $a\in\Lie(\GH)$ such that $\Dist(\GH)$ contains
an infinite d.p.s over $a$, i.e. an infinite d.p.s. $\{a_i\}_{i=0}^\infty$ with $a_1=a$.

The $\Gm$-action on $\GG$ induces a $\ZZ$-grading on $\Dist(\GG)$ compatible with the grading on $\Lie(\GG)$.
The algebras $\Dist(\GH_1)$ and $\Dist(\GH_2)$ are $\ZZ$-graded Hopf
subalgebras of $\Dist(\GG)$. 
One has $\GH_1=\GH_2$ if and only if $\Dist(\GH_1)=\Dist(\GH_2)$ by~\cite[Ch.\@~3, Corollary~3.3.9]{Tak}.
By~\cite[Ch.\@~3, Proposition~3.1.7]{Tak}
the functor $\Dist$ commutes with extensions of the base field, so we can assume from the start that $k=\bar k$.
In particular, $k$ is now perfect.

Let $X=\GH_1\cap \GH_2$ be the closed $k$\dash subgroup of $\GG$ which
is the intersection of $\GH_1$ and $\GH_2$ as $k$\dash subgroup schemes of $\GG$.
We have $\Dist(\GX)=\Dist(\GH_1)\cap\Dist(\GH_2)$ by~\cite[Ch.\@~3, Corollary 3.3.9]{Tak}.
We will show that for any $a\in\Lie(\GX)$ there is an infinite d.p.s. over $a$ in $\Dist(\GX)$. Then
$\Lie(\GX_{red})=\Lie(\GH_1)=\Lie(\GH_2)$, and consequently $\GX_{red}=\GH_1=\GH_2$ by~\cite[Ch.\@~II, \S 5, Corollary 5.6]{DeGa},
since $\GH_1$ and $\GH_2$ are connected.

Let $\GH$ be any of $\GH_1$, $\GH_2$, $\GG$. Since $\GH$ is smooth, we have $\GH_{red}=\GH$, and hence
for any $x\in\Lie(\GH)$ there is an infinite d.p.s. over $x$ in $\Dist(\GH)$. By~\cite[Lemma~4(i)]{Fau},
moreover, if $x\in\Lie(\GH)_j$, $j\in\ZZ$, then there exists a homogeneous infinite d.p.s. over $x$, i.e. a d.p.s.
$\{x_i\}_{i=0}^\infty$ such that $x_i\in\Dist(\GH)_{i\cdot j}$ for any $i\geq 0$.

We show by induction on $m$, $0\leq m\leq n$, that there is a homogeneous d.p.s. $\{x^{(k)}\}_{k=0}^\infty\subseteq\Dist(\GH)$
over $x\in\Lie(\GH)_j$ satisfying $x^{(k)}=\frac 1 {k!}x^k$ for any $0\leq k\leq m$.
The case $m=1$ is already known. Assume the claim for $m$; we will prove it for $m+1\leq n$.
Since $x$ is a primitive element, we have
\begin{multline*}
\Delta \Bigl( \tfrac 1{(m+1)!}x^{m+1}-x^{(m+1)} \Bigr) \\
\begin{aligned}
	&= \tfrac 1{m+1}\,\Delta(\tfrac 1{m!}x^m)\cdot (x\otimes 1+1\otimes x) - \textstyle\sum\limits_{i+j=m+1} x^{(i)}\otimes x^{(j)} \\
	&= \tfrac 1{m+1}\Bigl(\,\textstyle\sum\limits_{i+j=m} \tfrac 1{i!}x^i\otimes \tfrac 1{j!}x^j\Bigr)\cdot(x\otimes 1+1\otimes x) \\
	& \hspace*{4ex} - \textstyle\sum\limits_{\substack{i+j=m+1 \\[.7ex] i,j\neq 0}} \tfrac 1{i!}x^i\otimes \tfrac 1{j!}x^j
		- \bigl( x^{(m+1)}\otimes 1+1\otimes x^{(m+1)} \bigr) \\
	&= \bigl( \tfrac 1{(m+1)!}x^{m+1}\otimes 1+1\otimes \tfrac 1{(m+1)!}x^{m+1} \bigr) - \bigl( x^{(m+1)}\otimes 1+1\otimes x^{(m+1)} \bigr) \\
	&= \bigl( \tfrac 1{(m+1)!}x^{m+1}-x^{(m+1)} \bigr) \otimes 1 + 1\otimes \bigl( \tfrac 1{(m+1)!}x^{m+1}-x^{(m+1)} \bigr).
\end{aligned}
\end{multline*}
This means that $\frac 1{(m+1)!}x^{m+1}-x^{(m+1)}$ is a primitive element of $\Dist(\GH)$. Hence
there is a homogeneous d.p.s. $\{y_k\}_{k=0}^\infty$ over $y_1=\frac 1{(m+1)!}x^{m+1}-x^{(m+1)}$ in $\Dist(\GH)$.
By~\cite[Lemma 2]{Fau} the sequence $\{z_k\}_{k=0}^\infty$ where $z_k=y_{k/m+1}$ if $m+1 \mid k$ and
$z_k=0$ otherwise, is also a homogeneous d.p.s. Again by~\cite[Lemma 2]{Fau}, the ``product'' of two d.p.s.
\[
\{z_k\}_{k=0}^\infty\cdot \{x^{(k)}\}_{k=0}^\infty=\bigl\{\!\!\sum\limits_{p+q=k}z_{p}x^{(q)}\bigr\}_{k=0}^\infty
\]
is also a d.p.s. Clearly, it is a homogeneous d.p.s. over $x$ and its first $m+1$ members are of the form
$\frac 1{k!}x^{k}$, as required.

Now let $x\in\Lie(\GX)_j=\Lie(\GH_1)_j\cap\Lie(\GH_2)_j$, $j>0$. Let $\{x_k\}_{k=0}^\infty$ and
$\{y_k\}_{k=0}^\infty$ be two homogeneous d.p.s. over $x$ in $\Dist(\GH_1)$ and $\Dist(\GH_2)$ respectively,
satisfying $x_k=y_k=\frac 1{k!}x^k$ for any $1\leq k\leq n$. We claim that these two sequences are
equal. Indeed, assume that this is true for all $k\leq m$, where $m\geq n$. Then
\begin{align*}
	\Delta(x_{m+1}-y_{m+1})
	&= \Delta(x_{m+1})-\Delta(y_{m+1}) \\
	&= \textstyle\sum\limits_{i+j=m+1}x_i\otimes x_j-\sum\limits_{i+j=m+1}y_i\otimes y_j \\
	&= x_{m+1}\otimes 1+1\otimes x_{m+1}-y_{m+1}\otimes 1-1\otimes y_{m+1} \\
	&= (x_{m+1}-y_{m+1})\otimes 1+1\otimes (x_{m+1}-y_{m+1}).
\end{align*}
This means that $x_{m+1}-y_{m+1}$ is primitive element in $\Dist(\GG)$. Since, clearly, $x_{m+1}-y_{m+1}$ is
homogeneous of degree $j(m+1)$, and $\Lie(\GG)$ is $(2n+1)$\dash graded, we have $x_{m+1}-y_{m+1}=0$. Proceeding
by induction, we show that $\{x_k\}_{k=0}^\infty=\{y_k\}_{k=0}^\infty$. Consequently, there is a homogeneous
d.p.s. over $x$ in $\Dist(\GH_1)\cap\Dist(\GH_2)=\Dist(\GX)$.

It remains to note that if $x,y\in\Lie(\GX)$ are two homogeneous elements,
and $\{x_k\}_{k=0}^\infty$ and
$\{y_k\}_{k=0}^\infty$ are two homogeneous d.p.s. over $x$ and $y$ in $\Dist(\GX)$, then
$\bigl\{\sum\limits_{p+q=k}x_py_q\bigr\}_{k=0}^\infty$ is a d.p.s. over $x+y$ by~\cite[Lemma~2]{Fau}.
Therefore, there is an infinite d.p.s. in $\Dist(\GX)$ over any element of $\Lie(\GX)$, as required.
\end{proof}

\begin{lemma}\label{lem:expad}
Let $k$ be a field, $\Char k\neq 2,3$. Let $\GG$ be an adjoint simple algebraic group over $k$.
Let $\LL=\Lie(\GG)$ be its Lie algebra.
Let $\LL=\bigoplus\limits_{i=-2}^2\LL_i$ be any $5$\dash grading on $\LL$
such that $\LL_1\oplus\LL_{-1}\neq 0$. Let $R$ be any commutative $k$\dash algebra.

\begin{compactenum}[\rm (i)]
\item The $5$-graded Lie algebra $\LL$ is algebraic (in the sense of Definition~\ref{def:algebraic}).

\item Let $\GU_\sigma$ be the $k$\dash subgroups of $\GG\cong\Aut(\LL)^\circ$ introduced in Lemma~\ref{lem:5-parab}.
Then for any any commutative $k$\dash algebra $R$ one has
\[
\GU_{\sigma}(R)=\{e_\sigma(x,s) \mid (x,s)\in (\LL_\sigma\oplus\LL_{2\sigma})\otimes_k R\}.
\]
\end{compactenum}
\end{lemma}
\begin{proof}
Let $\GU_\sigma'$ denote $\LL_\sigma\oplus\LL_{2\sigma}$ considered as a $k$\dash group with the
operation
\[
(x,s)*(y,t)=(x+y,s+t+\frac 12[x,y]).
\]
The Lie algebra $\LL$ is isomorphic to $\Der(\LL)$ by Lemma~\ref{lem:liead}\eqref{liead:iso},
hence it contains a grading derivation. Then we can apply Lemma~\ref{lem:esigma'-prod} to $\LL$.
It implies that the morphism of $k$\dash schemes
\[
e_\sigma \colon \GU_\sigma'\to \End(\LL)
\]
factors through $\GL(\LL)\subseteq \End(\LL)$, and the induced homomorphism of $k$\dash groups
\[
e_\sigma \colon \GU_\sigma'\to\GL(\LL)
\]
is a closed embedding. Direct computation shows that
\[
\Lie(e_\sigma(\GU_\sigma'))=\LL_\sigma\oplus\LL_{2\sigma}
\]
as a Lie subalgebra of
\[
\LL=\Der(\LL)\subseteq\End(\LL)=\Lie(\GL(\LL)).
\]
Thus, $\Lie(\GU_\sigma)=\Lie(e_\sigma(\GU_\sigma'))$. Note that $\GS\cong\Gm$ induces a
9-grading on $\End(\LL)=\Lie(\GL(\LL))$ extending the grading on $\LL$, given by
\[ \End(\LL)_i := \{ \phi\in\End(\LL) \mid \phi(\LL_j)\subseteq\LL_{j+i} \ \text{ for all } -2 \leq j\leq 2 \} \]
for each $-4\leq i\leq 4$.
Clearly, both $\GU_\sigma$ and $e_\sigma(\GU_\sigma')$ are $\GS$-invariant.
Then by Lemma~\ref{lem:H1H2} we have $e_\sigma(\GU_\sigma')=\GU_\sigma$ as closed $k$\dash subgroups of $\GL(\LL)$.
In particular, each $e_\si(x,s)$ belongs to $\Aut(\LL)(R)$, and hence $\LL$ is algebraic.
\end{proof}

The following lemma is an easy corollary of the previous one combined with Theorem~\ref{thm:AGadjoint}.

\begin{lemma}\label{lem:Ch-alg}
Let $\A$ be a central simple structurable algebra over a field $k$ of characteristic $\neq 2,3$.
Let $\bar k$ be the algebraic closure of $k$.
Then $\A$ is algebraic if and only if
the simple Lie algebra $K(\A)\otimes_k\bar k$ is isomorphic to a simple Lie algebra over $\bar k$ of Chevalley
type (i.e. the commutator
Lie subalgebra of the Lie algebra of a split adjoint algebraic $\bar k$-group, cf. Lemma~\ref{lem:liead}).
\end{lemma}
\begin{proof}
If $\A$ is algebraic, the Lie algebra $K(\A)\otimes_k\bar k$ has the desired form by Theorem~\ref{thm:AGadjoint}.
Assume that $K(\A)\otimes_k\bar k$ is isomorphic a simple Lie algebra over $\bar k$ of Chevalley
type. Then by Lemma~\ref{lem:expad}(i) the Lie algebra $K(\A)\otimes_k\bar k$ is algebraic.
Since the natural inclusion $K(\A)\to K(\A)\otimes_k\bar k$ preserves the Lie bracket, it follows that $K(\A)$ is
algebraic, and hence $\A$ is algebraic.
\end{proof}

Relying on Lemma~\ref{lem:Ch-alg},
we can investigate the algebraicity of central simple
structurable algebras over a field of characteristic 5.
We approach this problem by relating algebraicity to the existence of absolute zero divisors. In the course of
the proof, we use the classification of simple Lie algebras over an algebraically closed field
of characteristic $\ge 5$~\cite{PrStr-V,PrStr-VI}.

\begin{definition}
	\begin{compactenum}[\rm (i)]
        \item
            Let $\A$ be a structurable algebra over a field $k$ of characteristic $\neq 2,3$.
            An element $x\in\A$ is called an \emph{absolute zero divisor} if $U_xy=0$ for any $y\in\A$.
            The algebra $\A$ is called \emph{non-degenerate} if it has no non-trivial absolute zero divisors.
        \item
            Let $\LL$ be a Lie algebra over a field $k$.
            An element $x\in\LL$ is called an \emph{absolute zero divisor} if $\ad_x^2=0$.
            The Lie algebra $\LL$ is called \emph{non-degenerate} if it has no non-trivial absolute zero divisors.
	\end{compactenum}
\end{definition}
If an element $x\in K(\A)_\si=\A_\si$ is an absolute zero divisor of $K(\A)$,
then it is represented by an absolute zero divisor of $\A$;
this follows from the fact that by Definition~\ref{def:Lie alg},
\[ [x_\sigma, [x_\sigma, y_{-\sigma}]] = -V_{x,y} x \in \A_\sigma \]
for all $x,y \in \A$.

\begin{lemma}\label{lem:azd-alg}
    Let $\A$ be a central simple structurable algebra over a field $k$ of characteristic $5$. Let
    $\bar k$ be an algebraic closure of $k$, and let $k^s\subseteq \bar k$ be a separable closure of $k$.
    Then the following conditions are equivalent.
	\begin{compactenum}[\rm (a)]
        \item\label{azd:alg} $\A$ is algebraic.
        \item\label{azd:bar} $\A\otimes_k \bar k$ is non-degenerate.
        \item\label{azd:sep} $\A\otimes_k k^s$ is non-degenerate.
        \item\label{azd:Kbar} $K(\A) \otimes_k \bar k$ is non-degenerate.
        \item\label{azd:Ksep} $K(\A) \otimes_k k^s$ is non-degenerate.
	\end{compactenum}
\end{lemma}
\begin{proof}
By the classification of simple finite-dimensional Lie algebras over an algebraically closed field of characteristic~$5$,
any such algebra is a Lie algebra of Chevalley (also called classical), Cartan, or Melikian type~\cite{PrStr-V,PrStr-VI}.
Among these Lie algebras, only the Lie algebras of Chevalley type are non-degenerate~\cite[p. 124]{Sel67},
while all simple Lie algebras of Cartan and Melikian type are degenerate
(this observation goes back to A. I. Kostrikin and A. Premet; see e.g.~\cite{Wil76} and~\cite[Corollary 12.4.7]{Strade-bookII}).
Since $K(\A)\otimes_k\bar k$ is a simple $\bar k$-Lie algebra, it follows from Lemma~\ref{lem:Ch-alg} that
\[ K(\A) \text{ is algebraic } \iff \A \text{ is algebraic} \iff K(\A)\otimes_k\bar k \text{ is non-degenerate} . \]
In particular, we have \eqref{azd:alg}~$\Leftrightarrow$~\eqref{azd:Kbar}.

Next, for any field extension $k'/k$, by~\cite[Corollary 2.5]{GGLN} $K(\A)\otimes_k k'$ is non-degenerate
if and only if the associated Kantor pair $(\A_-\otimes_k k',\A_+\otimes_k k')$ is non-degenerate.
The latter is by definition equivalent to the fact that the structurable algebra $\A\otimes_k k'$ is non-degenerate.
In particular, this shows that \eqref{azd:bar}~$\Leftrightarrow$~\eqref{azd:Kbar} and \eqref{azd:sep}~$\Leftrightarrow$~\eqref{azd:Ksep}.

Since the implication \eqref{azd:Kbar}~$\Rightarrow$~\eqref{azd:Ksep} is trivial,
it only remains to show \eqref{azd:Ksep}~$\Rightarrow$~\eqref{azd:Kbar}.
Pick $x\in K(\A)\otimes_k\bar k$ such that $\ad_x^2=0$.
Let $F$ denote the Frobenius automorphism of $\bar k$. It induces a $k$-linear automorphism
of $K(\A)\otimes_k\bar k$ and of $\End_{\bar k}(K(\A)\otimes_k\bar k)\cong\End_k(K(\A))\otimes_k\bar k$.
Since $\bar k/ k^s$ is purely inseparable, there is $n\ge 1$ such that
$F^n(x)\in K(\A)\otimes_k k^s$ inside $K(\A)\otimes_k\bar k$. Clearly, $F^n(\ad_x^2)=\ad_{F^n(x)}^2=0$.
Therefore, $F^n(x)=0$, and hence $x=0$.
\end{proof}

We can now apply our results to deduce that structurable {\em division} algebras are always algebraic.
\begin{theorem}\label{thm:div-alg}
Let $\A$ be a central simple structurable division algebra over a field $k$ of characteristic $\neq 2,3$.
Then $\A$ is algebraic.
\end{theorem}
\begin{proof}
By Lemma~\ref{lem:esigma'-alg} we only need to consider the case $\Char k=5$.
By Lemma~\ref{lem:azd-alg}, it is enough to show that $K(\A)\otimes_k k^s$ contains no non-trivial absolute zero divisors.
Assume the opposite. Then, clearly, there is a finite Galois extension $k'/k$ such that
$K(\A)\otimes_k k'$ contains non-trivial absolute zero divisors.
The argument in the second half of the proof of~\cite[p.~450, Theorem~2.4]{GGLN} now shows that there is, in fact, a non-trivial absolute zero divisor
$x\in (\A_-\oplus \A_+)\otimes_k k'$.
(Notice that the Lie algebra called $L_0$ in that proof is equal to $\bigl( K(\A)_{-2} \oplus K(\A)_0 \oplus K(\A)_2 \bigr) \otimes_k k'$,
and that the subspace called $L_1$ is equal to $\bigl( K(\A)_{-1} \oplus K(\A)_1 \bigr) \otimes_k k'$.)

Write $x = x_+ + x_-$, $x_+\in \A_+\otimes_k k'$, $x_-\in \A_-\otimes_k k'$.
Then for any homogeneous element $a \in K(\A)_i$, we rewrite $[x,[x,a]] = 0$ as
\begin{multline*}
    0 = [x_-, [x_-, a]] + \bigl( [x_-, [x_+, a]] + [x_+, [x_-, a]] \bigr) + [x_+, [x_+, a]] \\
        \in \bigl( K(\A)_{i-2} \oplus K(\A)_i \oplus K(\A)_{i+2} \bigr) \otimes_k k' ,
\end{multline*}
and we conclude that $x_+$ and $x_-$ are are absolute zero divisors of $K(\A)\otimes_k k'$.
In particular, $K(\A)\otimes_k k'$ contains a non-trivial absolute zero divisor of grading $\pm 1$.

Let $\DD_i$, $-2\le i\le 2$, denote the $k'$-linear span of absolute zero divisors in $K(\A)_i\otimes_k k'$.
By~\cite[Lemma 2.1(i)]{GGLN} the Lie bracket of two absolute zero divisors is again an absolute zero divisor,
hence $\DD=\sum_{i=-2}^2\DD_i$ is a $5$\dash graded Lie subalgebra of $K(\A)\otimes_k k'$. This graded
Lie subalgebra $\DD$ is invariant under all graded Lie algebra automorphisms of $K(\A)\otimes_k k'$, and
hence is defined over $k$ by Galois descent (see e.g.~\cite[\S 15.1]{Loos-book}). In other words,
there is a $5$-graded $k$-Lie subalgebra $\CC=\sum_{i=-2}^2\CC_i$ of $K(\A)$ such that
$\CC_i\otimes_k k'=\DD_i$ for each $-2\le i\le 2$.

The pair of $k$-vector spaces $(\CC_{-2},\CC_2)$ has a structure of a Jordan pair over $k$
induced by the Lie bracket in
$K(\A)$. Since $\A$ is a structurable division algebra, $(\CC_{-2},\CC_2)$ is a Jordan division pair in the
sense of~\cite{Loos-book}.
Therefore, its Jacobson radical $\Rad(\CC_{-2},\CC_2)$ equals $0$ by~\cite[Proposition 4.4]{Loos-book}.
Since $k'/k$ is separable, by~\cite[\S 15.2]{Loos-book} one has
$$
\Rad(\CC_{-2},\CC_2)\otimes_k {k'}=\Rad\bigl(\CC_{-2}\otimes_k k',\CC_2\otimes_k{k'}\bigr)=\Rad(\DD_{-2},\DD_2).
$$
On the other hand, the subspaces $\DD_{2\si}$ are spanned
by absolute zero divisors of $K(\A)\otimes_k k'$ contained in $K(\A)_{2\si}\otimes_k k'$,
and hence $(\DD_{-2},\DD_2)=\Rad(\DD_{-2},\DD_2)$.
Therefore, $\DD_{2\si}=0$, and hence $\CC_{2\si}=0$.

Since $\CC_{2\si}=0$, the pair of $k$-vector spaces $(\CC_{-1},\CC_1)$ also has
a structure of a Jordan pair over $k$ induced
by the Lie bracket of $K(\A)$. Again, it is a Jordan division pair, since $\A$ is division, and by the same token as above we
conclude that $\DD_\si=0$. However, this is not possible, since we have established that $K(\A)\otimes_k k'$
contains a non-trivial absolute zero divisor of grading $\pm 1$. We have arrived to a contradiction, and the theorem is proved.
\end{proof}

\section{Structurable division algebras from simple algebraic groups of $k$-rank~1}\label{ss:SSA-SAG}

In this section, we prove a partial converse to Theorem~\ref{thm:AGadjoint}.
Namely, we show that for any adjoint simple algebraic group $\GG$ of $k$\dash rank 1  over a field~$k$ of
characteristic $\neq 2,3$, there is a structurable division algebra $\A$ over $k$ such that
\[
\GG\cong\Aut(K(\A))^\circ.
\]
Moreover, such an algebra $\A$ is unique up to isotopy. Recall that by Theorem~\ref{thm:div-alg}
all structurable division algebras over $k$ are algebraic.

Our main result is the following theorem.

\begin{theorem}\label{thm:AG1}
Let $k$ be a field of characteristic $\neq 2,3$. The assignment
\[
\A\mapsto\GG=\Aut(K(\A))^\circ
\]
defines a one-to-one correspondence between isotopy classes of central simple structurable division algebras
$\A$ over $k$, and isomorphism classes of adjoint simple algebraic $k$\dash groups $\GG$ of $k$\dash rank 1.
\end{theorem}

In the proofs of Theorem~\ref{thm:AG1} and of the preliminary lemmas, we will extensively use notation of
Definition~\ref{def:rootsys} and Lemma~\ref{lem:5-parab}.

\begin{lemma}\label{lem:L'}
Let $k$ be a field, $\Char k\neq 2,3$. Let $\GG$ be an adjoint simple algebraic group over $k$.
Let $\LL=\Lie(\GG)$ be its Lie algebra, and let $\LL=\bigoplus\limits_{i=-2}^2\LL_i$ be any $5$\dash grading on $\LL$
such that $\LL_1\oplus\LL_{-1}\neq 0$. Then

\begin{compactenum}[\rm (i)]
\item\label{L':grade} $\LL'=[\LL,\LL]$ is a $\ZZ$-graded ideal of $\LL$ and $\LL'_i=\LL_i$ for all $i\neq 0$.
\item\label{L':i} $[\LL_\sigma,\LL_{\sigma}]=\LL_{2\sigma}$.
\item\label{L':ii} $[\LL_\sigma,\LL_{-\sigma}]=\LL'_0$.
\end{compactenum}
\end{lemma}
\begin{proof}
We can assume from the start that $k=\bar k$ is algebraically closed. We apply the results of
Lemma~\ref{lem:5-parab} in this situation. Then $\GG$ is a split group, $\GT$ is a split maximal torus of $\GG$, and
$\Phi=\Phi(\GT,\GG)$ is a reduced root system.

Since $\LL'$ is a characteristic ideal of $\LL$, it
is $\ZZ$-graded. By Lemma~\ref{lem:liead}, $\LL'$~has codimension $\leq 1$ in $\LL$.
Assume that  $M\subseteq\LL$ is a $1$\dash dimensional
graded complement to $\LL'$. Since $\Phi^+=-\Phi^-$, the equalities~\eqref{eq:lieP} then imply that
$\dim(\LL_1\oplus\LL_2)=\dim(\LL_{-1}\oplus\LL_{-2})$, and hence $M\subseteq\LL_0$. This proves~\eqref{L':grade}.

In order to prove~\eqref{L':i}, note that
we can choose basis elements $e_\alpha\in\LL_\alpha$, $\alpha\in\Phi$,
so that if $\alpha,\beta,\alpha+\beta\in\Phi$, then
\begin{equation}\label{eq:eaeb}
[e_\alpha,e_\beta]=\pm pe_{\alpha+\beta},
\end{equation}
where $p$ is the smallest positive integer such that $\beta-p\alpha\not\in\Phi$ (see
e.g.~\cite[Exp.\@~XXIII, Corollaire 6.5]{SGA3}). Since $p\in\{1,2,3\}$, we always have $p\in k^\times$.
Then by Lemma~\ref{lem:5-parab}\eqref{5-parab:height},
it is enough to show that for any $\alpha\in\Phi$ such that
$\sum_{\beta\in J}m_\beta(\alpha)=2$ there are $\gamma,\delta\in\Phi$ such that
$\sum_{\beta\in J}m_\beta(\gamma)=\sum_{\beta\in J}m_\beta(\delta)=1$ and $\gamma+\delta=\alpha$.
This is proved in~\cite[Lemma 2]{PS-elem}.

Now we prove~\eqref{L':ii}. By Lemma~\ref{lem:liead} $\LL'$  is a simple Lie algebra, therefore
\[
\LL'_0=[\LL_{-1},\LL_1]+[\LL_2,\LL_{-2}].
\]
Using~\eqref{L':i}, we deduce that $\LL'_0=[\LL_{-1},\LL_1]$, since
\begin{align*}
	\bigl[[\LL_1,\LL_1],[\LL_{-1},\LL_{-1}]\bigr] &\subseteq \bigl[\LL_1,[\LL_1,[\LL_{-1},\LL_{-1}]]\bigr]\\
	&\subseteq \bigl[\LL_1,[ [\LL_1,\LL_{-1}],\LL_{-1}] \bigr]\subseteq [\LL_1,[\LL_0,\LL_{-1}]]\subseteq [\LL_1,\LL_{-1}].
	\qedhere
\end{align*}
\end{proof}

\begin{lemma}\label{lem:Gkantor}
Let $k$ be a field of characteristic different from $2,3$. Let $\GG$ be an adjoint simple algebraic group over
$k$. Let $\LL=\Lie(\GG)$ be its Lie algebra, and let $\LL=\bigoplus\limits_{i=-2}^2\LL_i$ be any $5$\dash grading on
$\LL$ such that $\LL_1\oplus\LL_{-1}\neq 0$.
Let $\zeta\in\LL$ be the grading derivation. Then $(\LL_1,\LL_{-1})$ is a Kantor pair in the sense of
Definition~\ref{def:Kantor} with respect
to the triple product operation $\LL_\sigma\times \LL_{-\sigma}\times \LL_{\sigma}\to \LL_\sigma$
given by
\[
\lK x,y,z \rK=-[[x,y],z],
\]
and its standard $5$\dash graded Lie algebra $\G=\G(\LL_1\oplus\LL_{-1})$
is canonically isomorphic to the graded subalgebra $[\LL,\LL]+k\zeta$ of $\LL$.
\end{lemma}
\begin{proof}
By Theorem~\ref{thm:Kan-triple} in order to check that $(\LL_1,\LL_{-1})$ is a Kantor pair, it is enough to check that $\LL_1\oplus \LL_{-1}$
is sign-graded Lie triple system with the triple product given by~\eqref{eq:tripleprod}.
It is clear that $\LL_1\oplus \LL_{-1}$ considered as a $k$\dash subspace of the $5$\dash graded
Lie algebra $\LL$
is a sign-graded Lie triple system with respect to the triple product $[x,y,z]=[[x,y],z]$. Since this triple
product satisfies the equations~\eqref{eq:tripleprod}, we conclude that $(\LL_1,\LL_{-1})$ is indeed a Kantor pair.

Let $\G=\G(\LL_1\oplus \LL_{-1})$ be the $5$\dash graded Lie algebra which is the standard graded embedding of the Lie triple system
$\LL_1\oplus \LL_{-1}$ as described in~\S~\ref{se:oneinverse}. Let $\delta\in\G_0$ be the grading derivation of $\G$.
Note that there is a homomorphism of Lie algebras
\[
\phi \colon \LL\to \Der(\LL_1\oplus \LL_{-1})\oplus (\LL_1\oplus \LL_{-1})
\]
such that $\phi|_{\LL_1\oplus\LL_{-1}}=\id_{\LL_1\oplus \LL_{-1}}$, and for any
$x\in \LL_2\oplus\LL_0\oplus\LL_{-2}$ we have
\[
\phi(x)=\ad(x)|_{\LL_1\oplus\LL_{-1}}\in\Der(\LL_1\oplus \LL_{-1}).
\]
By Lemma~\ref{lem:liead} $[\LL,\LL]$ is simple, hence $\phi|_{[\LL,\LL]}$ is injective.
Then for any non-zero $x\in\LL$ and $u\in [\LL,\LL]$, the element
$[\phi(x),\phi(u)]=\phi([x,u])$ of $\phi([\LL,\LL])$ is non-zero as soon as $[x,u]$ is non-zero.
Since $\LL=\Der([\LL,\LL])$ by the same Lemma~\ref{lem:liead}, we conclude that $\phi$ is injective
on the whole of $\LL$.

Observe that $\delta$ and $\phi(\zeta)$ act on
$\LL_1\oplus\LL_{-1}$ in the same way, hence $\delta=\phi(\zeta)$.
By Lemma~\ref{lem:L'} $[\LL,\LL]$ is the subalgebra of $\LL$ generated by $\LL_1$ and $\LL_{-1}$.
Therefore,
\[
\G=k\phi(\zeta)+\phi([\LL,\LL]),
\]
and hence $\G$ is isomorphic to $k\zeta+[\LL,\LL]\subseteq\LL$.
\end{proof}

\begin{lemma}\label{lem:kantor-str}
In the setting of Lemma~\ref{lem:Gkantor}, assume moreover that $\GG$
has $k$\dash rank 1. Then there is a structurable division algebra $\A$ over $k$ such that the Kantor
pair $(\LL_1,\LL_{-1})$ is the Kantor pair associated with $\A$
in the sense of Remark~\ref{rem:AF99}.

\end{lemma}
\begin{proof}
We use the notation of Lemma~\ref{lem:Gkantor}.
By~\cite[Corollary 15 and p.541]{AF99} a Kantor pair is associated with a structurable algebra $\A$ if and only if
the standard $5$\dash graded Lie algebra $\G=\G(\LL_1\oplus\LL_{-1})$ of this Kantor pair contains an  element
$(x,0)\in\G_\sigma\oplus\G_{2\sigma}$ that is one-invertible, i.e. there are
$(y,t),(z,r)\in\G_{-\sigma}\oplus\G_{-2\sigma}$ such that the element
\begin{equation}\label{eq:kantor-n}
e_{-\sigma}(y,t)e_\sigma(x,0)e_{-\sigma}(z,r)=n\in\End(\G)
\end{equation}
satisfies $n(\zeta)=-\zeta$. By Lemma~\ref{le:1inv to div} in order to show that $\A$ is division,
it is enough to show that any element $(x,0)\in\G_\sigma\oplus\G_{2\sigma}$ has this property.
Note that by Lemma~\ref{lem:Gkantor} we can assume that
\[
\G=[\LL,\LL]+k\zeta\subseteq\LL.
\]
By Lemma~\ref{lem:L'} we have $\G_\sigma=\LL_\sigma$ and $\G_{2\sigma}=\LL_{2\sigma}$.

Let $\GP_\sigma$, $\GS$, $\GT$, $\Phi$, $\Pi$, $J$ etc.
be as in Lemma~\ref{lem:5-parab}. Since $\GG$ has $k$\dash rank 1, $\GS=\GT$ is a maximal split $k$\dash subtorus of $\GG$.
Then $\Phi$ is of type $A_1$ or $BC_1$, and $J=\Pi=\{\beta\}$.

By Lemma~\ref{lem:expad} for any $(x,s)\in\LL_\sigma\oplus\LL_{2\sigma}$
the map $e_\sigma(x,s)$, considered as an endomorphism of $\LL$, belongs to $\GU_\sigma(k)$, and,
conversely, any element of $\GU_\sigma(k)$ is of this form.
Then by the Bruhat decomposition~\cite[Th\'eor\`eme 5.15]{BorelTits}
there are $(y,t),(z,r)\in\LL_{-\sigma}\oplus\LL_{-2\sigma}$ such that
\[
e_{-\sigma}(y,t)e_\sigma(x,s)e_{-\sigma}(z,r)=n\in\Norm_{\GG}(\GS)(k),
\]
where $n$ acts on $\Phi$ as the Weyl reflection sending $\beta$ to $-\beta$ (the only non-trivial element
of the Weyl group). Then by Lemma~\ref{lem:5-parab}\eqref{5-parab:height} the adjoint action of $n$
on $\LL$ satisfies $n(\LL_i)=\LL_{-i}$ for any $-2\leq i\leq 2$. Consequently, $n(\zeta)=-\zeta$.
Since, clearly, $n$ preserves the Lie subalgebra $[\LL,\LL]$ of $\LL$,
it also preserves $\G\subseteq\LL$. This shows that the element $(x,s)$ of $\G$ is
one-invertible.
\end{proof}

\begin{proof}[Proof of Theorem~\ref{thm:AG1}]
The proof will be divided into three steps.
\begin{step} 
	If $\A$ is a central simple structurable division algebra over $k$,
	then $\GG=\Aut(K(\A))^\circ$ is an adjoint simple algebraic $k$\dash group of $k$\dash rank 1.
\end{step}
Set $\LL=\Lie(\GG)$. By Theorem~\ref{thm:AGadjoint}
the $k$\dash group $\GG$ is indeed an adjoint simple algebraic $k$\dash group, and $K(\A)=[\LL,\LL]$.

It remains to show that $\GG$ has $k$\dash rank 1. Let $\GS\subseteq\Aut(K(\A))$ be the
grading torus of $K(\A)$.
Since $\GS\cong\Gm$ is connected, we have $\GS\subseteq\GG$. Thus, the $k$\dash rank of $\GG$ is $\geq 1$.

Since $\A$
is algebraic, $e_\sigma(x,s)\in\GG(k)$ for any $(x,s)\in K(\A)_\sigma\oplus K(\A)_{2\sigma}$. Since $\A$
is division, Theorem~\ref{th:formula-oneinverse} implies that any $(x,s)\neq (0,0)$ is one-invertible. This
implies the existence of $n\in\GG$ such that $n(K(\A)_i)=K(\A)_{-i}$ for any $i\in\ZZ$. Consequently,
$n$ acts on $\GS$ by inversion.
Therefore, if we consider the $\ZZ$-grading on the whole of $\LL$ induced by the adjoint action of $\GS$, then
$n(\LL_i)=\LL_{-i}$ for any $i\in\ZZ$. Since $K(\A)=[\LL,\LL]$ has codimension at most 1 in $\LL$ by
Lemma~\ref{lem:liead}, its graded complement lies in degree $0$.

Let $\GT\subseteq\GG$ be a maximal split $k$\dash subtorus of $\GG$ containing $\GS$. Then
$\LL$ decomposes into a direct sum of root subspaces $\LL_\alpha$, $\alpha\in\Phi(\GT,\GG)$.
Since the graded complement of $K(\A)$ in $\LL$ lies in degree $0$, we have
$K(\A)_\alpha=\LL_\alpha$ for any $\alpha\neq 0$.
Then it is shown exactly as in~\cite[Theorem 3.1]{A86}, ``(ii)$\Rightarrow$(iii)'' that $\Phi(\GT,\GG)$ has rank 1
(this part of the proof does not use the $\Char k=0$ assumption included in the statement
of Allison's theorem). This implies that $\GG$ has $k$\dash rank 1.

\begin{step} 
	For any adjoint simple algebraic $k$\dash group $\GG$ of $k$\dash rank 1
	there is an algebraic central simple structurable division algebra $\A$ over $k$ such that
	$\GG\cong\Aut(K(\A))^\circ$.
\end{step}
Set $\LL=\Lie(\GG)$. Let $\GT$ be a maximal split subtorus
of $\GG$. Since $\GG$ has $k$\dash rank~$1$, $\GT$ is 1-dimensional, and the root system $\Phi=\Phi(\GT,\GG)$
has to be of type $A_1$ or $BC_1$. Let $\beta\in\Phi$ be a simple root. For any $-2\leq i\leq 2$ set
\begin{equation}
\LL_i=\LL_{i\beta}.
\end{equation}
Clearly, this is a $5$\dash grading on $\LL$ with $\LL_1\oplus\LL_{-1}\neq 0$.
By Lemma~\ref{lem:Gkantor} $(\LL_1,\LL_{-1})$ is a Kantor pair in the sense of~\cite{AF99} with respect
to the triple product operation $\LL_\sigma\times \LL_{-\sigma}\times \LL_{\sigma}\to \LL_\sigma$
given by
\[
\{x,y,z\}=-[[x,y],z],
\]
and the Lie subalgebra $[\LL,\LL]+k\zeta\subseteq\LL$, where $\zeta\in\LL$ is the grading derivation,
is the corresponding $5$\dash graded Lie algebra $\G(\LL_1\oplus\LL_{-1})$.
By Lemma~\ref{lem:kantor-str} there is a structurable division algebra $\A$ over $k$ such that
$(\LL_1,\LL_{-1})\cong(\A,\A)$ is the Kantor pair associated with $\A$ in the sense of Remark~\ref{rem:AF99}.
In particular, we have a natural isomorphism of graded Lie algebras
\[
[\LL,\LL]\cong K(\A).
\]
By Lemma~\ref{lem:liead} the Lie algebra $[\LL,\LL]$ is central simple, and
\[
\GG\cong\Aut([\LL,\LL])^\circ\cong\Aut(K(\A))^\circ.
\]
Since $K(\A)$ is central simple, by~\cite[Corollaries~6 and 9]{A79} the algebra $\A$ is central simple.
By Lemma~\ref{lem:expad}, $\A$ is algebraic.

\begin{step} 
	Isotopic structurable algebras give rise to isomorphic simple algebraic groups, and vice versa.
\end{step}
By Theorem~\ref{th:isotopic_isomor} structurable algebras $\A$ and $\A'$ over $k$ are isotopic
if and only if $K(\A)$ and $K(\A')$ are isomorphic as graded Lie algebras. This readily implies that
if $\A$ and $\A'$ are isotopic, then the $k$\dash groups $\GG=\Aut(K(\A))^\circ$ and $\GG'=\Aut(K(\A'))^\circ$
are isomorphic.

Conversely, assume that $\GG$ and $\GG'$ are two isomorphic adjoint simple $k$\dash groups. Without loss of generality,
we can assume that $\GG=\GG'$. By Theorem~\ref{thm:AGadjoint} we have
\[
K(\A)=K(\A')=[\Lie(\GG),\Lie(\GG)].
\]
Let $\GS$ (respectively, $\GS'$) be the
1-dimensional split $k$\dash subtorus of $\GG$ which is the grading torus of the $5$\dash graded algebra $K(\A)$
(respectively, $K(\A')$).
Since the $k$\dash rank of $\GG$ is 1,
both $\GS$ and $\GS'$ are maximal split subtori of $\GG$, and hence they
are conjugate by an element $g\in\GG(k)$. The adjoint action of $g$ on $\LL$ preserves
$[\LL,\LL]$, and transforms the canonical grading of $K(\A)$ into the canonical grading of $K(\A')$.
Therefore, $K(\A)$ and $K(\A')$ are isomorphic as graded Lie algebras. Then by Theorem~\ref{th:isotopic_isomor}
the algebras $\A$ and $\A'$ are isotopic.
\end{proof}

%
%
\chapter{Moufang sets and structurable division algebras}\label{se:MS-SDA}

We finally come to the main goal of the paper.
We will show that every structurable division algebra $\A$ over a field $k$ with $\Char(k) \neq 2,3$
gives rise to a Moufang set $\mouf(\A)$ (Theorem~\ref{mainth:moufset}),
and conversely, that every Moufang set arising from a linear algebraic $k$\dash group of $k$\dash rank $1$ is isomorphic to
$\mouf(\A)$ for some structurable division algebra $\A$ 
(Theorem~\ref{mainth:algmoufset}).
Recall that by Theorem~\ref{thm:div-alg}, every structurable division algebra $\A$ over a field $k$ with $\Char(k) \neq 2,3$ is algebraic.

\section{Moufang sets from structurable division algebras}

We will make the connection with Moufang sets through rank one groups; see Lemma~\ref{le:MSandR1G}.
We continue to use the terminology introduced in section~\ref{se:oneinverse}.

\begin{theorem}\label{th:Gisrank1group}
	Let $\A$ be a structurable division algebra over a field $k$ of characteristic different from $2$ and $3$.
	Then the elementary group $G$ of $\A$ is an abstract rank one group with unipotent subgroups $U_+$ and~$U_-$.
\end{theorem}
\begin{proof}
By definition, $G=\langle U_+,U_-\rangle$.
Next, we observe that $U_+$ and $U_-$ are nilpotent subgroups (of nilpotency class 2), since by Lemma \ref{le:esigma}
\[ \bigl[ [e_\si(x,s),e_\si(y,t)],e_\si(z,r) \bigr] = [e_\si(0,2\psi(x,y)),e_\si(z,r)] = 0 \]
for all $x,y,z\in\A$ and all $s,t,r\in\Ss$.

Now let $(x,s)\in \A\times\Ss\setminus\{(0,0)\}$, let $\si=\pm1$, and let $a=e_\si(x,s)\in U_\si\setminus\{0\}$.
By Corollary \ref{cor:oneinverse} $(x,s)$ is one-invertible; thus by Definition \ref{def:one inv} there exist unique elements $y,z\in \A, t\in \Ss$ such that
\begin{equation}\label{eq:forminv}
	h := e_{-\si}(z,t)\,e_\si(x,s)\,e_{-\si}(y,t)\in H_-.
\end{equation}
Define $b(a):=e_{-\si}(-y,-t) = h^{-1} e_{-\si}(z,t)a$.
Then using Lemma \ref{toevoeging}(i) we obtain
\[ U_\si^{b(a)}=U_{-\si}^{e_{-\si}(z,t)a} = (U_{-\si})^a. \]
By Definition \ref{def:rankonegroup}, this shows that $G$ is an abstract rank one group with unipotent subgroups $U_+$ and~$U_-$.
\end{proof}

Applying Lemma \ref{le:MSandR1G} to the abstract rank one group of the previous theorem, we obtain a Moufang set $\mouf$.
We will use Construction \ref{constr:MS} to give an explicit description of this Moufang set in the form $\mouf(\U, \tau)$ with $\tau$ a permutation of $\U \setminus \{ 0 \}$.

The groups $U_+$ and $U_-$ are root groups of the Moufang set we constructed.
We begin with defining an addition on $\A\times \Ss$ such that $\A\times\Ss$ is a group isomorphic to $U_+$ and $U_-$.
\begin{definition}\label{def:U=AxS}
	Let $\U:=\A\times \Ss$ be the (non-abelian) group with addition
	\[(x,s)+ (y,t)=(x+y,s+t+ \psi(x,y)).\]
	By Lemma~\ref{le:esigma}(ii), $\U\cong U_+\cong U_-$.
	We will also write $0$ for $(0,0) \in \U$, and we will use the notation $\U^*$ for $\U \setminus \{ 0 \}$.

	For each element $u = (x,s) \in \U$, we set
	\begin{align*}
		e_+(u) = e_+(x,s)  \quad \text{and} \quad
		e_-(u) = e_-(x,s).
	\end{align*}
\end{definition}

\begin{construction}\label{constr:moufstructalg}
\begin{compactenum}[(i)]
    \item
	Let $G=\langle U_+,U_-\rangle$ be the abstract rank one group from Theorem \ref{th:Gisrank1group}.
	Then the set $Y$ of the Moufang set $\mouf$ obtained from Lemma~\ref{le:MSandR1G} is given by
	\[ Y = \{ (U_-)^{e_+(u)} \mid u \in \U \} \cup \{ U_+ \} . \]
	We identify $Y$ with $X=\U\cup \{\infty\}$ through the map
	\begin{align*}
		(U_-)^{e_+(u)}&\longleftrightarrow u\\
		U_+&\longleftrightarrow \infty.
	\end{align*}
	The action of elements in $G = \langle U_+, U_- \rangle$ on $Y$ is given by conjugation,
	and this induces an equivariant action of $G$ on $X = \U \cup \{ \infty \}$.
	We denote
	\[ U_\infty := \infty= U_+ \quad \text{and} \quad U_0 :=0= U_- . \]
    \item
	Let $a\in \U$, then the unique element in $U_\infty$ mapping $0$ to $a$ is given by $\alpha_a=e_+(a)$.
	Indeed, $e_+(a)\in U_\infty$ and
	\[(0)e_+(a)=(U_-)^{e_+(0)e_+(a)}=(U_-)^{e_+(a)}=a.\]
	It follows that for all $a,b\in \U$ we have $a+b=a\alpha_b$ and  $\U\cong U_\infty=\{\alpha_a\mid a\in \U\}$.
    \item
	For each $u = (x,s) \in \U^*$, we define $\mu_u$ to be the unique element in the double coset $U_0 \alpha_u U_0$ interchanging the elements $0$ and $\infty$ of $X$
	(see~\eqref{eq:mu} on page~\pageref{eq:mu}).
	By \eqref{eq:forminv} and Lemma \ref{toevoeging}(i), we have
	\begin{align}\label{mu} \mu_u = \mu_{(x,s)} = e_-(z,t) \, e_+(x,s) \, e_-(y,t)\in H_- , \end{align}
	where $(y,t)=(u-tx,t)$ and $(z,t)=(u+tx,t)$ are the left and right inverse of $(x,s)$, respectively.
    \item
	Let $e = (1,0) \in \U^*$, and define $\tau = \mu_e$.
	Then $U_0=U_\infty^\tau$.
\end{compactenum}
\end{construction}

As in (ii), we let $\alpha_u=e_+(u)$ for each $u\in \U$; then
\[ U_\infty = \{ \alpha_u \mid u \in \U \} \quad \text{and} \quad U_0 = \{ \alpha_u^\tau \mid u \in \U \} . \]
Our next goal is to describe the action of $\tau$ on $\U^*$ explicitly.
First we determine the action of $\tau$ on the Lie algebra $\G$.

\begin{lemma}\label{le:e-alpha}
	The automorphism $\tau=\mu_{(1,0)}\in H_-$ is equal to the ``gradation flipping'' automorphism $\varphi$ defined in Definition~\ref{def:G}.
	In particular, $\tau$ is an involution, and for each $u \in \U$, we have $e_-(u) = e_+(u)^\tau$.
\end{lemma}
\begin{proof}
	We first observe that the left and right inverse of $e = (1,0)$ are both equal to $e = (1,0)$ again,
	and hence
	\[ \tau = \mu_{(1,0)} = e_-(1,0) \, e_+(1,0) \, e_-(1,0)\in H_- . \]
	Since $\tau\in H_-$, we have $\zeta.\tau=-\zeta$.
	We now verify that $\tau$ maps every element $x \in \G_\pm$ to the corresponding element $x \in \G_\mp$.

	Indeed, by Theorem~\ref{th:1inv}(iii), we know that $\tau|_{\G_+} = P_{(1,0)}$ and $\tau|_{\G_-}=P_{(1,0)}$,
	which we can compute explicitly.
	We get, for all $a \in \G_\pm$,
	\begin{align*}
		P_{(1,0)}a
		&= U_1(a+\tfrac{2}{3}\psi(1,a)1)
		= U_1(a+\tfrac{2}{3}(\overline{a}-a)1) \\
		&= \tfrac{1}{3}U_1(a+2\overline{a})
		= \tfrac{1}{3}(2(\overline{a}+2a)-(a+2\overline{a}))
		= a \in \G_\mp .
	\end{align*}
    Since $\tau$ is an automorphism of $\G$, we find for all $s\in \G_{\pm2}$ that
     \begin{align*}
        s.\tau=\half \psi(s,1).\tau=-\tfrac{1}{4}[s,1].\tau=-\tfrac{1}{4}[s.\tau,1.\tau]=\half\psi(s,1)=s\in \G_{\mp2},
     \end{align*}
     and for all $a\in \G_+$, $b\in \G_-$ that
     \[ V_{a,b}.\tau=-\half [a,b].\tau=-\half [a.\tau,b.\tau]=\half [b,a]=-V_{b,a}.\]
     We conclude that $\tau=\varphi$, which is clearly an involution.
    Since $\tau\in H_-$, it follows from Lemma  \ref{toevoeging}(i) that
	\[ e_+(x,s)^\tau = e_-(x.\tau, s.\tau)=e_-(x,s) \]
	for all $(x,s) \in \U$.
\end{proof}

We can now determine the action of $\tau$ on $\U^*$ using Lemma \ref{compute tau}.

\begin{theorem}\label{prop:action tau}
	The map $\tau = \mu_{(1,0)}$ maps each element $(x,s) \in \U^*$ to $(-y,-t)$,
	where $(y,t)$ is the left inverse of $(x,s)$.
\end{theorem}
\begin{proof}
Let $u = (x,s) \in \U^*$ be arbitrary.
By \eqref{mu} and Lemma~\ref{le:e-alpha}, we have
\[ \mu_u = \mu_{(x,s)} = \alpha_{(z, t)}^\tau \, \alpha_{(x,s)} \, \alpha_{(y, t)}^\tau . \]
On the other hand, it follows from Lemma \ref{compute tau} that
\[\mu_u = \al^\tau_{(-u)\tau^{-1}} \, \al_u \, \al^\tau_{-(u\tau^{-1})} . \]
By the uniqueness of the $\mu$-maps (see \eqref{eq:mu}), we conclude that
\begin{align*}
	(x,s) . \tau^{-1} &= (-y, -t) .
\end{align*}
The lemma follows since $\tau$ is an involution.
\end{proof}

We can now prove the main result of this section.
\begin{theorem}\label{mainth:moufset}
	Let $\A$ be a structurable division algebra over a field of characteristic different from $2$ and $3$.
	Define the group $\U:=A\times S$ with addition
	\[ (x,s)+ (y,t)=(x+y,s+t+ \psi(x,y)). \]
	Let
	\[ q_x \colon \Ss\to\Ss \colon s\mapsto \frac{1}{6}\psi(x,U_x(sx)) , \]
	and define the permutation $\tau$ of $\U^*$ by
	\begin{align*}
	(x,0) &\ \mapsto\ (-\hat{x},0),\\
	(0,s) &\ \mapsto\ (0,-\hat{s}),\\
	(x,s) &\ \mapsto\ \Bigl( \hat{s}\big((q_{\hx}(s)+\hat{s})^\wedge\hx\big)+ \big(s+q_x(\hat{s})\big)^\wedge x,\ -\big(s+q_x(\hat{s})\big)^\wedge \Bigr).
	\end{align*}
	for all $0\neq x\in \A$ and $0\neq s\in \Ss$.
	Then $\mouf(\U,\tau)$ is a Moufang set, which is isomorphic to the Moufang set induced by the abstract rank one group from Theorem~\ref{th:Gisrank1group}.
	We will denote this Moufang set by $\mouf(\A)$.
\end{theorem}
\begin{proof}
Let $\mouf$ be the Moufang set induced by the abstract rank one group from Theorem \ref{th:Gisrank1group},
let $\U$ be as in Definition~\ref{def:U=AxS}, and let $\tau = \mu_{(1,0)}$ be as in Construction~\ref{constr:moufstructalg}(iv).
Then by Theorem \ref{prop:action tau}, we can explicitly compute~$\tau$, using the formulas for the left inverse of $(x,s)$ in Corollary \ref{cor:oneinverse}.
The formula we obtain is precisely the formula given in the statement of the theorem.
\end{proof}

As an easy corollary, we can show that these Moufang sets are {\em special} if and only if they have abelian root groups,
which confirms that \cite[Conjecture 7.2.1]{DS}, the ``special implies abelian conjecture'' for Moufang sets,
holds for the examples arising from structurable division algebras.
\begin{definition}
    A Moufang set $\mouf(U, \tau)$ is called {\em special} if $(-x).\tau = -(x.\tau)$ for all $x \in U^*$.
\end{definition}
\begin{corollary}
	Let $\A$ be a structurable division algebra over a field of characteristic different from $2$ and $3$,
    and let $\mouf = \mouf(\A)$ be the corresponding Moufang set.
    Then $\mouf$ is special if and only if it has abelian root groups, i.e., if and only if $\A$ is a Jordan division algebra.
\end{corollary}
\begin{proof}
    Observe that, by Lemma~\ref{compute tau} and~\eqref{mu} on page~\pageref{mu}, $\mouf$ is special if and only if
    the left inverse and the right inverse of each element coincide.

    Recall that $\A$ is a Jordan division algebra if and only if $\Ss = 0$.
    If this holds, then indeed the left and right inverse are both equal to the usual inverse in the Jordan algebra.

    Conversely, assume that $\Ss \neq 0$, and consider $(1,s) \in \A \times \Ss$ for some $s \neq 0$.
    Then by Theorem~\ref{th:1inv}\eqref{it:l-r-inv}, the left and right inverse of $(1,s)$ are $(u-t,t)$ and $(u+t,t)$, respectively,
    and since $t \neq 0$ by Lemma~\ref{le:x=1inv}, we conclude that $\mouf$ is not special.
\end{proof}

\begin{remark}\label{rem:isotopic}
If $\A$ and $\A'$ are isotopic structurable algebras over $k$, then by Theorem \ref{th:isotopic_isomor}, the Lie algebras $K(\A)$ and $K(\A')$ are graded isomorphic. It follows from Lemma \ref{le:MSisom} and Theorem \ref{th:Gisrank1group} that the Moufang sets constructed by Theorem \ref{mainth:moufset} from $\A$ and $\A'$ are isomorphic.
Up to a field isomorphism, the converse of this fact is also true; see Corollary~\ref{co:iso} below.
\end{remark}

We will now determine the Hua maps of the Moufang set $\mouf(\A)$; see Definition~\ref{def:hua}.
We get a (surprisingly) elegant expression.
\begin{theorem}\label{th:hua}
	Let $\A$ be a structurable division algebra over a field of characteristic different from $2$ and $3$,
	and let $\mouf(\A)$ be as in Theorem~\ref{mainth:moufset}.
	Then
	\begin{align*}
		(a,r).h_{(x,s)} &= \bigl( P_{(x,s)}a,\ q_x(r)+\psi(x,s(rx))-s(rs) \bigr) \\
		&= \bigl( P_{(x,s)}a,\ \half \psi(P_{(x,s)} r, P_{(x,s)} 1) \bigr).
	\end{align*}
	for all $(x,s)\in \U^*$ and all $(a,r)\in\U$.
\end{theorem}
\begin{proof}
Let $(x,s)\in \U^*$, let $(a,r)\in\U$, and let $(y,t)$ and $(z,t)$ denote the left and right inverse of $(x,s)$, respectively.
By definition, $h_{(x,s)}=\tau \mu_{(x,s)}\in \Aut(\G)$; notice that $h_{(x,s)}$ preserves the gradation of $\G$.
Hence
\[ e_+(a,r)^{h_{(x,s)}} = e_+ \bigl( a.h_{(x,s)},r.h_{(x,s)} \bigr) . \]
On the other hand, by Lemma \ref{le:e-alpha},
\[ e_+(a,r)^{h_{(x,s)}} = e_-(a,r)^{\mu_{(x,s)}} = e_+ \bigl( a.\mu_{(x,s)},r.\mu_{(x,s)} \bigr) . \]
Combining the last two equalities, we find that
\begin{equation}\label{eq:hua}
(a.r).h_{(x,s)} = \bigl( a.\mu_{(x,s)},r.\mu_{(x,s)} \bigr);
\end{equation}
notice that the elements $a$ and $r$ in the right hand side expression have to be interpreted as elements of $\G_-$ and $\G_{-2}$, respectively
(but $x\in \G_+$ and $s\in \G_{+2}$).

By Construction \ref{constr:moufstructalg}(iii), we have $\mu_{(x,s)} = e_-(z,t) \, e_+(x,s) \, e_-(y,t)$.
By Theorem \ref{th:1inv}(iii), we have $\mu_{(x,s)}|_{\G_-}=P_{(x,s)}$.
Hence in \eqref{eq:hua} we find $a.\mu_{(x,s)}=P_{(x,s)}a$.
Since $\mu_{(x,s)}$ is a Lie algebra homomorphism, we have
\[r.\mu_{(x,s)}=\half[r,1].\mu_{(x,s)}=\half [r.\mu_{(x,s)},1.\mu_{(x,s)}]=\half \psi(P_{(x,s)} r, P_{(x,s)} 1),\]
proving the second formula in the statement of the theorem.

Using the theory developed in \cite{AF99}, we can obtain another equivalent formula.
Indeed, in the proof of \cite[Theorem 12]{AF99} it shown that for $h=\mu_{(x,s)}=e_-(z,t)e_+(x,s)e_-(y,t)$,
we have $h|_{\G_{-2}}=\eps_2e_+(x,s)|_{\G_{-2}}$, where $\eps_2$ denotes the projection $\G\to \G_{-2}$.
In the proof of \cite[Theorem 13]{AF99} it is shown that\footnote{This can be easily verified from the definition of $e_+(x,s)$.}
\[ \eps_2e_+(x,s)|_{\G_{-2}} = \tfrac{1}{24} \ad(x)^4 + \tfrac{1}{2} \ad(x)^2 \ad(s) + \half \ad(s)^2 . \]
Using the definition of the Lie bracket of $\G$, we now find
\[ r.\mu_{(x,s)} = \tfrac{1}{6}\psi(x,U_x(rx))+\psi(x,s(rx))-s(rs), \]
proving the first formula in the statement of the theorem.
\end{proof}

\section{Structurable division algebras from algebraic Moufang sets}

In this section, we will verify that every Moufang set arising from a linear algebraic group of $k$\dash rank one, as described in Theorem~\ref{th:M(G)},
is indeed isomorphic to a Moufang set $\mouf(\A)$ for some structurable division algebra $\A$, as described in Theorem~\ref{mainth:moufset}.
Due to the work we have done in section~\ref{ss:SSA-SAG}, this will only be a matter of combining some of our earlier results.

\begin{theorem}\label{mainth:algmoufset}
	Let $k$ be a field with $\Char(k) \neq 2,3$, and let
	$\GG$ be a semisimple linear algebraic group of $k$\dash rank one. There is a finite separable field
	extension $l/k$ and a central simple structurable division algebra $\A$ over $l$ such that
	$\mouf(\GG) \cong \mouf(\A)$.
\end{theorem}

\begin{lemma}\label{lem:GG-G}
Let $\A$ be an algebraic simple structurable algebra over a field~$k$ of characteristic $\neq 2,3$.
Let $\GG=\Aut(K(\A))^\circ$ be the adjoint simple algebraic $k$\dash group of Theorem~\ref{thm:AGadjoint}.
Let $\G$ be the $5$\dash graded Lie algebra of Definition~\ref{def:KacA}\eqref{KacA:frakg}, and let $G=\left<U_+,U_-\right>$ be the elementary
group of Definition~\ref{def:G}.
Then there exist

\begin{compactenum}[\rm (i)]
\item natural inclusions $K(\A)\subseteq\G\subseteq\Lie(\GG)$ of $5$\dash graded $k$\dash Lie algebras;
\item a $k$\dash group isomorphism
\begin{equation}\label{eq:Autg}
\GG\cong \Aut(\G)^\circ
\end{equation}
induced by the restriction of the adjoint action of $\GG$ on $\Lie(\GG)$ to $\G$;
\item two opposite parabolic $k$\dash subgroups $\GP_\sigma$ of $\GG$ with unipotent radicals $\GU_\sigma$
such that the isomorphism~\eqref{eq:Autg} restricts to isomorphisms $\GU_\sigma(k)\cong U_\sigma$;
\item an isomorphism of groups $\left<\GU_+(k),\GU_-(k)\right>\cong G$ induced by~\eqref{eq:Autg}.
\end{compactenum}
\end{lemma}
\begin{proof}
Set $\LL=\Lie(\GG)$.
By Definition~\ref{def:KacA}\eqref{it:KK'}, the Lie algebras $K(\A)$ and $K'(\A)$ are graded-isomorphic. We
identify $K(\A)$ with a graded Lie subalgebra of $\G$ via this isomorphism. By Theorem~\ref{thm:AGadjoint} we have $K(\A)=[\LL,\LL]$
and by Corollary~\ref{cor:newl-K} we have
$\LL\cong\Der(K(\A))\cong\newl{K(\A)}$. Let $\LL=\bigoplus\limits_{i=-2}^2\LL_i$ be the $5$\dash grading on $\LL$
induced by the isomorphism with $\newl{K(\A)}$. Then $K(\A)$ is a $5$\dash graded subalgebra of $\LL$.
Let $\zeta\in\LL$ be the grading derivation. Then, clearly, $\G\cong K(\A)+k\zeta$ is naturally
isomorphic to a $5$\dash graded Lie subalgebra of $\LL$. This settles (i).

In order to prove (ii), note that by Lemma~\ref{lem:liead}(i) $K(\A)=[\LL,\LL]$ is of codimension $\le 1$ in $\LL$.
Therefore, we have $\G=K(\A)$ or $\G=\LL$. Then by Lemma~\ref{lem:liead}(ii) we have $\GG\cong\Aut(\G)^\circ$.

Now we prove (iii) and (iv). Since $K(\A)_\sigma\neq 0$, we have $\LL_1\oplus\LL_{-1}\neq 0$, and hence by Lemma~\ref{lem:5-parab} there
are two opposite parabolic $k$\dash subgroups $\GP_\sigma$ of $\GG$ with unipotent radicals $\GU_\sigma$,
such that
\[
\Lie(\GP_\sigma)=\LL_0 \oplus \LL_{\sigma} \oplus \LL_{2\sigma} \quad\mbox{and}\quad \Lie(\GU_\sigma)=\LL_{\sigma}\oplus\LL_{2\sigma}.
\]
By Lemma~\ref{lem:expad} we have
\[
\GU_\sigma(k)=\{e_\sigma(x,s)\ |\ (x,s)\in\LL_\sigma\oplus\LL_{2\sigma}\},
\]
where $e_\sigma(x,s)$ is considered as an element of $\End(\LL)$.
If $\G=\LL$, this means exactly that $U_\sigma=\GU_\sigma(k)$. If $\G=K(\A)$ is of codimension 1 in $\LL$,
note that $\G=K(\A)+k\zeta$ is invariant under $e_\sigma(x,s)\in\End(\LL)$. Then, since the restriction map
$r \colon \GG\cong\Aut(\LL)^\circ\to\Aut(\G)^\circ$ is an isomorphism, we conclude that $r$ maps
$\GU_\sigma(k)$ isomorphically onto $U_\sigma\le\Aut(\G)(k)$. Thus, (iii) is proved.
Now the claim of (iv) follows, since by definition $G$ is the subgroup of $\Aut(\G)(k)$ generated by $U_+$ and $U_-$.
\end{proof}

\begin{proof}[Proof of Theorem~\ref{mainth:algmoufset}]
The algebraic $k$\dash group $\GG'=\GG/\Cent(\GG)$ is an adjoint semisimple $k$\dash group, and the natural
projection $p \colon \GG\to\GG'$ is a central $k$\dash isogeny. By Lemma~\ref{lem:M(G)-isog} $\mouf(\GG)\cong\mouf(\GG')$.
Also by Lemma~\ref{lem:M(G)-isog} there is a finite separable field extension $l/k$ and an adjoint simple
algebraic $l$-group $\GH$ such that $\mouf(\GG')\cong\mouf(\GH)$. Thus, we can assume that
$\GG$ is an adjoint simple group over $k$ from the start.

By Theorem~\ref{thm:AG1} there is a central simple structurable division algebra $\A$ over~$k$ (unique up to isotopy)
such that $\GG\cong\Aut(K(\A))^\circ$. Let $\G$ be the $5$\dash graded Lie algebra of Definition~\ref{def:KacA}\eqref{KacA:frakg},
and let $G=\left<U_+,U_-\right>$ be the elementary group of Definition~\ref{def:G}.
By Lemma~\ref{lem:GG-G} there is an isomorphism of $k$\dash groups $\phi \colon \GG\xrightarrow{\sim}\Aut(\G)^\circ$, and
there is a pair of opposite parabolic $k$\dash subgroups $\GP_\sigma$ of $\GG$ with unipotent
radicals $\GU_\sigma$ such that $\phi$ induces isomorphisms $\GU_\sigma(k)\cong U_\sigma$.
By definition of $\mouf(\GG)$, the subgroup $\left<\GU_+(k),\GU_-(k)\right>$ of $\GG(k)$ is an abstract
rank one group corresponding to $\mouf(\GG)$.
By Theorem~\ref{th:Gisrank1group} the group $G=\left<U_+,U_-\right>$ is an abstract rank one group.
By definition, $\mouf(\A)$ is the Moufang set corresponding to $G$.
Clearly, $\phi$ induces an isomorphism of the abstract rank one groups
$\left<\GU_+(k),\GU_-(k)\right>$ and $G=\left<U_+,U_-\right>$ that preserves their unipotent subgroups.
Therefore, by Lemma~\ref{le:MSisom} the Moufang sets $\mouf(\GG)$ and $\mouf(\A)$ are isomorphic.
\end{proof}

\begin{notation}\label{not:lambda}
Let $\lambda \colon k\to k'$ be a homomorphism of fields. For any $k$\dash algebra $R$, we denote by $\prescript{\lambda}{} R$
the $k'$\dash algebra $R\otimes_k k'$ which is the scalar extension of $R$ by means of $\lambda$.
Note that
\[
R\mapsto \prescript{\lambda}{} R
\] is a covariant functor from the category of $k$\dash algebras to the category of
$k'$\dash algebras. Restricting this functor to the category of commutative $k$\dash algebras, we obtain a
covariant functor from the category of affine $k$\dash schemes to the category of affine
$k'$\dash schemes that we denote by $\lambda^*$ (the base change functor).
Note that if $A$ is the dual $k$\dash Hopf algebra of an algebraic $k$\dash group $\GG$, i.e. $\GG=\Spec A$, then
\[
\lambda^*(\GG)=\Spec(\prescript{\lambda}{} A).
\]
We also denote by $\lambda^* \colon \GG(R)\to \lambda^*(\GG)(\prescript{\lambda}{} R)$ the canonical
homomorphism of the groups of points.
\end{notation}

\begin{remark}
Observe that in the setting of Notation~\ref{not:lambda}, if
$k=k'$ and $\lambda \colon k\xrightarrow{\sim} k$ is a field automorphism, then we can identify
$\prescript{\lambda}{} R$ and $R$ as rings (via the canonical ring isomorphism),
but the scalar multiplication in $\prescript{\lambda}{} R$ is given by the formula
$t\cdot x=\lambda(t)x$, $t\in k$, $x\in R$.
\end{remark}

\begin{remark}\label{rem:lambdaG}
Let $\GG$ be a linear algebraic $k$\dash group together with a $k$\dash representation $\phi \colon \GG\to\GL_{n,k}$. Then
we can identify the group $\lambda^*(\GG)$ with the subgroup of $\GL_{n,k'}$ given by equations obtained
by applying $\lambda$ to the equations determining $\phi(\GG)$. In this setting, the abstract group homomorphism
$\lambda^* \colon \GG(k)\to\lambda^*(\GG)(k')$ is the map that applies $\lambda$ to each entry
of matrices representing the elements of $\GG(k)$.
\end{remark}

\begin{theorem}\label{thm:kk'GM}
Let $k$ and $k'$ be two fields of characteristic $\neq 2,3$.
Let $\GG$ be an adjoint absolutely simple algebraic $k$\dash group of $k$\dash rank 1, and
let $\GG'$ be an adjoint absolutely simple algebraic $k'$\dash group of $k'$\dash rank 1.
Let $X$ and $X'$ be the sets of proper parabolic subgroups of $\GG$ and $\GG'$ respectively.
Let $\phi \colon X\to X'$ be a bijection that induces an isomorphism
$\phi \colon \mouf(\GG)\xrightarrow{\sim}\mouf(\GG')$ of Moufang sets. Then
there are a unique field isomorphism $\lambda \colon k\xrightarrow{\sim}k'$, and an isomorphism of $k'$\dash groups
$f \colon \lambda^*(\GG)\xrightarrow{\sim}\GG'$ such that $f\circ\lambda^*|_{X}=\phi$.
\end{theorem}
\begin{proof}
Let $\GP$ and $\GQ$ be two opposite proper parabolic $k$\dash subgroups of $\GG$. Let $\GP'=\phi(\GP)$
and $\GQ'=\phi(\GQ')$ be the corresponding parabolic $k'$\dash subgroups of $\GG'$.
By Lemma~\ref{lem:M(G)-E} the groups
$\GG(k)^+=\left<\radu(\GP)(k),\radu(\GQ)(k)\right>$ and $\GG'(k')^+=\left<\radu(\GP')(k'),\radu(\GQ')(k')\right>$
are the little projective groups of $\mouf(\GG)$ and $\mouf(\GG')$.
Then by Lemma~\ref{le:MSisom}(iii) $\phi$ induces a group isomorphism $\phi^+ \colon \GG(k)^+\xrightarrow{\sim}\GG'(k)^+$.
We will show the existence of $\lambda$ and $f$ as in the claim of the theorem satisfying
$f\circ\lambda^*|_{\GG(k)^+}=\phi^+$.

Note that the fields $k$ and $k'$ are both infinite or both finite, since the group $\radu(\GP)(k)$ is
infinite or finite depending on whether $k$ is infinite or finite.

Assume first that both $k$ and $k'$ are infinite.
By~\cite[Th\'eor\`eme 8.11(i)]{BoTi-hom} the existence of the injective homomorphism $\phi^+ \colon \GG(k)^+\to\GG'(k')$
implies that there is a field isomorphism $\lambda \colon k\xrightarrow{\sim}k'$ and
a $k'$\dash isogeny $f \colon \lambda^*(\GG)\to\GG'$ such that $f\circ\lambda^*|_{\GG(k)^+}=\phi^+$.
By~\cite[(3.3) and Th\'eor\`eme 8.11(iii)]{BoTi-hom}, $\lambda$~is a central $k'$\dash isogeny, since $\Char k'\neq 2,3$.
It remains to note that in our particular case, both groups $\lambda^*(\GG)$ and $\GG'$ are adjoint and absolutely simple,
and therefore the $k'$\dash isogeny $f$ is an isomorphism.

From now on, we assume that both $k$ and $k'$ are finite.
In this case by Lang's theorem $\GG$ and $\GG'$ are quasi-split. Since $\GG$ and $\GG'$ are absolutely simple and have rank 1 over the
base field, this implies that $\GG$ and $\GG'$ have type $A_1$ or $\prescript{2}{}A_2$. Then, since the characteristic
of the base fields is different from $2,3$, the groups
$\GG(k)^+$ and $\GG'(k')^+$ are finite simple groups of Lie type of the form $\PSL_2(q)$ or
$\PSU_3(q^2)$, where $q=p^n$, $n\ge 1$, is a power of the characteristic of the respective base field. The assumption
on the characteristic also implies that there are no exceptional automorphisms between such finite simple
groups~\cite{Wilson},
therefore,  $\GG(k)^+\cong\GG'(k')^+\cong E$, where $E$ is one of the groups $\PSL_2(q)$ or $\PSU_3(q)$
for a fixed $q$. In particular, $k\cong k'\cong \FF_q$,
and both $\GG$ and $\GG'$ are $\FF_q$-isomorphic to the algebraic $\FF_q$-group $\GH$, where $\GH=\PGL_{2,\FF_q}$ or
$\GH=\PGU_{3,\FF_q}$, depending on the type of $E$.

We identify $k$ with $\FF_q$, $\GG$ with $\GH$ and $\GG(k)^+$ with $E$.
Note that any two pairs of opposite parabolic subgroups in $\GG$ are conjugate, therefore, we can assume
that $\GP$ and $\GQ$ are the two opposite Borel subgroups of upper triangular and lower triangular
matrices in the standard presentation of $\GH$ by equivalence classes of $2\times 2$ matrices with entries in $\FF_q$,
or, respectively, of $3\times 3$ matrices with entries in $\FF_{q^2}$.
Fix any isomorphism of $\FF_q$-groups $\alpha \colon \GH\xrightarrow{\sim}\GG'$ that takes the
pair $(\GP,\GQ)$ to the pair $(\GP',\GQ')$, and denote $\alpha^+=\alpha|_{E}$. Then $(\alpha^+)^{-1}\circ\phi^+\in\Aut(E)$. The group
$\Aut(E)$ is known by~\cite[3.2]{Ste-aut}; we consider two cases.

If  $\GH=\PGL_{2,\FF_q}$ and $E=\PSL_2(q)$, then
by~\cite[3.2]{Ste-aut} there is a unique field automorphism $\lambda\in\Aut(\FF_q)$ and
an element $g\in \GH(\FF_q)$ such that $(\alpha^+)^{-1}\circ\phi^+=g\circ\lambda^+$,
where $g$ acts on $E\le \GH(\FF_q)$ by conjugation and $\lambda^+$ acts on the $2\times 2$ matrices
representing $E=\PSL_2(q)$  by applying $\lambda$ to each entry.
Note that by Remark~\ref{rem:lambdaG} the $\FF_q$-group $\lambda^*(\GH)$ naturally identifies with
$\GH$, so that
$\lambda^+$ is the restriction of the canonical abstract group homomorphism
$\lambda^* \colon \GH(\FF_q)\to\lambda^*(\GH)(\FF_q)$. Consider the matrix class $g$ as an element of
$\lambda^*(\GH)(\FF_q)$ instead of $\GH(\FF_q)$; then the conjugation by $g$ induces an $\FF_q$-automorphism
of $\lambda^*(\GH)$. Summing up, we have $\phi^+=f\circ\lambda^*|_{E}$, where $f=\alpha\circ g \colon \lambda^*(\GH)\to\GG'$
is an isomorphism of $\FF_q$-groups.

Now assume that $\GH=\PGU_{3,\FF_q}$ and $E=\PSU_3(q)$.
Let $F \colon \FF_{q^2}\to\FF_{q^2}$ be the Frobenius automorphism. Then $F^n$, where $q=p^n$, is an involutory
automorphism of $\FF_q$.
Consider the $\FF_{q^2}$-automorphism $\tau$ of $\PGL_{3,\FF_{q^2}}$ induced by the automorphism
of $\GL_{3,\FF_{q^2}}$ that is the composition of the transpose, the inverse and conjugation by
the $3\times 3$ matrix $e_{13}-e_{22}+e_{31}$. We identify the group $\GH=\PGU_{3,\FF_q}$ with the subgroup of
$\PGL_{3,\FF_{q^2}}$ stabilized by the automorphism $\tau\circ F^n$.
Clearly, we have $\tau^2=\id$, $\tau\circ F=F\circ\tau$ and
$\tau(\PGU_{3,\FF_q})=F^n(\PGU_{3,\FF_q})=\PGU_{3,\FF_q}$.
Consequently, $F^n|_\GH=\tau|_\GH$ is an $\FF_q$-automorphism
of $\GH$. By~\cite[3.2]{Ste-aut} there is a unique field automorphism $\lambda\in\Aut(\FF_{q^2})$ and
an element $g\in \GH(\FF_q)$ such that $(\alpha^+)^{-1}\circ\phi^+=g\circ\lambda^+$,
where $g$ acts on $E\le \GH(\FF_q)$ by conjugation and $\lambda^+$ acts on the $3\times 3$ matrices
representing $E=\PSU_3(q)\le\PSL_3(q^2)$  by applying $\lambda$ to each entry.
One has $\lambda=F^l$, where $0 \le l< 2n$. Set $\mu=\lambda|_{\FF_q}=F^{\,l\!\mod n}|_{\FF_q}$.
  As in the previous case, the $\FF_{q^2}$-group $\lambda^*(\PGL_{3,\FF_{q^2}})$
naturally identifies with $\PGL_{3,\FF_{q^2}}$, so that the corresponding abstract group homomorphism
$\lambda^*$ restricts to $\lambda^+$. Since the natural group scheme homomorphism $\PGU_{3,\FF_q}\to\PGL_{2,\FF_{q^2}}$
is compatible with the inclusion of the respective base fields $\FF_q\to\FF_{q^2}$, the above
 identification restricts to an identification of the $\FF_q$-groups $\mu^*(\GH)$ and $\GH$.
If $l < n$, then $\lambda^+$ coincides with the restriction of
$\mu^* \colon \GH(\FF_q)\to \mu^*(\GH)(\FF_q)$ and we conclude that $\phi^+=(\alpha\circ g)\circ\mu^*|_E$
as in the previous case. If $n\le l<2n$, then $\lambda^+=F^n\circ \mu^*|_E$, and
hence $\phi^+=(\alpha\circ g\circ\tau|_\GH)\circ\mu^*|_E$, where $\alpha\circ g\circ\tau|_\GH$
is an $\FF_q$-isomorphism between $\mu^*(\GH)\cong\GH$ and $\GG'$.
\end{proof}

\begin{corollary}\label{co:iso}
Let $k$ and $k'$ be two fields of characteristic $\neq 2,3$. Let $\A$ and $\A'$ be two central simple
structurable division algebras over $k$ and $k'$ respectively. Then $\mouf(\A)\cong\mouf(\A')$ if and only
if there is an isomorphism of fields $\lambda \colon k\xrightarrow{\sim} k'$ such that $\prescript{\lambda}{}\A$ and $\A'$
are isotopic as structurable algebras over $k'$.
\end{corollary}
\begin{proof}
By Theorem~\ref{thm:AG1}, $\GG=\Aut(K(\A))^\circ$
and $\GG'=\Aut(K(\A'))^\circ$ are adjoint absolutely simple algebraic groups of relative rank 1 over $k$ and $k'$,
respectively. By Lemma~\ref{lem:GG-G}, the Moufang sets corresponding to $\GG$ and $\GG'$ are isomorphic
to the Moufang sets corresponding to $\A$ and $\A'$, respectively, since the corresponding
rank one groups coincide. Also, observe that for any field isomorphism $\lambda \colon k\xrightarrow{\sim}k'$,
there are isomorphisms of $k'$\dash groups
\[
\Aut(K(\prescript{\lambda}{}\A))^\circ\cong\Aut(\prescript{\lambda}{} K(\A))^\circ\cong\lambda^*\bigl(\Aut(K(\A))^\circ\bigr).
\]
Therefore, $\prescript{\lambda}{}\A$ belongs to the isotopy class of structurable algebras
corresponding to the $k'$\dash group $\lambda^*(\GG)$ in the sense of Theorem~\ref{thm:AG1}. Now the claim
of the corollary readily follows from Theorem~\ref{thm:kk'GM};
see also Remark~\ref{rem:isotopic}.
\end{proof}

%
%
\chapter{Examples}\label{se:examples}

We will now describe the Moufang sets for each of the classes of structurable algebras as described in section~\ref{se:struct-ex}.
In each case, we will also point out what the corresponding linear algebraic groups are.
We assume throughout this section that $(\A, \barop)$ is a structurable division algebra over a field $k$ with $\Char(k) \neq 2,3$.

\section{Associative algebras with involution}\label{ss:ex-assoc}

Notice that any associative algebra with involution is, in fact, a hermitian structurable algebra for which the hermitian space $W$ is trivial.
We will therefore include this case in section~\ref{ss:ex-herm} below; see Remark~\ref{rem:W=0}.

\section{Jordan algebras}

Recall that a structurable algebra $(\A, \barop)$ is a Jordan algebra if and only if the involution $\barop$ is trivial,
and that in this case $\A$ is a Jordan division algebra, with $\hx=x^{-1}$, where $x^{-1}$ is the Jordan inverse in $J$.
Then $\Ss=0$, and by Theorem~\ref{mainth:moufset}, $U$ is the additive group of $J$ and $x.\tau=-\hx=-x^{-1}$ for $x\in U^*$.
We thus recover the Moufang sets described in Theorem~\ref{th:jordan}.

The corresponding linear algebraic groups of $k$\dash rank one are precisely those for which the $k$\dash root groups are abelian;
recall in particular that $\mouf(J)$ is the Moufang set of an exceptional linear algebraic group of type $E_{7,1}^{78}$
if $J$ is an exceptional Jordan division algebra (see Example~\ref{ex:abmoufsets}\eqref{it:Albert}),
and that it arises from a classical linear algebraic group in all other cases.

\section{Hermitian structurable algebras}\label{ss:ex-herm}

Let $\A=E\oplus W$ be a structurable division algebra of hermitian type; see section \ref{hermitian}.
In particular, $E$ is an associative division algebra with involution $\barop$, and $W$ is a left $E$-module equipped with a hermitian form $h$,
such that $h(x,x)\neq 0,1$ for all $0\neq x\in W$.
Notice that we allow the case $W=0$, in which case $\A = E$, i.e.\@ $\A$ is itself an associative algebra with involution; see section~\ref{ss:ex-assoc}
and Remark~\ref{rem:W=0} below.

For the sake of readability, we introduce the notation
\[ \pi(e+w) = e\overline{e}-h(w,w)\in\mathcal{H}(E),\]
for all $e\in E$ and $w\in W$;
notice that this makes $\pi$ into a (hermitian) pseudo-quadratic form on $\A$.
By \eqref{inverse hermitian}, we have $\widehat{e+w} = \pi(e+w)^{-1} (e-w)$ and
\[ \pi(\widehat{x}) = \pi(x)^{-1} \]
for all $0\neq x=e+w\in \A$.
Using formulas \eqref{psi hermitian} and \eqref{V hermitian}, it is a straightforward calculation to verify that
\[q_{x}(s)=\pi(x)s\pi(x)\in\Ss\]
for all $x\in \A$ and all $s\in \Ss(E)$, where $q_x$ is as in Definition \ref{def:qx}.

Now we apply Theorem \ref{mainth:moufset}.
It follows that $U=\A\times \Ss(E)$ with the group operation (denoted by $+$) given by
\begin{multline}\label{summoufherm}
     (e_1+w_1,s_1)+(e_2+w_2,s_2)\\=\bigl(e_1+e_2+w_1+w_2,s_1+s_2+e_1\overline{e_2}-e_2\overline{e_1}+h(w_2,w_1)-h(w_1,w_2)\bigr).
\end{multline}
Next, for all $0\neq x = e+w \in \A$ and all $0\neq s \in \Ss$, we have $(x,s).\tau=(-u+tx,-t)$, where $(u,t)$ is as in Theorem \ref{th:formula-oneinverse}.
We simplify the expressions for $u$ and $t$ using the associativity of $E$ and the identity $(a+b)^{-1}a(a-b)^{-1}=(a-ba^{-1}b)^{-1}$:
\begin{multline*}
        t = (s+q_{x}(\hat{s}))^\wedge
        = \overline{(s+\pi(x)\overline{s}^{-1}\pi(x))}^{-1} \\
        = (-s+\pi(x)s^{-1}\pi(x))^{-1}
        = (\pi(x)+s)^{-1} s (\pi(x)-s)^{-1}.
\end{multline*}
The expression for $u$ simplifies to
\begin{align*}
        u &= -\hat{s}\bigl((\hat{s}+q_{\widehat{x}}(s))^\wedge \hat{x}\bigr) \\
        &= s^{-1}\bigl((s^{-1} - \pi(x)^{-1}s\pi(x)^{-1})^{-1}\pi(x)^{-1}(e-w)\bigr) \\
        &= \bigl( \pi(x)^{-1} - s \pi(x)^{-1}s \bigr)^{-1} (e-w) \\
        &= (\pi(x)+s)^{-1}\pi(x)(\pi(x)-s)^{-1} (e-w).
\end{align*}
Therefore
\begin{multline}\label{taustructherm}
        (x,s).\tau
        = \Bigl( -(\pi(x) + s)^{-1} e + (\pi(x) - s)^{-1} w , \\*
            - (\pi(x) + s)^{-1} s (\pi(x) - s)^{-1} \Bigr) .
\end{multline}
Notice that this formula remains valid if either $x=0$ or $s=0$.
Indeed, if $x=0$ then we get $(0,s).\tau = (0, s^{-1})$, and if $s=0$, we get
$(x,0).\tau = (-\pi(x)^{-1} e + \pi(x)^{-1} w, 0) = (-\widehat{e+w}, 0)$, which is consistent with the formulas for these cases in Theorem \ref{mainth:moufset}.

\begin{remark}\label{rem:W=0}
    If $W=0$, then $\A = E$, so $U = E \times \Ss(E)$ with group operation given by
    \[ (e_1, s_1) + (e_2, s_2) = (e_1 + e_2, s_1 + s_2 + e_1\overline{e_2}-e_2\overline{e_1}) , \]
    and also the formula for $\tau$ simplifies slightly, to
    \[ (e, s).\tau = \Bigl( -(e \overline{e} + s)^{-1} e , \ - (e \overline{e} + s)^{-1} s (e \overline{e} - s)^{-1} \Bigr) . \]
\end{remark}

\begin{remark}\label{rem:herm-isom}
    The Moufang sets arising from hermitian structurable algebras with non-trivial involution are precisely the Moufang sets of skew-hermitian type as introduced
    in Definition~\ref{def:mouf-skewherm} (up to isomorphism).
    The description given there looks much simpler than the description we just obtained.
    It requires some effort to produce an explicit isomorphism between the two descriptions, so we sketch the procedure.

    Let $D$ be a skew field with involution $\sigma$, let $V$ be a non-trivial right $D$\dash module,
    and let $h \colon V \times V \to D$ be an anisotropic skew-hermitian form on~$V$.
    Let $D_\sigma$ be the subspace of elements of $D$ fixed by $\sigma$, and
    let $\Ss$ be the subspace of skew elements of $D$ (i.e.\@ the elements negated by $\sigma$).
    Fix an arbitrary element $\xi \in V \setminus \{ 0 \}$, and let $s_0 := -2 h(\xi, \xi)^{-1} \in \Ss$.
    We define a new involution $\rho$ on $D$ given by $a^\rho := s_0 a^\sigma s_0^{-1}$ for all $a \in D$.
    Notice that left multiplication by $s_0$ induces a bijection from $D_\sigma$ to $\Ss_\rho(D) := \{ a \in D \mid a^\rho = -a \}$.

    We now make $V$ into a {\em left} $D$-module by defining $a * v := va^\rho$ for all $v \in V$ and all $a \in D$.
    The form
    \[ H := \tfrac{1}{2} s_0 h \colon V \times V \to D \]
    is now a hermitian form (for this left scalar multiplication) with respect to the involution $\rho$.

    We now identify $D$ with the subspace $D * \xi$ of $V$, so we get $V = D \oplus D^\perp$.
    We can now make $V$ into a structurable division algebra $\A$ of hermitian type, with the choice $E = D$ and $W = D^\perp$
    and with hermitian form $H$ restricted to $W \times W$.
    When we write elements $v,w \in V$ as $v = e + v' \in D \oplus W$ and $w = f + w' \in D \oplus W$, then we get
    \[ h(v,w) = -2 s_0^{-1} \bigl( ef^\rho - H(v',w') \bigr) . \]
    The isomorphism between the root group $U_h = \{ (v,a) \in V \times D \mid h(v,v) = a - a^\sigma \}$ and the root group $U_\A = \A \times \Ss_\rho(D)$
    is given by
    \[ \varphi \colon U_h \to U_\A \colon (v,a) \mapsto \bigl( v, s_0(a - \tfrac{1}{2}h(v,v)) \bigr) . \]
    It is now a matter of computation to verify that $\varphi$ is a group isomorphism,
    and that $\varphi$ transforms the map $\tau$ from Definition~\ref{def:mouf-skewherm} into the map $\tau$ from equation~\eqref{taustructherm}.
\end{remark}

\begin{corollary}\label{cor:classical}
    Every Moufang set arising from a hermitian structurable algebra with non-trivial involution is isomorphic to a Moufang set
    arising from a classical linear algebraic group of relative type $BC_1$ (i.e.\@ with non-abelian root groups).

    Conversely, every Moufang set arising from a classical linear algebraic group of relative type $BC_1$
    is isomorphic to a Moufang set arising from a hermitian structurable algebra with non-trivial involution.
\end{corollary}
\begin{proof}
    This now follows from Remark~\ref{rem:herm-isom} together with Corollary~\ref{cor:mouf-classical}.
    Notice that the root groups of a Moufang set arising from a structurable division algebra $(\A, \barop)$ are abelian
    if and only if the involution $\barop{}$ is trivial.
\end{proof}

\section{Structurable algebras of skew-dimension one}\label{ss:skewdim1}

We will now have a closer look at the exceptional groups for which the corresponding structurable division algebra has skew-dimension one.
From Table~\ref{ta:exc} on page~\pageref{ta:exc}, we know that these are the groups with Tits index
$\prescript{3,6}{}D_{4,1}^9$, $\prescript{2}{}E_{6,1}^{35}$, $E_{7,1}^{66}$ or $E_{8,1}^{133}$.
To the best of our knowledge, not all corresponding structurable division algebras have been explicitly constructed, so in most cases, we will have
to content ourselves with the description of the algebra after a suitable scalar extension.
We will only be able to give a complete answer for groups of type $\prescript{3,6}{}D_{4,1}^9$.

Some of these forms are given by the Cayley--Dickson process (see Definition~\ref{def:CD}), specifically of the form $\CD(J, \mu)$
where $J$ is a Jordan division algebra with a Jordan norm of degree $4$, and $\dim(J)=4, 10, 16, 28$.
(Notice that $\dim(J)-1=3, 9, 15, 27$.)
However, see Remark~\ref{rem:skewdim1}.

\subsection{Forms of type $^{3,6}D_{4,1}^9$}

As we will see, all forms of type $\prescript{3,6}{}D_{4,1}^9$ arise from the Cayley--Dickson process.
We first recall the definition of a {\em quartic Cayley algebra} from \cite[Section 9]{A90}.
\begin{definition}
	A {\em quartic Cayley algebra} is a simple structurable algebra of skew-dimension $1$, degree $4$, and dimension $8$.
	Equivalently, a structurable algebra is a quartic Cayley algebra if and only if it is isotopic
	to $\CD(\B, \mu)$ for some separable commutative associative algebra $\B$ of dimension~$4$ and some $\mu \in k^\times$;
	see Definition~\ref{def:CD}.
\end{definition}
Observe that a structurable algebra arising from a hermitian form over an associative algebra $E$ with involution $\sigma$ as in Definition~\ref{def:herm}
has skew-dimension $1$, degree $4$, and dimension $8$ if and only if $E$ is a quaternion algebra
and $\sigma$ is a non-standard involution (i.e.\@ $\mathcal{H}(E)$ is $3$-dimensional).
The following proposition determines when a quartic Cayley algebra is also of hermitian type.
\begin{proposition}\label{pr:QCA}
	Let $(\A, \barop)$ be a quartic Cayley algebra $\A = \CD(\B, \mu)$.
	The following are equivalent:
	\begin{compactenum}[\rm (a)]
	    \item
		$\A$ is isotopic to a structurable algebra of hermitian type;
	    \item
		$\A$ is isomorphic to a structurable algebra of hermitian type;
	    \item
		$\B$ contains a $2$-dimensional unital subalgebra.
	\end{compactenum}
\end{proposition}
\begin{proof}
	Since an isotope of a structurable algebra of hermitian type is again of hermitian type, (a) and (b) are equivalent.

	(b) $\Rightarrow$ (c). \quad
	By assumption, there is an isomorphism $\varphi \colon \A \to \A'$ for some hermitian structurable algebra $\A' = E \oplus W$
    arising from a hermitian form $h \colon W \times W \to E$; see Definition~\ref{def:herm}.
    As we observed, this implies that $E$ is a quaternion algebra with non-standard involution $\sigma$,
	and $W$ is a $1$\dash dimensional left $E$-module.
    In particular, $\A' = E \oplus W$ contains $\B' := \varphi(\B)$, which is a $4$\dash dimensional separable commutative associative unital subalgebra.
    Since $\B \subset \Herm(\A)$, we also have $\B' \subset \Herm(\A') = \Herm(E) \oplus W$.

    We will show that $\B' \cap E$ is a $2$-dimensional unital subalgebra of $\B'$.
    Let $a = e+w \in \B'$ be a generator for $\B'$ over $k$.
    In particular, $a$ is {\em associative}, i.e.\@ $a^2 a = a a^2$; see \cite[Section 7]{A90}.
    Explicitly expressing this equality in the algebra $\A'$ using the formulas in Definition~\ref{def:herm} gives
    $e h(w,w) = h(w,w) e$.
    (Notice that $e^\sigma = e$ because $a \in \Herm(\A')$.)
    Let $K$ be the unital subalgebra of $E$ generated over $k$ by $e$ and $h(w,w)$.
    Since $e$ and $h(w,w)$ commute, we have $\dim_k K \leq 2$.

    By \cite[Theorem 7.3]{A90}, $\B' = k[a] = k \oplus ka \oplus ka^2 \oplus ka^3$.
    Since
    \begin{align*}
        a^2 &= \bigl( e^2 + h(w,w) \bigr) + 2ew , \\
        a^3 &= \bigl( e^3 + 3e h(w,w) \bigr) + \bigl( 3e^2 + h(w,w)\bigr) w ,
    \end{align*}
    we see that $\B' = k[a] \subseteq K \oplus Kw \subset E \oplus W$.
    Since $\dim_k \B' = 4$, this now implies that $\dim_k K = 2$, so $\B' = K \oplus Kw$, and $K = \B' \cap E$ is a subalgebra of $\B'$.
    Hence $\varphi^{-1}(K)$ is a $2$-dimensional unital subalgebra of $\B$.

	(c) $\Rightarrow$ (b). \quad
    We now assume that $\B$ contains a $2$-dimensional unital subalgebra $K$.
    Let $b_0 \in \B$ be a generator for $\B$ over $K$ with $T_{\B/K}(b_0) = 0$; then $\delta := b_0^2 \in K$ and $\B = K \oplus b_0 K$.
    Moreover, $T_{\B/k}(b_0K) = 0$ since $T_{\B/k} = T_{K/k} \circ T_{\B/K}$.
    Also observe that $T_{\B/k}{\restriction_K} = 2 T_{K/k}$.

    Recall that $\A = \B \oplus s_0 \B$.
    Define
    \[ E := K \oplus s_0 K \subset \A \quad \text{and} \quad W := b_0 K \oplus s_0 \cdot b_0 K \subset \A , \]
    and notice that $\A = E \oplus W$.
    Let $\B_0$ be the subspace of trace zero elements of~$\B$, i.e.  $\B_0 := \{ b \in \B \mid T_{\B/k}(b) = 0 \}$.
    Also recall from Definition~\ref{def:CD} that the involution on $\A$ is given by $\overline{b_1 + s_0 b_2} = b_1 - s_0 b_2^\theta$
    where $b^\theta = -b + \tfrac{1}{2} T_{\B/k}(b)$ for all $b \in \B$.
    In particular, we see that $\Herm(\A) = \B \oplus s_0 \B_0$ and hence also $W \subset \Herm(\A)$.
    Moreover, if $t \in K$, then $t^\theta = -b + T_{K/k}(t) = t^\sigma$, where $\sigma$ is the unique non-trivial $k$\dash automorphism of $K$.
    In particular, the restriction of $\theta$ to $K$ is multiplicative.

    We claim that $E$ is a subalgebra of $\A$ which is a quaternion algebra, and that the involution of $\A$ restricts to a non-standard involution on $E$.
    Indeed, the first claim follows immediately from the multiplication formula
    \begin{equation}\label{eq:Emult}
        (t_1 + s_0t_2)(t_3 + s_0t_4) = (t_1t_3 + \mu t_2^\sigma t_4) + s_0( t_1^\sigma t_4 + t_2 t_3 )
    \end{equation}
    arising from the formula in Definition~\ref{def:CD} because $\theta$ is multiplicative.
    The second claim follows from the observation that the involution of $\A$ maps $t_1 + s_0 t_2$ to $t_1 - s_0 t_2^\sigma$,
    which is indeed a non-standard involution on $E$.

    We now want to make $W$ into a $1$-dimensional $E$-module. We define a left action $\bullet$ of $E$ on $W$ by the rule
    \[ (t_1 + s_0 t_2) \bullet (b_0 t_3 + s_0 \cdot b_0 t_4) := b_0 \bigl( t_1 t_3 - \mu t_2^\sigma t_4 \bigr)
        + s_0 \cdot b_0 \bigl( t_1^\sigma t_4 - t_2 t_3 \bigr) \]
    for all $t_1,t_2,t_3,t_4 \in K$.
    In particular, $(t_1 + s_0 t_2) \bullet b_0 = b_0 t_1 - s_0 \cdot b_0 t_2$.
    It is easily checked that $\bullet$ indeed makes $W$ into a left $E$-module w.r.t.\@ the multiplication~\eqref{eq:Emult}.

    Now let $h \colon W \times W \to E$ be the unique hermitian form on $W$ such that $h(b_0, b_0) = \delta$.
    By definition, this means that $h(e_1 \bullet b_0, e_2 \bullet b_0) = e_1 \delta \overline{e_2}$ for all $x_1,x_2 \in E$.
    It now only requires some calculations, using the formulas in Definition~\ref{def:CD}, to verify that
    \begin{align*}
        \overline{e+w} &= \overline{e}+w, \\
        (e_1+w_1)(e_2+w_2) &= (e_1e_2+h(w_2,w_1))+(e_2 \bullet w_1+\overline{e_1} \bullet w_2),
    \end{align*}
    for all $e,e_1,e_2 \in E $ and all $w,w_1,w_2\in W$.
    Comparing this with Definition~\ref{def:herm} shows that $\A$ is isomorphic to the hermitian structurable algebra corresponding to $h$.
\end{proof}

We are now ready to describe the Moufang sets arising from groups of trialitarian type $D_4$.
\begin{theorem}\label{th:mouf-D4}
	Let $\B/k$ be a $4$-dimensional separable quartic field extension whose splitting field has Galois group $\Alt_4$ or $\Sym_4$,
	and let $\mu \in k^\times \setminus N_{\B/k}(\B) (k^\times)^2$.
	Then $\A = \CD(\B, \mu)$ is a structurable division algebra,
	and $\mouf(\A)$ is a Moufang set arising from a linear algebraic group of type $\prescript{3}{}D_{4,1}^9$ or $\prescript{6}{}D_{4,1}^9$,
	depending on whether the Galois group is $\Alt_4$ or $\Sym_4$, respectively.
	Moreover, every Moufang set arising from a linear algebraic group of type $\prescript{3,6}{}D_{4,1}^9$ can be obtained in this fashion.
\end{theorem}
\begin{proof}
	For fields of characteristic zero, this follows from~\cite[Theorem~9.4]{A90-D4} in the context of Lie algebras.
	An approach for arbitrary fields of characteristic $\neq 2$ in the context of algebraic groups was obtained in \cite{Gar-D4},
	but without the explicit description of the corresponding structurable algebras.

	Assume that $\B/k$ is a $4$-dimensional separable quartic field extension whose splitting field has Galois group $\Alt_4$ or $\Sym_4$,
	let $\mu \in k^\times \setminus N_{\B/k}(\B) (k^\times)^2$, and let $\A = \CD(\B, \mu)$.
	Observe that by the fundamental theorem of Galois theory, the condition on the Galois group implies that $\B/k$ does not admit quadratic subfields.
	Therefore, \cite[Theorem~6.3]{A86} tells us that the condition on $\mu$ implies that $\A$ is a division algebra.
	By~\cite[Theorem~9.1]{A90}, $\A$ is a quartic Cayley division algebra.

    By Proposition~\ref{pr:QCA}, $\A$ is not isotopic to a hermitian structurable algebra.
    Therefore, by Corollary~\ref{cor:classical}, the Moufang set $\mouf(\A)$ is
	a Moufang set arising from an exceptional algebraic group.
    Comparing dimensions using Table~\ref{ta:exc} on page~\pageref{ta:exc},
	we conclude that the algebraic group is necessarily of type $\prescript{3,6}{}D_{4,1}^9$.

	Conversely, assume that $\mouf(\A)$ is the Moufang set arising from an exceptional algebraic group of type $\prescript{3,6}{}D_{4,1}^9$.
	Then $\A$ is a structurable division algebra of dimension $8$ and skew-dimension $1$.
	By \cite[Proposition~4.4]{A90}, the degree of $\A$ is either $2$ or $4$.
	By \cite[Theorem~4.11(c)]{A90} however, the degree of $\A$ cannot be equal to $2$.
	By \cite[Theorem~9.1]{A90}, this implies that $\A$ is a quartic Cayley algebra.
	Since isotopic structurable algebras give rise to isomorphic Moufang sets,
	we may assume (by \cite[Theorem~9.1]{A90} again) that $\A \cong \CD(\B, \mu)$ for some separable commutative associative
	$4$\dash dimensional $k$\dash algebra $\B$ and some $\mu \in k^\times$.
	Since $\A$ is a division algebra, also $\B$ is a division algebra, i.e.\@ $\B/k$ is a separable quartic field extension,
	and by \cite[Theorem 6.3]{A86}, $\mu \not\in N_{\B/k}(\B) (k^\times)^2$.
    Moreover, Proposition~\ref{pr:QCA} implies that $\B$ does not admit quadratic subfields.

	It remains to show that the splitting field of $\B/k$ has Galois group $\Alt_4$ or $\Sym_4$.
    The only other possibilities for the Galois group are $\Cyc_4$ or $\Dih_8$.
    In both cases, however, $\B/k$ would admit quadratic subfields (by the fundamental theorem of Galois theory),
    which we had excluded.
\end{proof}

\subsection{Forms of type $^{2}E_{6,1}^{35}$, $E_{7,1}^{66}$ and $E_{8,1}^{133}$}

The structurable division algebras corresponding to forms of type $\prescript{2}{}E_{6,1}^{35}$, $E_{7,1}^{66}$ or $E_{8,1}^{133}$
are algebras of skew-dimension one, and of dimension $20$, $32$ and $56$, respectively.
Since not all corresponding structurable division algebras have been explicitly constructed,
we can currently do no better than the following result.
\begin{theorem}\label{th:mouf-E}
	Let $\A$ be a structurable division algebra over $k$, and assume that there is a quadratic field extension $E/k$
    such that $\A \otimes_k E$ is isomorphic to a matrix structurable algebra $M(J,\eta)$ as introduced in Definition~\ref{def:structmatalg},
    where $J$ is a Jordan algebra associated with a non-degenerate cubic norm structure of dimension $9$, $15$ or $27$, respectively.
    Then $\mouf(\A)$ is a Moufang set arising from a linear algebraic group of type $\prescript{2}{}E_{6,1}^{35}$, $E_{7,1}^{66}$ or $E_{8,1}^{133}$, respectively.

	Moreover, every Moufang set arising from a linear algebraic group of type $\prescript{2}{}E_{6,1}^{35}$, $E_{7,1}^{66}$ or $E_{8,1}^{133}$ can be obtained in this fashion.
\end{theorem}
\begin{proof}
	First assume that $\A$ is a structurable division algebra over $k$, and assume that there is a quadratic field extension $E/k$
    such that $\A \otimes_k E$ is isomorphic to a matrix structurable algebra $M(J,\eta)$
    where $J$ is a Jordan algebra associated with a non-degenerate cubic norm structure of dimension $9$, $15$ or $27$, respectively.
    By \cite[Lemmas 4.1 and 4.2]{A90}, $\A$ has degree $4$, and by \cite[Theorem 4.11]{A90}, this implies that $\A$ is not of hermitian type.
    Therefore, $\mouf(\A)$ is the Moufang set of a linear algebraic group of exceptional type.
    Since $\dim(\A)$ is equal to $20$, $32$ or $56$, respectively, an inspection of Table~\ref{ta:exc} reveals that
    we are in type $\prescript{2}{}E_{6,1}^{35}$, $E_{7,1}^{66}$ or $E_{8,1}^{133}$.

    Conversely, assume that $\mouf(\A)$ is the Moufang set of a linear algebraic group
    of type $\prescript{2}{}E_{6,1}^{35}$, $E_{7,1}^{66}$ or $E_{8,1}^{133}$.
    Then it follows again from Table~\ref{ta:exc} that $\A$ has skew-dimension $1$,
    and moreover $\A$ is not of hermitian type.
    It now follows from \cite[Theorem 4.11]{A90} again that $\A$ does not have degree $2$.
    By Corollary~\ref{cor:mat}, there exists a quadratic field extension $E/k$ such that $\A \otimes_k E$ is isomorphic to a matrix structurable algebra $M(J,\eta)$.
    Since $\deg(\A) \neq 2$, \cite[Lemma 4.2]{A90} now implies that $J$ has a non-zero norm, and by Definition~\ref{def:structmatalg},
    this means that $J$ is a Jordan algebra associated with a non-degenerate cubic norm structure.
    Comparing dimensions, we finally obtain that $J$ has dimension $9$, $15$ or $27$, respectively.
\end{proof}

\section{Forms of the tensor product of two composition algebras}

Let $(\A, \barop)$ be a structurable division algebra which is a form of a tensor product of two composition algebras; see section \ref{ex:compalg}.
If we assume moreover that $\A$ is not an associative algebra with involution, then by Proposition~\ref{pr:formtensor}, we know that either
$\A$ is an $(8,m)$-product algebra with $m \in \{ 1,2,4,8 \}$, or $\A$ is a twisted $(8,8)$-product algebra, and in this case there is a quadratic extension
$E/k$ such that $\A \otimes_k E$ is isomorphic to an $(8,8)$-product algebra.

These algebras correspond precisely to the rank one forms of the exceptional groups which have skew-dimension greater than $1$:
\begin{proposition}\label{prop:(8,m)}
    Let $X_0 = F_{4,1}^{21}$, $X_1 = \prescript{2}{}E_{6,1}^{29}$, $X_2 = E_{7,1}^{49}$ and $X_3 = E_{8,1}^{91}$.
    Let $\A$ be a structurable division algebra over $k$, and let $\GG = \Aut(K(\A))^\circ$ be the corresponding linear algebraic group of $k$\dash rank $1$,
    as in Theorem~\ref{thm:AGadjoint}.
    Then for each $i \in \{ 0,1,2,3 \}$, the following are equivalent.
    \begin{compactenum}[\rm (a)]
        \item
            $\GG$ has type $X_i$;
        \item
            $\A$ is an $(8,m)$-product algebra with $m = 2^i$, or a twisted $(8,8)$-product algebra (when $i=3$).
    \end{compactenum}
\end{proposition}
\begin{proof}
    Notice that both (a) and (b) imply that
    \begin{equation}\label{eq:dimCC}
        (\dim \A, \dim \Ss) \in \{ (8,7), (16,8), (32,10), (64,14) \} ,
    \end{equation}
    according to whether $i = 0,1,2,3$ respectively.
    In particular, $\dim \Ss > 1$.
    The result now follows since we already know that the classical linear algebraic groups with non-abelian root groups
    correspond precisely to the case where $\A$ is of hermitian type (including the associative algebras with involution).
\end{proof}
\begin{remark}
    Notice that the only case in which the dimensions~\eqref{eq:dimCC} do not uniquely determine the type of $G$ is the
    case $(\dim \A, \dim \Ss) = (32,10)$.
    Indeed, if $\A$ is of hermitian type, then $\A = E \oplus W$ and $\Ss = \Ss_E$, where $E$ is a central simple division algebra of degree $n$, and $W$ is an $E$-module,
    so $\dim A$ is divisible by $n^2$ and $\dim \Ss = n(n \pm 1)/2$ (where the sign depends on the type of the involution).
    Comparing this with~\eqref{eq:dimCC}, this only leaves the possibility $(\dim \A, \dim \Ss) = (32,10)$ and $n=4$.
\end{remark}

\begin{corollary}\label{co:tensor}
\begin{compactenum}[\rm (i)]
    \item
        Let $\A$ be an octonion division algebra over $k$ equipped with the standard involution.
        Then $\mouf(\A)$ is a Moufang set arising from a linear algebraic group of type $F_{4,1}^{21}$.
        Moreover, every Moufang set arising from a linear algebraic group of type $F_{4,1}^{21}$ can be obtained in this fashion.
    \item
        Let $E/k$ be a separable quadratic field extension with Galois involution~$\sigma$,
        let $\A$ be an octonion division algebra over $E$ with standard involution $\rho$,
        and equip $\A$ with the involution $\sigma \otimes \rho$ of the second kind over $E$.
        Then $\mouf(\A)$ is a Moufang set arising from a linear algebraic group of type $\prescript{2}{}E_{6,1}^{29}$.
        Moreover, every Moufang set arising from a linear algebraic group of type $\prescript{2}{}E_{6,1}^{29}$ can be obtained in this fashion.
    \item
        Let $\A$ be a division algebra which is the tensor product of an octonion division algebra and a quaternion division algebra.
        Then $\mouf(\A)$ is a Moufang set arising from a linear algebraic group of type $E_{7,1}^{49}$.
        Moreover, every Moufang set arising from a linear algebraic group of type $E_{7,1}^{49}$ can be obtained in this fashion.
    \item
        Let $\A$ be a division algebra which is the tensor product of two octonion division algebras,
        or which is isomorphic to such an algebra after extending scalars to some separable quadratic extension field $E/k$.
        Then $\mouf(\A)$ is a Moufang set arising from a linear algebraic group of type $E_{8,1}^{91}$.
        Moreover, every Moufang set arising from a linear algebraic group of type $E_{8,1}^{91}$ can be obtained in this fashion.
\end{compactenum}
\end{corollary}

\begin{remark}
    In cases (i) and (ii) from Corollary~\ref{co:tensor}, we recover the known results from \cite{DV2} and \cite{CD}, respectively.
    These results hold over fields of any characteristic, whereas our results only hold over fields of characteristic different from $2$ and $3$.
    It is very likely that also Corollary~\ref{co:tensor}(iii) and (iv) hold over fields of any characteristic.
\end{remark}

\section{Classification theorem for structurable division algebras}\label{se:classchar5}

\begin{theorem}
Let $\A$ be a central simple structurable division algebra over a field $k$ of characteristic $\neq 2,3$.
Then $\A$ belongs to one of the classes $(1)$--$(5)$ described in section~\ref{se:struct-ex}.
Namely, $\A$ is isomorphic to one of the following:
\begin{enumerate}[\rm (1)]
\item a central simple associative division algebra with involution;
\item a central simple Jordan division algebra;
\item a structurable algebra constructed from a non-degenerate anisotropic hermitian form
over a central simple associative division algebra with involution;
\item a central simple structurable division algebra of skew-dimension $1$;
\item a structurable division algebra that is a form of a tensor product of two composition algebras.
\end{enumerate}
\end{theorem}
\begin{proof}
By Theorem~\ref{thm:AG1} the algebraic $k$-group $\GG=\Aut(K(\A))^\circ$ is an adjoint simple algebraic group
of $k$-rank $1$.
In particular,
the relative type of $\GG$ is either $A_1$ or $BC_1$.

If $\GG$ has relative type $A_1$, then
the unipotent radicals of the parabolic subgroups of $\GG$ are abelian, and $K(\A)_{2\si}=0$.
Hence $\A$ is a Jordan division algebra.

Assume that $\GG$ has type $BC_1$.
By the classification of  Tits indices~\cite{Boulder,PS-tind} of simple algebraic groups,
either $\GG$ is of classical type, or its index is in Table~\ref{ta:exc}.
In the first case, $\A$ is a structurable algebra constructed from a non-degenerate anisotropic hermitian form
over a central simple associative division algebra with involution by Remark~\ref{rem:herm-isom}
and Corollary~\ref{cor:mouf-classical}. In the second case, we see from Table~\ref{ta:exc}
that either $\dim(Z(U))=1$, where $U$ is the root group of the corresponding Moufang set,
or $\GG$ is of type
$F_{4,2}^{21}$, $\prescript{2}{}E_{6,1}^{29}$, $E_{7,1}^{48}$ or $E_{8,1}^{91}$. Since
\[
\dim(Z(U))=\dim K(\A)_{2\si}=\dim\Ss
\]
by e.g.
Theorem~\ref{mainth:moufset}, the cases $\dim(Z(U))=1$ are exactly the ones where $\A$ is a structurable
algebra of skew-dimension $1$.
In the remaining four cases, $\A$ is a form of a tensor product of two composition algebras by Proposition~\ref{prop:(8,m)}.
\end{proof}

\backmatter

\bibliographystyle{amsalpha}
\bibliography{bibl3}

\providecommand{\bysame}{\leavevmode\hbox to3em{\hrulefill}\thinspace}
\providecommand{\MR}{\relax\ifhmode\unskip\space\fi MR }
\providecommand{\MRhref}[2]{%
  \href{http://www.ams.org/mathscinet-getitem?mr=#1}{#2}
}
\providecommand{\href}[2]{#2}
\begin{thebibliography}{DMVM10}

\bibitem[AF84]{AF84}
B.~N. Allison, J.~R. Faulkner, \emph{A {C}ayley-{D}ickson process for a class
  of structurable algebras}, Trans. Amer. Math. Soc. \textbf{283} (1984),
  no.~1, 185--210. \MR{735416 (85i:17001)}

\bibitem[AF92]{AF92}
B.~N. Allison, J.~R. Faulkner, \emph{Norms on structurable algebras}, Comm.
  Algebra \textbf{20} (1992), no.~1, 155--188. \MR{1145331 (93b:17005)}

\bibitem[AF93]{AF_35dim}
B.~N. Allison, J.~R. Faulkner, \emph{The algebra of symmetric octonion
  tensors}, Nova J. Algebra Geom. \textbf{2} (1993), no.~1, 47--57. \MR{1254152
  (95a:17003)}

\bibitem[AF99]{AF99}
B.~N. Allison, J.~R. Faulkner, \emph{Elementary groups and invertibility for
  {K}antor pairs}, Comm. Algebra \textbf{27} (1999), no.~2, 519--556.
  \MR{1671930 (99k:17002)}

\bibitem[AH81]{AH81}
B.~N. Allison, W.~Hein, \emph{Isotopes of some nonassociative algebras with
  involution}, J. Algebra \textbf{69} (1981), no.~1, 120--142. \MR{613862
  (82k:17013)}

\bibitem[All78]{A78}
B.~N. Allison, \emph{A class of nonassociative algebras with involution
  containing the class of {J}ordan algebras}, Math. Ann. \textbf{237} (1978),
  no.~2, 133--156. \MR{507909 (81h:17003)}

\bibitem[All79]{A79}
B.~N. Allison, \emph{Models of isotropic simple {L}ie algebras}, Comm. Algebra
  \textbf{7} (1979), no.~17, 1835--1875. \MR{547712 (81d:17005)}

\bibitem[All86a]{A86_2}
B.~N. Allison, \emph{Conjugate inversion and conjugate isotopes of alternative
  algebras with involution}, Algebras Groups Geom. \textbf{3} (1986), no.~3,
  361--385. \MR{900489 (88m:17027)}

\bibitem[All86b]{A86}
B.~N. Allison, \emph{Structurable division algebras and relative rank one
  simple {L}ie algebras}, Lie algebras and related topics ({W}indsor, {O}nt.,
  1984), CMS Conf. Proc., vol.~5, Amer. Math. Soc., Providence, RI, 1986,
  pp.~139--156. \MR{832197 (87j:17001)}

\bibitem[All88]{A88}
B.~N. Allison, \emph{Tensor products of composition algebras, {A}lbert forms
  and some exceptional simple {L}ie algebras}, Trans. Amer. Math. Soc.
  \textbf{306} (1988), no.~2, 667--695. \MR{933312 (89e:17004)}

\bibitem[All90a]{A90-D4}
B.~N. Allison, \emph{Isotropic simple {L}ie algebras of type {$D_4$}}, Lie
  algebra and related topics ({M}adison, {WI}, 1988), Contemp. Math., vol. 110,
  Amer. Math. Soc., Providence, RI, 1990, pp.~1--21. \MR{1079097 (92b:17009)}

\bibitem[All90b]{A90}
B.~N. Allison, \emph{Simple structurable algebras of skew-dimension one}, Comm.
  Algebra \textbf{18} (1990), no.~4, 1245--1279. \MR{1059949 (91h:17001)}

\bibitem[BDM13]{BD1}
Lien Boelaert, Tom De~Medts, \emph{Exceptional {M}oufang quadrangles and
  structurable algebras}, Proc. London Math. Soc. \textbf{107} (2013), no.~3,
  590--626, DOI: http://dx.doi.org/10.1112/plms/pds088.

\bibitem[Blo62]{block62}
Richard Block, \emph{Trace forms on {L}ie algebras}, Canad. J. Math.
  \textbf{14} (1962), 553--564. \MR{0140555 (25 \#3973)}

\bibitem[Boe13]{Bo-thes}
Lien Boelaert, \emph{From the {M}oufang world to the structurable world and
  back again}, {P}h. {D.} thesis, Ghent University, 2013.

\bibitem[Bor66]{Borel}
Armand Borel, \emph{Linear algebraic groups}, Algebraic {G}roups and
  {D}iscontinuous {S}ubgroups ({P}roc. {S}ympos. {P}ure {M}ath., {B}oulder,
  {C}olo., 1965), Amer. Math. Soc., Providence, R.I., 1966, pp.~3--19.
  \MR{0204532 (34 \#4371)}

\bibitem[Bor91]{Bo-book}
Armand Borel, \emph{Linear algebraic groups}, second ed., Graduate Texts in
  Mathematics, vol. 126, Springer-Verlag, New York, 1991. \MR{1102012
  (92d:20001)}

\bibitem[Bou81]{Bou}
Nicolas Bourbaki, \emph{\'{E}l\'ements de math\'ematique}, Masson, Paris, 1981,
  Groupes et alg{\`e}bres de Lie. Chapitres 4, 5 et 6. [Lie groups and Lie
  algebras. Chapters 4, 5 and 6]. \MR{647314 (83g:17001)}

\bibitem[Bro69]{Brown}
Robert~B. Brown, \emph{Groups of type {$E_{7}$}}, J. Reine Angew. Math.
  \textbf{236} (1969), 79--102. \MR{0248185}

\bibitem[BT65]{BorelTits}
Armand Borel, Jacques Tits, \emph{Groupes r\'eductifs}, Inst. Hautes \'Etudes
  Sci. Publ. Math. (1965), no.~27, 55--150. \MR{0207712 (34 \#7527)}

\bibitem[BT72a]{BorelTits-compl}
Armand Borel, Jacques Tits, \emph{Compl\'ements \`a l'article: ``{G}roupes
  r\'eductifs''}, Inst. Hautes \'Etudes Sci. Publ. Math. (1972), no.~41,
  253--276. \MR{0315007 (47 \#3556)}

\bibitem[BT72b]{BruhatTits}
F.~Bruhat, J.~Tits, \emph{Groupes r\'eductifs sur un corps local}, Inst. Hautes
  \'Etudes Sci. Publ. Math. (1972), no.~41, 5--251. \MR{0327923 (48 \#6265)}

\bibitem[BT73]{BoTi-hom}
Armand Borel, Jacques Tits, \emph{Homomorphismes ``abstraits'' de groupes
  alg\'ebriques simples}, Ann. of Math. (2) \textbf{97} (1973), 499--571.
  \MR{0316587 (47 \#5134)}

\bibitem[BT87]{BruhatTitsII}
F.~Bruhat, J.~Tits, \emph{Sch\'emas en groupes et immeubles des groupes
  classiques sur un corps local. {II}. {G}roupes unitaires}, Bull. Soc. Math.
  France \textbf{115} (1987), no.~2, 141--195. \MR{919421 (89b:20098)}

\bibitem[CDM14]{CD}
Elizabeth Callens, Tom De~Medts, \emph{Moufang sets arising from polarities of
  {M}oufang planes over octonion division algebras}, Manuscripta Math.
  \textbf{143} (2014), no.~1-2, 171--189. \MR{3147447}

\bibitem[CHM08]{CHM}
Arjeh~M. Cohen, Sergei Haller, Scott~H. Murray, \emph{Computing in unipotent
  and reductive algebraic groups}, LMS J. Comput. Math. \textbf{11} (2008),
  343--366. \MR{2452553 (2010a:20096)}

\bibitem[DG70a]{SGA3}
M.~Demazure, A.~Grothendieck, \emph{Sch{\'e}mas en groupes ({SGA} 3)}, Lecture
  Notes in Mathematics, vol. 151--153, Springer-Verlag, Berlin-Heidelberg-New
  York, 1970.

\bibitem[DG70b]{DeGa}
Michel Demazure, Pierre Gabriel, \emph{Groupes alg\'ebriques. {T}ome {I}:
  {G}\'eom\'etrie alg\'ebrique, g\'en\'eralit\'es, groupes commutatifs}, Masson
  \& Cie, \'Editeur, Paris; North-Holland Publishing Co., Amsterdam, 1970, Avec
  un appendice {{\i}t Corps de classes local} par Michiel Hazewinkel.
  \MR{0302656 (46 \#1800)}

\bibitem[DM06]{DMnota}
T.~De~Medts, \emph{{M}oufang sets arising from {M}oufang polygons of type
  ${E}_6$ and ${E}_7$}, 2006, Available at
  cage.ugent.be/~tdemedts/notes/E6E7.pdf.

\bibitem[DMS09]{DS}
T.~De~Medts, Y.~Segev, \emph{A course on {M}oufang sets}, Innov. Incidence
  Geom. \textbf{9} (2009), 79--122. \MR{2658895 (2011h:20058)}

\bibitem[DMVM10]{DV2}
T.~De~Medts, H.~Van~Maldeghem, \emph{Moufang sets of type {$F_4$}}, Math. Z.
  \textbf{265} (2010), no.~3, 511--527. \MR{2644307 (2011d:20056)}

\bibitem[DMW06]{DW}
T.~De~Medts, R.~M. Weiss, \emph{Moufang sets and {J}ordan division algebras},
  Math. Ann. \textbf{335} (2006), no.~2, 415--433. \MR{2221120 (2007e:17027)}

\bibitem[Fau00]{Fau}
John~R. Faulkner, \emph{Jordan pairs and {H}opf algebras}, J. Algebra
  \textbf{232} (2000), no.~1, 152--196. \MR{1783919 (2001e:17042)}

\bibitem[Fau04]{Fau2}
John~R. Faulkner, \emph{Hopf duals, algebraic groups, and {J}ordan pairs}, J.
  Algebra \textbf{279} (2004), no.~1, 91--120. \MR{2078388 (2005d:16062)}

\bibitem[Gar98]{Gar-D4}
R.~Skip Garibaldi, \emph{Isotropic trialitarian algebraic groups}, J. Algebra
  \textbf{210} (1998), no.~2, 385--418. \MR{1662339 (99k:20092)}

\bibitem[Gar01]{Gar-stralg}
R.~Skip Garibaldi, \emph{Structurable algebras and groups of type {$E_6$} and
  {$E_7$}}, J. Algebra \textbf{236} (2001), no.~2, 651--691. \MR{1813495}

\bibitem[GGLN11]{GGLN}
Esther Garc{\'{\i}}a, Miguel G{\'o}mez~Lozano, Erhard Neher,
  \emph{Nondegeneracy for {L}ie triple systems and {K}antor pairs}, Canad.
  Math. Bull. \textbf{54} (2011), no.~3, 442--455. \MR{2857919 (2012m:17007)}

\bibitem[HT98]{HT}
D.~W. Hoffmann, J.-P. Tignol, \emph{On {$14$}-dimensional quadratic forms in
  {$I^3$}, {$8$}-dimensional forms in {$I^2$}, and the common value property},
  Doc. Math. \textbf{3} (1998), 189--214 (electronic). \MR{1647511
  (2000b:11034)}

\bibitem[Jac68]{J_Jordan}
N.~Jacobson, \emph{Structure and representations of {J}ordan algebras},
  American Mathematical Society Colloquium Publications, Vol. XXXIX, American
  Mathematical Society, Providence, R.I., 1968. \MR{0251099 (40 \#4330)}

\bibitem[Jan03]{Jantzen}
Jens~Carsten Jantzen, \emph{Representations of algebraic groups}, second ed.,
  Mathematical Surveys and Monographs, vol. 107, American Mathematical Society,
  Providence, RI, 2003. \MR{2015057 (2004h:20061)}

\bibitem[KMRT98]{KMRT}
M.~Knus, A.~Merkurjev, M.~Rost, J.-P. Tignol, \emph{The book of involutions},
  American Mathematical Society Colloquium Publications, vol.~44, American
  Mathematical Society, Providence, RI, 1998, With a preface in French by J.
  Tits. \MR{1632779 (2000a:16031)}

\bibitem[Kru07]{Krut}
Sergei Krutelevich, \emph{Jordan algebras, exceptional groups, and {B}hargava
  composition}, J. Algebra \textbf{314} (2007), no.~2, 924--977. \MR{2344592}

\bibitem[Lam05]{L}
T.~Y. Lam, \emph{Introduction to quadratic forms over fields}, Graduate Studies
  in Mathematics, vol.~67, American Mathematical Society, Providence, RI, 2005.
  \MR{2104929 (2005h:11075)}

\bibitem[Loo75]{Loos-book}
Ottmar Loos, \emph{Jordan pairs}, Lecture Notes in Mathematics, Vol. 460,
  Springer-Verlag, Berlin-New York, 1975. \MR{0444721 (56 \#3071)}

\bibitem[Loo78]{Loos-homog}
Ottmar Loos, \emph{Homogeneous algebraic varieties defined by {J}ordan pairs},
  Monatsh. Math. \textbf{86} (1978), no.~2, 107--129. \MR{516835}

\bibitem[Loo79]{Loos-pairs}
Ottmar Loos, \emph{On algebraic groups defined by {J}ordan pairs}, Nagoya Math.
  J. \textbf{74} (1979), 23--66. \MR{535959}

\bibitem[Loo15]{Loos-divpairs}
Ottmar Loos, \emph{Division pairs: a new approach to {M}oufang sets}, Adv.
  Geom. \textbf{15} (2015), no.~2, 189--210. \MR{3334024}

\bibitem[McC66]{Mc1}
Kevin McCrimmon, \emph{A general theory of {J}ordan rings}, Proc. Nat. Acad.
  Sci. U.S.A. \textbf{56} (1966), 1072--1079. \MR{0202783 (34 \#2643)}

\bibitem[McC04]{Mc}
K.~McCrimmon, \emph{A taste of {J}ordan algebras}, Universitext,
  Springer-Verlag, New York, 2004. \MR{2014924 (2004i:17001)}

\bibitem[MPW15]{descent-book}
Bernhard M{\"u}hlherr, Holger~P. Petersson, Richard~M. Weiss, \emph{Descent in
  buildings}, Annals of Mathematics Studies, vol. 190, Princeton University
  Press, Princeton, NJ, 2015. \MR{3364836}

\bibitem[MW15]{descent-2}
Bernhard M{\"u}hlherr, Richard~M. Weiss, \emph{Galois involutions and
  exceptional buildings}, submitted (2015), 41 pp.

\bibitem[PS07]{PrStr-V}
Alexander Premet, Helmut Strade, \emph{Simple {L}ie algebras of small
  characteristic. {V}. {T}he non-{M}elikian case}, J. Algebra \textbf{314}
  (2007), no.~2, 664--692. \MR{2344582 (2008j:17039)}

\bibitem[PS08a]{PS-elem}
V.~A. Petrov, A.~K. Stavrova, \emph{Elementary subgroups in isotropic reductive
  groups}, Algebra i Analiz \textbf{20} (2008), no.~4, 160--188. \MR{2473747
  (2009j:20069)}

\bibitem[PS08b]{PrStr-VI}
Alexander Premet, Helmut Strade, \emph{Simple {L}ie algebras of small
  characteristic. {VI}. {C}ompletion of the classification}, J. Algebra
  \textbf{320} (2008), no.~9, 3559--3604. \MR{2455517 (2009j:17013)}

\bibitem[PS11]{PS-tind}
Victor Petrov, Anastasia Stavrova, \emph{The {T}its indices over semilocal
  rings}, Transform. Groups \textbf{16} (2011), no.~1, 193--217. \MR{2785501
  (2012c:20144)}

\bibitem[Sch85]{Sc}
R.~D. Schafer, \emph{On structurable algebras}, J. Algebra \textbf{92} (1985),
  no.~2, 400--412. \MR{778459 (86h:17001)}

\bibitem[Sel67]{Sel67}
G.~B. Seligman, \emph{Modular {L}ie algebras}, Ergebnisse der Mathematik und
  ihrer Grenzgebiete, Band 40, Springer-Verlag New York, Inc., New York, 1967.
  \MR{0245627 (39 \#6933)}

\bibitem[Sel76]{Seligman}
G.~B. Seligman, \emph{Rational methods in {L}ie algebras}, Marcel Dekker Inc.,
  New York, 1976, Lecture Notes in Pure and Applied Mathematics, Vol. 17.
  \MR{0427394 (55 \#428)}

\bibitem[Smi90]{S2}
O.~N. Smirnov, \emph{An example of a simple structurable algebra}, Algebra i
  Logika \textbf{29} (1990), no.~4, 491--499, 504. \MR{1155474 (93c:17002)}

\bibitem[Smi92]{S}
O.~N. Smirnov, \emph{Simple and semisimple structurable algebras}, Proceedings
  of the {I}nternational {C}onference on {A}lgebra, {P}art 2 ({N}ovosibirsk,
  1989) (Providence, RI), Contemp. Math., vol. 131, Amer. Math. Soc., 1992,
  pp.~685--694. \MR{1175866 (93f:17008)}

\bibitem[Spr73]{Springer}
Tonny~Albert Springer, \emph{Jordan algebras and algebraic groups},
  Springer-Verlag, New York-Heidelberg, 1973, Ergebnisse der Mathematik und
  ihrer Grenzgebiete, Band 75. \MR{0379618 (52 \#523)}

\bibitem[Sta09]{St-thes}
Anastasia~K. Stavrova, \emph{The structure of isotropic reductive groups},
  {P}h. {D.} thesis, St. Petersburg State University, 2009 (Russian).

\bibitem[Ste60]{Ste-aut}
Robert Steinberg, \emph{Automorphisms of finite linear groups}, Canad. J. Math.
  \textbf{12} (1960), 606--615. \MR{0121427 (22 \#12165)}

\bibitem[Str09]{Strade-bookII}
Helmut Strade, \emph{Simple {L}ie algebras over fields of positive
  characteristic. {II}}, de Gruyter Expositions in Mathematics, vol.~42, Walter
  de Gruyter \& Co., Berlin, 2009, Classifying the absolute toral rank two
  case. \MR{2573283 (2011c:17035)}

\bibitem[Tak74]{Tak}
Mitsuhiro Takeuchi, \emph{Tangent coalgebras and hyperalgebras. {I}}, Japan. J.
  Math. \textbf{42} (1974), 1--143. \MR{0389896 (52 \#10725)}

\bibitem[Tim01]{Timm}
Franz~Georg Timmesfeld, \emph{Abstract root subgroups and simple groups of
  {L}ie type}, Monographs in Mathematics, vol.~95, Birkh\"auser Verlag, Basel,
  2001. \MR{1852057 (2002f:20070)}

\bibitem[Tit64]{Tits64}
J.~Tits, \emph{Algebraic and abstract simple groups}, Ann. of Math. (2)
  \textbf{80} (1964), 313--329. \MR{0164968 (29 \#2259)}

\bibitem[Tit66]{Boulder}
J.~Tits, \emph{Classification of algebraic semisimple groups}, Algebraic
  {G}roups and {D}iscontinuous {S}ubgroups ({P}roc. {S}ympos. {P}ure {M}ath.,
  {B}oulder, {C}olo., 1965), Amer. Math. Soc., Providence, R.I., 1966, 1966,
  pp.~33--62. \MR{0224710 (37 \#309)}

\bibitem[Tit92]{Durham}
Jacques Tits, \emph{Twin buildings and groups of {K}ac-{M}oody type}, Groups,
  combinatorics \& geometry ({D}urham, 1990), London Math. Soc. Lecture Note
  Ser., vol. 165, Cambridge Univ. Press, Cambridge, 1992, pp.~249--286.
  \MR{1200265 (94d:20030)}

\bibitem[TW02]{TW}
J.~Tits, R.~M. Weiss, \emph{Moufang polygons}, Springer Monographs in
  Mathematics, Springer-Verlag, Berlin, 2002. \MR{1938841 (2003m:51008)}

\bibitem[Wil76]{Wil76}
Robert~Lee Wilson, \emph{A structural characterization of the simple {L}ie
  algebras of generalized {C}artan type over fields of prime characteristic},
  J. Algebra \textbf{40} (1976), no.~2, 418--465. \MR{0412239 (54 \#366)}

\bibitem[Wil09]{Wilson}
Robert~A. Wilson, \emph{The finite simple groups}, Graduate Texts in
  Mathematics, vol. 251, Springer-Verlag London, Ltd., London, 2009.
  \MR{2562037 (2011e:20018)}

\bibitem[Win84]{winter}
David~J. Winter, \emph{Outer derivations and classical-{A}lbert-{Z}assenhaus
  {L}ie algebras}, Canad. J. Math. \textbf{36} (1984), no.~6, 961--972.
  \MR{771921 (86k:17009)}

\bibitem[Zas39]{Zas}
Hans Zassenhaus, \emph{\"{U}ber {L}ie'sche {R}inge mit
  {P}rimzahlcharakteristik}, Abh. Math. Sem. Univ. Hamburg \textbf{13} (1939),
  no.~1, 1--100. \MR{3069699}

\end{thebibliography}

\end{document}